\documentclass[11pt, reqno]{amsart} 
\usepackage{amssymb,amscd,amsfonts,amsbsy}
\usepackage{latexsym}
\usepackage{exscale}
\usepackage{amsmath,amsthm,amsfonts}
\usepackage{mathrsfs}
\usepackage{xcolor} 
\usepackage[colorlinks=true,linkcolor=blue,citecolor=red,urlcolor=red, 
]{hyperref} 
\usepackage{esint} 
\usepackage{graphicx}
\usepackage{subcaption}
\usepackage{bm} 
\usepackage{enumitem}

\usepackage[utf8]{inputenc}

\usepackage{tikz}
\usetikzlibrary{arrows.meta, positioning} 
\tikzstyle{ultrathin}=[line width=.5pt]
\tikzstyle{thiner}=[line width=1pt]
\tikzstyle{thin}=[line width=1.5pt]
\tikzstyle{fat}=[line width=2pt]
\tikzstyle{heavier}=[line width=3pt]
\tikzstyle{ultrafat}=[line width=4pt]
\tikzstyle{ultrafat-pt}=[line width=7pt]
\definecolor{myred}{HTML}{E53935}
\definecolor{myblue}{HTML}{1E88E5}
\definecolor{mygreen}{HTML}{43A047}
\definecolor{myyellow}{HTML}{FDD835}
\definecolor{myorange}{HTML}{FB8C00}
\definecolor{mygold}{HTML}{F9A825}
\definecolor{mypurple}{HTML}{8E24AA}
\definecolor{mygray}{HTML}{BDBDBD}
\definecolor{mybrown}{HTML}{6D4C41}
\definecolor{mynavy}{HTML}{1A237E}
\definecolor{mypink}{HTML}{ffbfca}
\definecolor{myseagreen}{HTML}{26A69A}
\definecolor{myviolet}{HTML}{f07ef0}
\definecolor{mydarkblue}{HTML}{0D47A1}
\definecolor{mydarkcyan}{HTML}{E0FFFF}
\definecolor{darkgray}{rgb}{0.66, 0.66, 0.66}
\definecolor{mydarkgreen}{HTML}{1B5E20}
\definecolor{mydarkmagenta}{HTML}{AD1457}
\definecolor{mydarkorange}{HTML}{EF6C00}
\definecolor{lightblue}{rgb}{0.68, 0.85, 0.9}
\definecolor{lightcyan}{rgb}{0.88, 1.0, 1.0}
\definecolor{lightgray}{rgb}{0.83, 0.83, 0.83}
\definecolor{mylightgreen}{HTML}{81C784}
\definecolor{lightyellow}{rgb}{1.0, 1.0, 0.88}
\definecolor{myshadow}{rgb}{0.5, 0.5, 0.5}
\definecolor{pink}{rgb}{1.0, 0.75, 0.8}
\definecolor{violet}{rgb}{0.93, 0.51, 0.93}
\definecolor{myauxcolor}{RGB}{245, 255, 255}
\definecolor{mylightgreen}{RGB}{193, 225, 159}

\calclayout

\makeatletter
\renewcommand{\l@subsection}{\@tocline{2}{0pt}{2.5em}{3.2em}{\small}} 
\makeatother

\allowdisplaybreaks

\theoremstyle{plain}
\newtheorem{theorem}[equation]{Theorem}
\newtheorem{lemma}[equation]{Lemma}

\newtheorem{proposition}[equation]{Proposition}

\theoremstyle{definition}
\newtheorem{definition}[equation]{Definition}

\newtheorem{remark}[equation]{Remark}

\theoremstyle{remark}

\numberwithin{equation}{section} 

\newcommand{\mbf}[1]{\mathbf{#1}}

\renewcommand{\emptyset}{\mbox{\textup{\O}}}

\newcommand{\abs}[1]{\left\lvert #1 \right\rvert}

\newcommand{\norm}[1]{\left\lVert #1 \right\rVert}

\newcommand\restr[2]{\ensuremath{\left.#1\right|_{#2}}} 

\def\Xint#1{\mathchoice
{\XXint\displaystyle\textstyle{#1}}%
{\XXint\textstyle\scriptstyle{#1}}%
{\XXint\scriptstyle\scriptscriptstyle{#1}}%
{\XXint\scriptscriptstyle\scriptscriptstyle{#1}}%
\!\int}
\def\XXint#1#2#3{{\setbox0=\hbox{$#1{#2#3}{\int}$ }
	\vcenter{\hbox{$#2#3$ }}\kern-.585\wd0}}
\def\barint{\Xint-}
\newcommand{\bariint}{\barint\mkern-11.5mu\barint}

\def\all{{\operatorname{all}}}
\def\BMO{\operatorname{BMO}}

\DeclareMathOperator{\capac}{cap}
\def\diam{\operatorname{diam}}

\def\dist{\operatorname{dist}}

\def\good{{\operatorname{good}}}

\def\Lip{\operatorname{{Lip}}}

\def\parab{{\operatorname{par}}}
\def\supp{\operatorname{supp}}

\def\Stop{\operatorname{Stop}}
\def\Strip{\operatorname{Strip}}
\def\SStrip{\operatorname{\widetilde{Strip}}}

\def\Tree{\operatorname{Tree}}

\def\HD{\operatorname{HD}}
\def\LD{\operatorname{LD}}

\def\B{\mathbb{B}}
\def\C{\mbf{C}}
\def\D{\mathbb{D}}

\def\F{\mathcal{F}}
\def\G{\mathbb{G}}

\def\H{\mathcal{H}}
\def\N{\mathbb{N}}

\def\R{\mathbb{R}}
\def\Rn{\mathbb{R}^n}

\def\ree{{\R^{n+1}}}
\def\S{\mathbf{S}}
\def\T{\mathcal{T}}

\def\Z{\mathbb{Z}}
\def\X{\mathbf{X}}
\def\Y{\mathbf{Y}}

\newcommand{\pom}{{\partial \Omega}}

\newcommand{\om}{\Omega}

\begin{document}

\title[{\tiny Parabolic uniform rectifiability is characterized by Carleson measure estimates}]{Parabolic uniform rectifiability is characterized by Carleson measure estimates for bounded caloric functions}
\author[S. Bortz, P. Hidalgo-Palencia, S. Hofmann, J. Warta]
{Simon Bortz, Pablo Hidalgo-Palencia, Steve Hofmann, James Warta}



\address{Simon Bortz, Department of Mathematics, University of Alabama, Tuscaloosa, AL 35487, USA}
\email{sbortz@ua.edu}

\address{Pablo Hidalgo-Palencia, Departament de Matemàtiques i Informàtica, Universitat de Barcelona, Gran Via de les Corts Catalanes 585, 08007 Barcelona, Spain}
\email{pablo.hidalgo@ub.edu}

\address{Steve Hofmann, Department of Mathematics, University of Missouri, Columbia, MO 65211, USA}
\email{hofmanns@missouri.edu}

\address{James Warta, Department of Mathematics, University of Missouri, Columbia, MO 65211, USA}
\email{jw34r@missouri.edu}

\thanks{S.B. was supported by NSF grant no.~DMS-2555449, Simons Foundation grant MPS-TSM-00959861, and an AWI-CONSERVE Fellowship at the University of Alabama.
P.H.-P. has been supported by the grants CEX2019-000904-S-20-3 funded by MCIN/AEI/10.13039/501100011033, and PCI2024-155066-2 funded by MICIU/AEI/10.13039/501100011033/UE and co-funded by the European Union. He also acknowledges financial support from MCIN/AEI/10.13039/501100011033 grants CEX2019-000904-S and PID2019-107914GB-I00.
S.H. and J.W. were supported by NSF grant no.~DMS-2349846.}

\date{\today}


\begin{abstract}
	We show that the parabolic uniform rectifiability of a set can be characterized by an interior PDE property, namely Carleson measure estimates for bounded
	solutions to the heat equation. In particular, under very mild background
	hypotheses on an open set $\Omega$ and its boundary $\pom$,
	we show that the following are equivalent:
	\begin{itemize}
		\item[(i)] The quantity $|\nabla u(\cdot)|^2 \dist(\cdot, \partial\Omega)$ is the density of a Carleson measure on $\Omega$ for all bounded solutions to the heat equation in $\Omega$.
		\item[(ii)] The caloric measure for $\Omega$ admits a corona decomposition.
		\item[(iii)] $\pom$ is parabolic uniformly rectifiable.
	\end{itemize}
	The implication (iii) implies (i) is the main result of the cited paper \cite{BHHLN}.  The remaining implications are new, and are direct parabolic analogues of results of Garnett, Mourgoglou and Tolsa \cite{GMT} in the elliptic setting for the Laplacian; however, our proofs require several new ideas. In particular, we must adapt arguments to account both 
	for time-lag, and
	for the fact that some characterizations of (elliptic) uniform rectifiability have either not been developed or have been shown to be untrue in the parabolic setting. 
	
	Our background assumptions, as noted above, are very mild: we assume that the domain satisfies the time-symmetric capacity density condition and has interior corkscrews, and that the boundary is (non time-directed) Ahlfors-David regular. These are weaker than the background assumptions of Bortz, Hofmann, Martell and Nystr\"{o}m in \cite{BHMN}, and are essentially optimal for the boundary continuity properties of caloric functions on which our arguments rely. 
\end{abstract}

\maketitle

\tableofcontents


\section{Introduction}

The heat equation is the fundamental model for diffusion phenomena, but boundary value problems for it in rough domains behave in ways that can differ sharply from the elliptic case. In this paper we give the first characterization of parabolic uniform rectifiability (a purely geometric property of a space-time boundary) by an interior PDE estimate: a Carleson measure estimate for gradients of bounded solutions to the heat equation. Such characterizations linking PDE and geometry have a well-developed elliptic precedent, beginning with Dahlberg's theorem \cite{D} stating that the Laplace equation is $L^2$-solvable on Lipschitz domains, and culminating in \cite{AHMMT, HLMN}, where (elliptic) uniform rectifiability is shown to be the exact geometric condition for $L^p$-solvability of the Dirichlet problem. The parabolic theory, however, deviates from this template in an essential way, for reasons which have no elliptic counterpart.

Shortly after Dahlberg's result, Kaufman and Wu \cite{KW} showed that parabolic Lipschitz domains are not sufficient for $L^p$-solvability of the Dirichlet problem for the heat equation, for any $p < \infty$; in fact, caloric measure and surface measure can be mutually singular. The obstruction has no elliptic analogue, since it arises entirely from oscillations in the time variable of a parabolic Lipschitz graph. Identifying the correct extra regularity (namely, that the half-order time derivative of the parametrization of the boundary lies in BMO) took twenty years \cite{LM, HL}, and showing that this extra condition is also necessary for $L^p$-solvability has taken another thirty \cite{BHMN_graph, BHMN}. These works show that parabolic uniform rectifiability (Subsection~\ref{subsec:UR}), the quantitative condition capturing this additional temporal regularity, plays exactly the role that (elliptic) uniform rectifiability plays in the elliptic theory: it is the geometric condition under which boundary value problems for the heat equation satisfy the expected solvability properties. It is this condition that we characterize here.

In the elliptic setting, the analogous PDE characterization, i.e. that Carleson measure estimates for bounded harmonic functions are equivalent to uniform rectifiability, is due to Martell, Mayboroda and the third author \cite{HMM} and Garnett, Mourgoglou and Tolsa \cite{GMT}. Our result is the parabolic analogue of this equivalence, completing the same circle of ideas to understand the relationship between PDE and geometry well beyond graphical domains. As remarked above, however, the passage from elliptic to parabolic here is not routine, and requires several new ideas that we describe below.


\subsection{Main result}

In this paper, we characterize parabolic uniform rectifiability by a purely interior PDE condition. Together with the recent work of Martell, Nystr\"om and the first and third authors in \cite{BHMN}, where this geometric property was shown to arise naturally in the study of quantitative solvability of the Dirichlet problem, our results further highlight the central role of parabolic uniform rectifiability in the analysis of rough space-time domains.

Our PDE property is a Carleson measure estimate on the gradients of bounded solutions to the heat equation, in the same sense used in the elliptic theory of \cite{HMM, GMT}. A bounded caloric function can oscillate only by developing large gradients, and the estimate controls how much these oscillations can accumulate over scales and locations. The weight $\delta_\Omega$ makes the estimate meaningful despite the possible blow-up of $\abs{\nabla u}$ near $\pom$ and, along with the normalization by $r^{n+1}$ in the right hand side, makes the condition scale-invariant.
For the definitions of cylinders $\C_r$ and the distance to the boundary $\delta_\Omega$, see Subsection~\ref{sec:notation}. 

\begin{definition}[Carleson measure estimates, p-CME] \label{def:CME}
	Let $\Omega \subset \ree = \R^n \times \R$ be an open set in space-time. We say that \textit{bounded caloric functions in $\Omega$ satisfy Carleson measure estimates}, which we frequently abbreviate by ``p-CME holds in $\Omega$'', if there exists $C_{\mathrm{CME}} > 0$ such that
	for every $(y, s) \in \pom$ and $r > 0$, and for every bounded caloric function $u$ in $\Omega$ (i.e. satisfying $\partial_t u - \Delta u = 0$ in $\Omega$), it holds 
	\begin{equation*}
		\iiint_{\Omega \cap \C_r(y, s)} \abs{\nabla u(X, t)}^2 \delta_{\Omega}(X, t) \, dX \, dt 
		\leq
		C_{\mathrm{CME}} \norm{u}_{L^\infty(\Omega)}^2 r^{n+1}.
	\end{equation*}
\end{definition}

Carleson measure estimates have played a central role in harmonic analysis since the work of Fefferman on the $H^1$-BMO duality \cite{FS} and later became fundamental in the study of boundary value problems, where they provide a natural endpoint notion of solvability for the Dirichlet problem, see e.g. \cite{KKoPT, DKP, KKiPT, HMM, GMT, HLe, GH_BMO}.

The main result of the paper is the following, stating that parabolic uniform rectifiability can be characterized by the fact that solutions to the heat equation satisfy the Carleson measure estimates defined right above, under weaker background hypotheses than those of \cite{BHMN}. For precise definitions, we refer to Section~\ref{sec:preliminaries}. 

\begin{theorem}[p-CME implies p-UR] \label{th:main}
	Let $\Omega \subset \ree = \R^n \times \R$ be an open set in space-time. Assume that $\Omega$ satisfies the interior corkscrew condition, the time-symmetric capacity density condition (TSCDC), and that $\Sigma := \pom$ is parabolic Ahlfors-David regular (p-ADR). Then, bounded caloric functions in $\Omega$ satisfy Carleson measure estimates if, and only if, $\Sigma$ is parabolic uniformly rectifiable (p-UR). 
\end{theorem}

In this paper, we will only prove the direction that p-CME in $\Omega$ implies that $\Sigma$ is parabolic uniformly rectifiable, because the other implication had already been shown in a paper of Hoffman, Luna-García, Nystr\"{o}m and the first and third authors \cite[Theorem 1.3]{BHHLN}. We note that our result is of the nature of a free boundary problem: we are showing regularity of the boundary of a domain, given only information about solutions of an associated PDE. 

Theorem~\ref{th:main} is established in two steps, each of independent interest: first, we show that p-CME implies a corona decomposition for caloric measure (Theorem~\ref{th:corona}); second, we prove that such a corona decomposition implies parabolic uniform rectifiability (Theorem~\ref{th:main2}). Roughly speaking, this corona decomposition provides a quantitative multiscale description of caloric measure, asserting that caloric measure and surface measure (equivalently, the Poisson kernel in rough domains) are quantitatively comparable on a large collection of scales and locations, with the exceptional scales satisfying a Carleson packing condition (see Definition~\ref{def:corona}).

\begin{theorem}[Corona decomposition for caloric measure implies p-UR] \label{th:main2}
	Let $\Omega \subset \ree$ be an open set in space-time.	Assume that $\Omega$ satisfies the time-symmetric capacity density condition, and that $\Sigma := \pom$ is parabolic Ahlfors-David regular. If the caloric measure for $\Omega$ admits a corona decomposition as in Definition \ref{def:corona}, then $\Sigma$ is parabolic uniformly rectifiable.
\end{theorem}

The preceding results yield the following equivalence, which shows that parabolic uniform rectifiability admits three equivalent descriptions: a geometric one, a measure-theoretic one through caloric measure, and a purely interior PDE one through Carleson measure estimates.

\begin{theorem}[p-CME $\iff$ corona for caloric measure $\iff$ p-UR] \label{th:all}
	Let $\Omega \subset \ree = \R^n \times \R$ be an open set in space-time. Assume that $\Omega$ satisfies the interior corkscrew condition, the time-symmetric capacity density condition, and that $\Sigma := \pom$ is parabolic Ahlfors-David regular. Then, the following are equivalent:
	\begin{enumerate}
		\item p-CME holds in $\Omega$.
		\item The caloric measure admits a corona decomposition as in Definition \ref{def:corona}.
		\item $\Sigma$ is parabolic uniformly rectifiable.
	\end{enumerate}
\end{theorem}

As mentioned above, (3) $\implies$ (1) was shown in \cite[Theorem 1.3]{BHHLN}, and we prove (1) $\implies$ (2) and (2) $\implies$ (3) in Theorems~\ref{th:corona} and \ref{th:main2}. We emphasize that the additional background assumptions in Theorem~\ref{th:all} are only needed for the implications proved in the present paper, since (3) $\implies$ (1) holds under the sole assumption of p-ADR, as in the elliptic theory (see \cite[Appendix A]{AHMMT}).

More broadly, Theorem~\ref{th:all} provides several equivalent perspectives on parabolic uniform rectifiability. Analogous characterizations have played a central role in the development of the elliptic theory, and we expect them to provide a flexible framework for further investigations of caloric measure and boundary value problems in rough space-time domains.

Although Theorem~\ref{th:all} mirrors its elliptic counterpart, the proofs of the implications (1) $\implies$ (2) and (2) $\implies$ (3) require several new ideas reflecting the genuinely parabolic nature of the problem:
\begin{itemize}
	\item First, we construct caloric functions whose oscillations are localized at prescribed scales in Section~\ref{sec:GMT}. The construction is inspired by the elliptic argument of \cite{GMT}, but its extension to the parabolic setting requires new ideas. In particular, we introduce a geometric construction that simultaneously permits repeated applications of the time-directed interior Harnack inequality while approaching the rough boundary.
	
	\item  Second, we develop a stopping-time argument involving two simultaneous poles for caloric measure in Section~\ref{sec:corona}. Indeed, the construction in Section~\ref{sec:WHSA} requires poles to be sufficiently far into the future, and the argument in Section~\ref{sec:GMT} requires them to be very close to the boundary. Since no single pole can satisfy both requirements, we overcome this obstacle by running the stopping-time argument simultaneously with two poles.
	
	\item Third, we establish the p-CME property for the approximating subdomains constructed in Section~\ref{sec:WHSA}. This requires combining local regularity theory for parabolic Lipschitz graphs with a careful quantitative analysis of stopping-time distances, along with a delicate integration by parts argument adapted to rough geometries.
	
	\item A further important feature of our approach is that Theorem~\ref{th:all} is established under substantially weaker background assumptions than those in \cite{BHMN}: we assume only the TSCDC and parabolic (non time-directed) ADR, rather than the stronger time-symmetric ADR condition. The TSCDC is a quantitative analogue of the classical Wiener criterion for the heat equation \cite{EG_Wiener} and is essentially optimal among assumptions ensuring Hölder continuity of caloric functions up to the boundary, just as the corresponding quantitative Wiener condition is in the elliptic setting \cite{Aikawa,CHMPZ}. The remaining background assumption, i.e. the interior corkscrew condition, is already required in the elliptic analogue of our results (see \cite[Remark 5.17]{CHM}). 
\end{itemize} 


\subsection{Historical context} 

Assuming that $\Omega$ lies above a parabolic Lipschitz graph $\psi$, the condition that $D_t^{1/2} \psi \in \BMO_\parab$ (see Subsection~\ref{subsec:UR}) was found to be sufficient for the solvability of the heat equation with boundary data in $L^p$ by Lewis and Murray in \cite{LM}. This was inspired by earlier work of Murray in \cite{M}, and the sharp exponent $p=2$ was obtained shortly after by Lewis and the third author in \cite{HL} under an additional smallness assumption. It took 30 more years to see that this condition is also necessary for solvability of the $L^p$ Dirichlet problem for some $p > 1$, which is the recent work of Martell, Nystr\"{o}m and the first and third authors in \cite{BHMN_graph}. More recently,  in the non-graphical setting, the necessity of parabolic uniform rectifiability of the boundary for $L^p$ solvability of the heat equation has been obtained in \cite{BHMN}. The graphical result \cite{BHMN_graph} was also extended to some suitable variable coefficient operators in \cite{BFHH}. There are several other relevant works dealing with non-graphical settings and/or variable coefficient equations, like \cite{HL01, HL05, DPP, NS, BHHLN_22b, GH_RH, GH_BMO}.

In the elliptic setting, the characterization of uniform rectifiability by Carleson measure estimates was completed through the works of Martell, Mayboroda and the third author \cite{HMM} and Garnett, Mourgoglou and Tolsa \cite{GMT}. The present paper completes the analogous program for the heat equation: we establish the parabolic analogue of \cite{GMT}, namely the free-boundary implication p-CME implies p-UR, whereas the converse implication was obtained by Hoffman, Luna-García, Nystr\"{o}m and the first and third authors in \cite{BHHLN}.

Nevertheless, the proof is not a direct adaptation of \cite{GMT}. While our construction of the corona decomposition is inspired by that work (yet it requires several nontrivial adaptations), the remainder of the argument follows a fundamentally different philosophy. In \cite{GMT}, the corona decomposition for harmonic measure is converted into uniform rectifiability through the boundedness of suitable Riesz transforms and the deep theorem of Nazarov, Tolsa and Volberg \cite{NTV}. No analogous machinery is currently available in the parabolic setting, and recent work of Hoffman and Jaye \cite{HJ} suggests that such an approach would face significant additional obstacles. Moreover, the alternative elliptic proof of \cite{AGMT}, which avoids singular integrals, is likewise unavailable because it relies on characterizations of elliptic uniform rectifiability that are known to fail in the parabolic setting \cite[Observation 4.19]{BHHLN_22a}. 

Consequently, after constructing the corona decomposition, we pursue a completely different, more geometric strategy inspired by Le, Martell, Nystr\"om and the third author \cite{HLMN} (in turn based on work of Lewis and Vogel \cite{LV1, LV2}), relying on explicit approximating Lipschitz graphs obtained through Green functions. This approach was recently adapted to the parabolic setting by Martell, Nyström and the first and third authors \cite{BHMN}, and our work builds on those ideas.

Finally, we remark that the time-symmetric capacity density condition is not merely a technical assumption, but rather arises naturally in the theory of boundary regularity for parabolic equations. It is closely related to the Wiener criterion for the heat equation established in \cite{EG_Wiener} (building on \cite{L} and the original elliptic \cite{Wiener}), with important extensions to variable-coefficient operators in \cite{GL,FGL}. Quantitative versions of these capacitary conditions have recently played a central role in several problems involving rough domains \cite{MP,HHK}.


\subsection{Outline of the proof and the paper}

The proof of Theorem~\ref{th:all} is divided into four main steps. See also Section~\ref{sec:proof} and Figure~\ref{fig:roadmap}. 

\begin{figure}[h]
	\centering
	\resizebox{\textwidth}{!}{%
		\begin{tikzpicture}[
			box/.style={draw, rounded corners, align=center, minimum width=2.8cm, minimum height=1.3cm,
				font=\small, text width=2.5cm},
			side/.style={draw, rounded corners, align=center, minimum width=2.9cm, minimum height=1.3cm,
				font=\small, fill=gray!8, text width=2.6cm},
			arr/.style={-Latex, thick},
			node distance=0.7cm
			]
			\node[box] (h) {p-CME holds in $\Omega$ \\ \scriptsize (hypothesis)};
			\node[box, right=of h] (c) {Corona decomp.\ for $\omega$ \\ \scriptsize (Section \ref{sec:corona})};
			\node[box, right=of c] (g) {Approx. by \ Lip$(1,1/2)$ subdomains \\ \scriptsize (Section \ref{sec:WHSA})};
			\node[box, right=of g] (p) {p-CME on subdomains \\ \scriptsize (Section \ref{sec:pushing_CME})};
			\node[box, right=of p] (u) {$\Sigma$ is p-UR \\ \scriptsize (Section \ref{sec:proof})};
			
			\node[side, above=0.7cm of c] (osc) {Oscillating solutions \\ \scriptsize (Section \ref{sec:GMT})};
			\node[side, above=0.7cm of u] (reg) {Graph regularity \\ \scriptsize \cite{HW, BHMN_graph}};
			
			\draw[arr] (h) -- (c);
			\draw[arr] (c) -- (g);
			\draw[arr] (g) -- (p);
			\draw[arr] (p) -- (u);
			\draw[arr] (osc) -- (c.north);
			\draw[arr] (reg) -- (u.north);
	\end{tikzpicture}}
	\caption{Proof roadmap for Theorem~\ref{th:main}.}
	\label{fig:roadmap}
\end{figure}
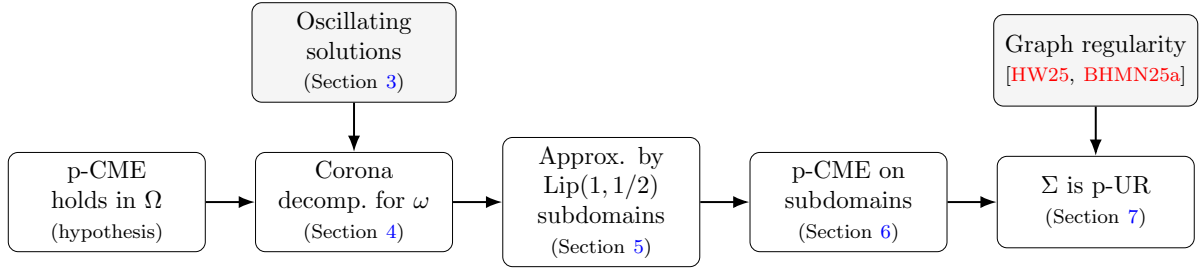

We first show that the p-CME assumption yields a corona decomposition for caloric measure (Theorem~\ref{th:corona}). The main difficulty is to control the low-density cubes, i.e. those at which caloric measure $\omega$ is much smaller than surface measure $\sigma$. Following the strategy of Garnett, Mourgoglou and Tolsa \cite{GMT}, we construct caloric functions oscillating at prescribed scales, but the construction requires substantial new ideas in the parabolic setting (Proposition~\ref{prop:GMT}). A further difficulty is that later on in the paper, we require poles to lie far into the future, whereas the construction of oscillating solutions requires poles close to the boundary. This is overcome by introducing the two-pole stopping-time argument described in Section~\ref{sec:corona}.

After having established this corona decomposition for $\omega$ with respect to $\sigma$, we adapt the geometric approximation machinery developed in \cite{BHMN} to construct very flat Lip(1,1/2) graphs which approximate $\Sigma$ from the interior of $\Omega$ (that is, some parabolic version of the so-called WHSA property, but with the additional property of the graphs being very flat). This is the content of Proposition~\ref{prop:WHSA}. 

The main remaining obstacle is to prove that these approximating graphs are \textbf{regular} Lip(1,1/2). Unlike in \cite{BHMN}, we cannot rely on the $A_\infty$ property of caloric measure. Instead, we establish p-CME on the approximating subdomains in Section~\ref{sec:pushing_CME}, using a new argument based on the corona decomposition, quantitative control of stopping-time distances, and an integration by parts argument.

Once p-CME has been established in the approximating subdomains, the recent characterization of the third and fourth authors \cite{HW} in graphical domains yields the $A_\infty$ condition. We may then invoke the graphical characterization of \cite{BHMN_graph} to conclude that the approximating graphs are regular, completing the proof of Theorem~\ref{th:all}. The graphical nature of the approximating domains is therefore essential in the final step.

Many constants will appear along the proofs. We summarize their dependencies in Remark~\ref{rem:dependencies}.



\section{Notation and preliminaries} \label{sec:preliminaries}

\subsection{Notation} \label{sec:notation}

\begin{itemize}
	
	\item We shall use \textbf{lower case letters} to denote (the spatial component of) points over boundaries (like $\partial\Omega$), and \textbf{capital letters} for generic points in $\R^{n+1}$ (in particular those in open sets like $\Omega$). Time coordinates will always be written in lower case.
	
	\item We also use \textbf{boldface letters} to denote points and sets in space-time $\R^{n+1} = \Rn \times \R$ to alleviate notation when the argument does not really require to distinguish the spatial and time coordinates of the point of interest, as in 
	\[\qquad\quad  \mbf{X}=(X,t), \, \Y = (Y, s) \in \Rn \times \R,
	\quad\mathrm{and}\quad
	\mbf{x}=(x,t), \, \mbf{y} = (y, s) \in \pom \subset \Rn \times \R.\]
	Similarly, we also denote the space-time origin by $\mbf{0} = (0, 0) \in \Rn \times \R$. 
	
	\item In Section~\ref{sec:pushing_CME}, there will be a further decomposition of the spatial variables because the first one will be distinguished (graphs will be parametrized in this direction), and we will follow the above conventions, but thinking on the coordinates 
	\begin{equation*}
		(x_0, x, t) \in \R \times \R^{n-1} \times \R = \R^{n+1}.
	\end{equation*}
	
	\item We denote solid integrals, taken over subsets of space-time $\R^{n+1}$, by $\iiint$. Integrals taken over a codimension 1 object, like a boundary of an open set or the last coordinates in $\R^{n-1} \times \R$ in Section~\ref{sec:pushing_CME}, will be denoted by $\iint$. Lastly, we reserve the notation $\int$ for integrals taken over only one of these variables. The symbols $\fint, \bariint$ denote averages.
	
	\item Given a time $t \in \R$, we denote the restriction to the past/future/present by 
	\[
	\qquad 
	\T_{< t} := \Rn \times (-\infty, t), 
	\qquad 
	\T_{= t} := \Rn \times \{t\}, 
	\qquad 
	\T_{> t} := \Rn \times (t, +\infty).
	\]
	
	\item Given $A\subset\R^{n+1}$, we denote initial and terminal times by 
	\[
	\qquad\;\;
	T_{\min}(A)
	:= \inf \big\{t \in \R: A\cap \T_{=t} \neq\emptyset\big\}, 
	\quad
	T_{\max}(A):=\sup\big\{t \in \R: A\cap \T_{=t} \neq\emptyset \big\},
	\]
	which may take infinite values.
	
	\item (Parabolic distance)
	Given $(X,t)\in\R^{n}\times\R$ in space-time, we say its \textit{parabolic norm} is
	\[\lVert (X, t) \rVert : = \max \big\{ |X|, |t|^{1/2} \big\}, \]
	although other equivalent choices like $|X| + |t|^{1/2}$ would lead to the very same conclusions.
	The distance induced by this norm will be called \textit{parabolic distance}. Further, given $A\subset\R^{n+1}$, its diameter with respect to the parabolic distance is
	\[\diam_{\text{par}}(A):=\sup_{(\mbf{X},\mbf{Y})\in A\times A}\lVert\mbf{X}-\mbf{Y}\rVert.\]
	
	\item (Space-time cubes) 
	Given $(X,t)\in\R^{n}\times\R$ in space-time and $r>0$, we define the (space-time) \textit{parabolic cube} centered at $(X, t)$ with radius $r$ by
	\[\C_r(X,t):=\big\{(Y,s)\in\R^{n}\times\R:\lVert(X,t)-(Y,s)\rVert<r\big\}.\]
	Further, we define the \textit{time-backward} and \textit{time-forward} parabolic cubes by
	\[\C_r^-(X,t):= \C_r(X, t) \cap \T_{< t}, \qquad 
	\C_r^+(X,t):= \C_r(X, t) \cap \T_{> t}.\]
	
	In Section~\ref{sec:pushing_CME}, we will work in $\R^n = \R^{n-1} \times \R$, where the first spatial coordinate of $\R^{n+1}$ is deleted. In this case, we denote cylinders by 
	\[
	C'_r(x, t) := \big\{(y,s)\in\R^{n-1}\times\R:\lVert(x,t)-(y,s)\rVert<r\big\},
	\]
	where $(x, t) \in \R^{n-1} \times \R$ and $r > 0$.
	
	\item (Purely spatial cubes) 
	We also use the notation of cubes for purely spatial regions around $X \in \Rn$ (note the lack of boldface fonts): 
	\[
	C_r(X) := \big\{ Y\in\R^n: | X-Y | <r \big\}.
	\]
	
	\item  If $A \subset \R^{n+1}$, we denote the (parabolic) distance to its boundary by
	\begin{equation*}
		\delta_A(\mbf{X}) 
		:=
		\inf \big\{ r > 0 : \overline{\C_r(\mbf{X})} \cap \partial A \neq \emptyset \},
		\qquad \mbf{X} \in \R^{n+1}.
	\end{equation*}
	In the special case that we consider our distinguished open set $\Omega$, we abbreviate 
	\begin{equation*}
		\delta(\mbf{X}) := \delta_\Omega(\mbf{X}) \approx \dist(\mbf{X}, \Sigma), 
		\qquad \mbf{X} \in \R^{n+1}.
	\end{equation*}
	
	\item We denote by $\Gamma$ the heat kernel (the fundamental solution of the heat equation), namely
	\begin{equation} \label{eq:heat_kernel}
		\Gamma(X, t; Y, s)
		=
		(4 \pi (t-s))^{-n/2} e^{-\frac{\abs{X-Y}^2}{4(t-s)}} \mathbf{1}_{t > s},
		\qquad (X, t), (Y, s) \in \R^{n+1},
	\end{equation}	
	which concretely satisfies (see e.g. \cite[(A.8)]{GH_RH}, it is an easy real variable exercise)
	\begin{equation} \label{eq:bound_gamma}
		\Gamma(X, t; Y, s)
		\lesssim 
		\norm{(X, t) - (Y, s)}^{-n}.
	\end{equation} 
	
	\item The symbol $\nabla$ denotes the gradient only in the spatial variables, and $\Delta$ is the Laplacian in the spatial variables only, too, i.e. $\Delta = \sum_{j=1}^n \partial_{x_j x_j}$.
\end{itemize}


\subsection{Geometric framework} \label{subsec:geometry}

The two underlying geometric assumptions along the whole paper, and concretely in Theorem~\ref{th:main}, are that $\Sigma$ is p-ADR and $\Omega$ satisfies the TSCDC. These assumptions are classical in the area, and ensure that $\Sigma$ has ``codimension 1'' in a quantitative way, and that the complement of $\Omega$ is thick enough in order for caloric functions to behave well near the boundary. However, in the parabolic setting, the definitions are very delicate because of the special role of time, so let us introduce all the necessary concepts.

\begin{definition}[Parabolic Hausdorff measure and	dimension] \label{def:Hausdorff}
	Given $\alpha \geq 0$, we define a \textit{parabolic Hausdorff measure of {homogeneous}
	dimension $\alpha$}, denoted $\H_{\text{par}}^\alpha$, by first setting, for $\delta>0$ and $E\subset \R^{n+1}$,
	\begin{equation*}
		\H_{\text{par},\delta}^\alpha(E)
		:= 
		\inf 
		\left\{ \sum_{k \in \N} (\diam_{\text{par}}(E_k))^\alpha : E \subset \bigcup_{k \in \N} E_k, \; \diam_{\text{par}}(E_k)\leq \delta \text{ for all } k \in \N \right\},
	\end{equation*}
	where diameters are measured with respect to the parabolic distance\footnote{
		That is, this parabolic Hausdorff measure is defined in the same way that one defines standard (Euclidean) Hausdorff measure, but instead using the {parabolic} distance.
	}, and then we define
	\begin{equation*}
		\H_{\text{par}}^\alpha (E) 
		:= 
		\lim_{\delta\to 0^+} \H_{\text{par},\delta}^\alpha(E)\,.
	\end{equation*}
	As is the case of classical Hausdorff measure, $ \H_{\text{par}}^\alpha$ is a Borel regular measure.
	We refer the reader to \cite[Chapter 2]{EG} for a discussion of the basic properties of standard Hausdorff measure, which readily adapt to the parabolic counterpart $\H_{\text{par}}^\alpha$.
	In particular, one obtains a measure equivalent to $\H_{\text{par}}^\alpha$ if one defines $\H_{\text{par},\delta}^\alpha$ in terms of coverings by parabolic cylinders, rather than
	arbitrary sets of parabolic diameter at most $\delta$.  As in the classical setting, we define the \textit{parabolic homogeneous dimension} of a set $A\subset \R^{n+1}$ by
	\[\dim_{\H_{\text{par}}}(A):= \inf \left\{ 0\leq \alpha<\infty: \,\H_{\text{par}}^\alpha(A)=0\right\}.
	\]
	We observe that $\dim_{\H_{\text{par}}}(\ree)=n+2$.
\end{definition}

\begin{definition}[Surface measure]
	Given a closed set $\Sigma \subset \R^{n+1}$, we define a \textit{surface measure} on $\Sigma$ as the restriction of $\H_{\text{par}}^{n+1}$ to $\Sigma$, i.e.,
	\begin{equation*} 
		\sigma = \sigma_\Sigma:=  \H_{\text{par}}^{n+1}|_\Sigma\,.
	\end{equation*}
\end{definition}

\begin{definition}[Ahlfors-David regular, p-ADR] \label{def:ADR}
	We say that the closed set $\Sigma \subset \R^{n+1}$ is (parabolic) \textit{Ahlfors-David regular} (p-ADR for short) if there exists $C_{\mathrm{ADR}} \geq 1$ such that
	\begin{equation*}
		C_{\mathrm{ADR}}^{-1} r^{n+1} \leq \sigma(\C_r(\mbf{x}) \cap \Sigma) \leq C_{\mathrm{ADR}} r^{n+1}, 
		\qquad 
		\forall \, \mbf{x} \in \Sigma, \, 0 < r < \diam(\Sigma).
	\end{equation*}
\end{definition}

\begin{remark}
	Under the ADR assumption, there are other equivalent ways to define surface measure. For instance, one that has been frequently used in the literature consists in defining a product-like measure by measuring time with $\H^1$ (Euclidean Hausdorff measure), and using $\H^{n-1}$ (again Euclidean Hausdorff measure) in each time slice. For more details and a rigorous proof of the equivalence, we refer to Remark 2.8 and Appendix B in \cite{BHHLN_corona}.
\end{remark}

The ADR property makes $\Sigma$ a space of homogeneous type, whence classical constructions in harmonic analysis give a very handy decomposition of $\Sigma$ on nested cubes, that resemble the standard dyadic cubes in the Euclidean space. The following construction follows those from \cite{Christ, DS1, DS2, HytKai}.

\begin{lemma}[Dyadic cubes] \label{lem:cubes}  
	Assume that $\Sigma  \subset \mathbb R^{n+1}$ is (parabolic) ADR, with constant $C_{\mathrm{ADR}}$. Then $\Sigma$ admits a parabolic dyadic decomposition in the sense that there exists $C_{\mathrm{cub}} > 0$ depending only on $n$ and $C_{\mathrm{ADR}}$, such that the following holds. There exists a grid of cubes 
	\begin{equation*}
		\mathbb{D}= \bigcup_{k \in \Z} \D_k,
	\end{equation*}
	where each generation $\D_k$, $k \in \Z$, is composed of (dyadic) cubes
	\begin{equation*}
		\mathbb{D}_k:=\big\{ Q_{j}^k\subset \Sigma: j\in \mathfrak{I}_k \big\},
	\end{equation*}
	 ($\mathfrak{I}_k$ is just the countable set of indices) such that the following hold:
	\begin{list}{$(\theenumi)$}{\usecounter{enumi}\leftmargin=1cm \labelwidth=1cm \itemsep=0.2cm \topsep=.2cm \renewcommand{\theenumi}{\roman{enumi}}}
		\item (each generation covers $\Sigma$) $\Sigma=\bigcup_{j \in \mathfrak{I}_k} Q_{j}^k\,$ for each $k\in{\mathbb Z}$,
		
		\item (cubes are nested) if $m \geq k$, then either $Q_{i}^{m}\subset Q_{j}^{k}$ or $Q_{i}^{m}\cap Q_{j}^{k}=\emptyset$,
		
		\item (ancestors are unique) for each $m > k \in \Z$ and $i \in \mathfrak{I}_m$, there is a unique $j \in \mathfrak{I}_k$ such that $Q_{i}^m\subset Q_{j}^k$,
		
		\item (size within a generation is $\approx 2^{-k}$) for each $k \in \Z$ and $j \in \mathfrak{I}_k$, it holds  \begin{equation*}
			\diam\big(Q_{j}^k\big)\leq C_{\mathrm{cub}} 2^{-k},
		\end{equation*}
		and there exists $\mbf{x}^k_{j} \in \Sigma$, which we call \textit{center of} $Q_k^j$, such that 
		\begin{equation*}
			\C_{C_{\mathrm{cub}}^{-1} 2^{-k}} (\mbf{x}^k_j) \cap \Sigma \subset Q_j^k.
		\end{equation*}
		In fact, we define 
		\begin{equation*}
			\ell(Q_k^j) := 2^{-k}, 
		\end{equation*}
		so that it holds $C_{\mathrm{cub}}^{-1} \ell(Q_k^j) \leq \diam(Q_k^j) \leq C_{\mathrm{cub}} \ell(Q_k^j)$.
	\end{list}
\end{lemma}

We will also need that our domain is open in some quantitative sense, as expressed by the following classical condition. 

\begin{definition}[Interior corkscrews] \label{def:corkscrew}
	Given an open set $\Omega \subset \R^{n+1}$, we say that $\Omega$ satisfies the \textit{interior corkscrew condition} if there exists $C_{\mathrm{CKS}} \geq 1$ such that if $\mbf{x} \in \Sigma := \pom$ and $r > 0$, there exists $\mbf{Y} \in \Omega$ such that 
	\begin{equation*}
		\C_{r / C_{\mathrm{CKS}}} (\mbf{Y}) \subset \Omega \cap \C_r (\mbf{x}).
	\end{equation*}
\end{definition}


\subsection{Parabolic potential theory}

On the other hand, we want the exterior of $\Omega$ to be sufficiently thick in order for caloric functions to decay appropriately when approaching the boundary. Instead of relying on the strong time directed versions of ADR (TBADR, TFADR, TSADR, see e.g. \cite{GH_RH}), we will use much weaker time directed capacitary assumptions for this (whose power has been shown recently e.g. in \cite{MP} or \cite{HHK}), which we define next.

\begin{definition}[(Thermal) capacity] \label{def:capacity}
	Given a Borel measure $\mu$ on $\R^{n+1}$, we define its \textit{heat potential} by
	\begin{equation*}
		\Gamma \mu(\mbf{X}) 
		:=
		\iiint_{\R^{n+1}} \Gamma(\mbf{X}; \mbf{Y}) \, d\mu(\mbf{Y}),
	\end{equation*}
	where $\Gamma$ is the heat kernel (recall \eqref{eq:heat_kernel}). Then, given a compact set $K \subset \R^{n+1}$, we define its thermal capacity to be
	\begin{equation*}
		\capac(K)
		:=
		\sup \big\{ \mu(K) : \; \Gamma \mu \leq 1 \text{ in } \R^{n+1}, \; \mu \geq 0, \; \textrm{spt}(\mu) \subset K \big\}.
	\end{equation*}
\end{definition}

With this, we can define our main assumption regarding the thickness of the complement of our domain of interest. It will be necessary to have time-directed versions of it.

\begin{definition}[Capacity density conditions: TBCDC, TFCDC and TSCDC] \label{def:TBCDC}
	Let $\Omega \subset \R^{n+1}$ be an open set. We say that $\Omega$ satisfies the \textit{time-backward capacity density condition} (TBCDC for short) if there exist $a\in(0,1)$ and $c_{\mathrm{CDC}}>0$ such that
	\begin{equation*} 
		\frac{ \capac \left( \big( \overline{C_r(x_0)} \times [t_0-r^2,t_0-(ar)^2] \big) \, \cap \, \Omega^c\right)} {\capac\left(\overline{C_r(x_0)}\times [t_0-r^2,t_0-(ar)^2] \right)} 
		\geq 
		c_{\mathrm{CDC}}, 
		\qquad 
		\forall \, (x_0, t_0) \in \Sigma, \, r > 0.
	\end{equation*}
	Similarly, we say that $\Omega$ satisfies the \textit{time-forward capacity density condition} (TFCDC) if
	\begin{equation*} 
		\frac{ \capac \left( \big( \overline{C_r(x_0)} \times [t_0+(ar)^2,t_0+r^2] \big) \, \cap \, \Omega^c\right)} {\capac\left(\overline{C_r(x_0)}\times [t_0+(ar)^2,t_0+r^2] \right)} 
		\geq 
		c_{\mathrm{CDC}},
		\qquad 
		\forall \, (x_0, t_0) \in \Sigma, \, r > 0.
	\end{equation*}
	We say that $\om$ satisfies the \textit{time-symmetric capacity density condition} (TSCDC) if it satisfies both the TBCDC and TFCDC conditions.	
\end{definition}

We will sometimes alleviate notation by denoting $C_{\mathrm{CDC}} := \max \{ a^{-1}, c_{\mathrm{CDC}}^{-1} \}$ to remember dependencies on the CDC constants.

\begin{remark} \label{rem:capac_cylinders}
	As shown in \cite[Lemma 3.11]{HHK}, the denominators right above are actually simple because the capacity of cylinders satisfies an elementary estimate: if $0 < a < b \leq 1$,
	\begin{equation*}
		\capac \big(\overline{C_r(x)} \times [t-(br)^2, t-(ar)^2] \big) \approx r^n, 
		\qquad 
		\forall \, (x, t) \in \R^{n+1},
	\end{equation*}
	where the implicit constant depends only on $n, a, b$.
\end{remark}

\begin{remark} \label{rem:boundaries}
	Under the TSCDC assumption, it is easy to check (see \cite[Section 3.3]{HHK}) that the quasi-lateral boundary $\Sigma$ coincides with the topological boundary $\pom$, and also with the essential boundary $\partial_e \Omega$ because the abnormal boundary $\partial_a \Omega$ is empty (for the definitions of the different parts of a parabolic boundary, check e.g. the exposition in \cite[Section 2.2]{HHK} and the references therein). Thus, along the text, we frequently simply identify $\Sigma = \pom$ without further comments.
\end{remark}

We will need the following elementary lemma later (it extends \cite[Lemma 2.6]{BHMN}).

\begin{lemma} \label{lem:2.6BHMN}
	Let $\Omega \subset \R^{n+1}$ be an open set satisfying the TSCDC. Let $\mbf{Y} = (Y, s) \in \Omega$, and consider $\mbf{x} = (x, t) \in \Sigma$ a closest point, satisfying $\norm{\mbf{Y} - \mbf{x}} = \delta(\mbf{Y})$. Then it holds $\mbf{x} \in \partial C_{\delta(\mbf{Y})}(Y) \times [s-\delta(\mbf{Y})^2, s+\delta(\mbf{Y})^2]$.
\end{lemma}
\begin{proof}
	Abbreviate $d := \delta(\mbf{Y})$. Since it clearly holds $\mbf{x}\in \partial \C_d(\mbf{Y})$, a space-time cube, we just need to check that it belongs to its lateral face, and not its back or front faces (the extrema in time). Indeed, $\partial \C_d(\mbf{Y}) = \Gamma_{\mathrm{lat}} \cup \Gamma_{\mathrm{back}} \cup \Gamma_{\mathrm{front}}$, where 
	\begin{equation*}
		\Gamma_{\mathrm{lat}} 
		:=
		\partial C_d(Y) \times [s-d^2, s+d^2], 
		\quad 
		\Gamma_{\mathrm{back}} 
		:=
		C_d(Y) \times \{ s-d^2\}, 
		\quad 
		\Gamma_{\mathrm{front}} 
		:=
		C_d(Y) \times \{ s+d^2\}.
	\end{equation*}
	Let us show that  $\mbf{x}$ cannot belong to  $\Gamma_{\mathrm{back}}$. Indeed, if it happened that $\mbf{x} = (x, t) \in \Gamma_{\mathrm{back}}$, then there would exist $\varepsilon = \varepsilon(\mbf{x}) > 0$ very small (smaller than the distance from $x$ to $\partial C_d(Y)$) such that $\C^+_\varepsilon (\mbf{x}) \subset \C_d(\mbf{Y}) \subset \Omega$ (the last inclusion is by definition of $d$), which is a contradiction with the TFCDC applied at $\mbf{x} \in \Sigma$ with scale $\varepsilon$. Analogously, one can rule out that $\mbf{x}$ belongs to $\Gamma_{\mathrm{front}}$ using the TBCDC, so it must hold $\mbf{x} \in \Gamma_{\mathrm{lat}}$, which is what we wanted.
\end{proof}


\subsection{Whitney regions} \label{subsec:whitney}

In the following we let $\mathcal{W} := \mathcal{W}(\Omega) := \{\mbf{I}\}_{\mbf{I} \in \mathcal{W}}$ denote a collection of (closed) dyadic (parabolic) Whitney cubes of $\Omega$, which
form a covering of $\Omega$ with non-overlapping interiors, such that
\begin{equation*}
	4\, {\rm{diam}}(\mbf{I})\leq \dist(4 \mbf{I},\pom) \leq  \dist(\mbf{I},\pom) \leq 40 \, {\rm{diam}}(\mbf{I}),
	\qquad \mbf{I} \in \mathcal{W},
\end{equation*}
and also, with an implicit constant only depending on $n$, it holds
\begin{equation*}
	\diam(\mbf{I}_1)\approx\diam(\mbf{I}_2), \mbox{ whenever $\mbf{I}_1, \mbf{I}_2 \in \mathcal{W}$ meet.}
\end{equation*}

We next associate Whitney regions to the dyadic cubes on the boundary. Assuming that $\Sigma=\pom$ is parabolic ADR, and given $Q\in \mathbb{D}$, we introduce, for a given parameter $K_0\gg 1$ (its value will be fixed in Section~\ref{sec:WHSA} and later used in Section~\ref{sec:pushing_CME}),
\begin{equation*}
	\mathcal{W}_Q
	:=
	\mathcal{W}_Q(K_0)
	:= 
	\big\{\mbf{I}\in \mathcal{W}:\,K_0^{-1} \ell(Q)\leq \ell(\mbf{I})
	\leq K_0\,\ell(Q),\, {\rm and}\, \dist(\mbf{I},Q)\leq K_0\, \ell(Q)\big\}.
\end{equation*}
This allows us to introduce the \textit{Whitney region associated to $Q \in \D$} by
\begin{equation*}
	\mbf{U}_Q:= \bigcup_{I\in \mathcal{W}_Q} (1+\vartheta)I,
\end{equation*}
where $\vartheta > 0$ is a small dilation parameter so that the enlarged Whitney regions $\{(1+10\vartheta) \mbf{I}\}_{\mbf{I} \in \mathcal{W}}$ still retain the properties of Whitney decompositions.
We observe that the Whitney region $\mbf{U}_Q$ may have more than one connected component (we will denote them by $\{\mbf{U}_Q^i\}_i$), but that the number of distinct components is uniformly bounded, depending only upon $n$, $K_0$ and $\vartheta$. Note that by construction
\begin{equation}\label{eq:whitney_bounded}
	\dist(\mbf{Y}, Q) \le C(n)K_0 \, \ell(Q), 
	\quad 
	\forall \, \mbf{Y} \in \mbf{U}_Q.
\end{equation}

Finally, let us define some enlarged Whitney regions. In the following $\varepsilon>0$ is a small but fixed parameter (its value will also be fixed in Section~\ref{sec:WHSA} and later used in Section~\ref{sec:pushing_CME}). Then, given $Q\in\mathbb{D}$,
we write
\begin{equation}\label{epschain}
	\mbf{X} \sim_{\varepsilon,Q} \, \mbf{Y}
\end{equation}
if $\mbf{X}$ may be connected to $\mbf{Y}$ by a chain of
at most $\varepsilon^{-1}$ parabolic ``Whitney cylinders", i.e., cylinders of the form
$\C_k:=\C_{\delta(\mbf{Y}_k)/2}(\mbf{Y}_k)$, with
$1\leq k \leq N\leq \varepsilon^{-1}$, such that
$\C_k \cap \C_{k+1} \neq \emptyset$ when $1\leq k\leq N-1$, and
\begin{equation*}
	\varepsilon^3\ell(Q)
	\leq
	\delta(\mbf{Y}_k)
	\leq 
	\varepsilon^{-3}\ell(Q), 
	\quad 
	\text{for $1 \leq k \leq N$.}
\end{equation*}

With this in mind, we define the \textit{enlarged Whitney region associated to $Q \in \D$} by 
\begin{equation*}
	\widetilde{\mbf{U}}_Q
	:=
	\bigcup_i \widetilde{\mbf{U}}^i_Q,
\end{equation*}
where the union is taken over all the connected components $\mbf{U}_Q^i$ of $\mbf{U}_Q$, and each enlarged region $\widetilde{\mbf{U}}^i_Q$ is defined as 
\begin{equation*}
	\widetilde{\mbf{U}}^i_Q
	:=
	\widetilde{\mbf{U}}^{i}_Q(\varepsilon)
	:= 
	\left\{\mbf{X} \in\Omega:\, \mbf{X} \sim_{\varepsilon,Q} \mbf{Y} \  \text{ for some } \mbf{Y} \in \mbf{U}^{i}_Q \right\}.
\end{equation*}
We note here that, although the connected components $\mbf{U}_Q^i$ of $\mbf{U}_Q$ do not intersect, their enlarged versions $\widetilde{\mbf{U}}_Q^i$ may intersect. Nevertheless, this is harmless for our arguments.

It is important to also note that the ``chain regions'' $\widetilde{\mbf{U}}_Q$, in the parabolic setting, should \textit{not} be thought of as Harnack chain regions, as opposed to their elliptic counterparts. Indeed, in the definition of the relation \eqref{epschain}, the time directedness of the chain of points $\mbf{Y}_k$ is not taken into account, whereas true parabolic Harnack chains (allowing for successive application of Harnack's inequality as in Lemma~\ref{lem:harnack}) are necessarily time directed. Accordingly, we will indeed never compare values of caloric functions across arbitrary points of the $\widetilde{\mbf{U}}_Q$ regions.


\subsection{Parabolic uniform rectifiability} \label{subsec:UR}

Let $\Sigma \subset \mathbb R^{n+1}$ be a closed set, and assume that $\Sigma$ is {parabolic Ahlfors-David regular} with constant $M\geq 1$ (see Definition~\ref{def:ADR}). Then, the parabolic version of the ``beta"-numbers of P. Jones are:
\begin{equation*}
	\beta(\mbf{z}, r  )
	:=
	\beta_\Sigma( \mbf{z}, r ) 
	:=
	\inf_{P}  
	\biggl ( \, \bariint_{  \C_r ( \mbf{z}) \cap \Sigma } \, \biggl(\frac {\dist ( \mbf{y}, P )}{r}\biggr)^2  d \sigma (\mbf{y} )\biggr )^{1/2},
	\qquad \mbf{z} \in \Sigma, \; 0 < r < \diam(\Sigma),
\end{equation*}
where the infimum ranges over the set of all $ n $-dimensional
hyperplanes $ P $ containing a line
parallel to the $ t $ axis (that is, $t$-independent planes).

We also introduce the following measure on $(\mbf{z}, r) \in \Sigma \times (0,\diam(\Sigma))$:
\begin{equation*}
	d \nu( \mbf{z}, r  ) 
	:=
	d \nu_\Sigma ( \mbf{z}, r)
	:=
	\big(\beta_\Sigma  ( \mbf{z}, r)\big)^2 \, d \sigma ( \mbf{z}) \,\frac{dr}{r}
\end{equation*}

\begin{definition}[Parabolic uniform rectifiability, p-UR]\label{def:UR}
	Assume that $\Sigma  \subset \mathbb R^{n+1}$ is parabolic ADR with constant $C_{\mathrm{ADR}}$. Let $\nu=\nu_\Sigma$ be defined as above. We say that \textit{$\Sigma$ is parabolic uniformly rectifiable} with constants $(C_{\mathrm{ADR}},C_{\mathrm{UR}})$ if
	\begin{equation*}
		\sup_{\substack{\mbf{x} \in\Sigma \\ 0<\rho<\diam(\Sigma)}}  \rho^{ -(n+1) } \, \nu \big( (\C_\rho(\mbf{x}) \cap \Sigma ) \times ( 0, \rho) \big) \leq
		C_{\mathrm{UR}}.
	\end{equation*}
\end{definition}

In the case that $\Omega$ is the open set lying above the graph of a function $\psi$, it is well-known that $\pom$ is \textit{elliptic} uniformly rectifiable whenever $\psi$ is a Lipschitz function. The greatest difference and difficulties when working in the parabolic setting arise from the fact that $\psi$ being (parabolic) Lipschitz does not suffice for $\pom$ to be \textit{parabolic} uniformly rectifiable: some extra regularity property is needed, which we explain next. In view of the result of \cite{BHMN_graph}, this extra regularity property is actually necessary in order for parabolic uniform rectifiability to be a relevant notion for applications to PDE.

\begin{definition}[Lip(1,1/2) and regular Lip(1,1/2) functions] \label{def:rlip}
	A function $\psi:\mathbb R^{n-1}\times\mathbb R\to \mathbb R$ is called Lip(1,1/2) with constant $b_1$, if it is Lipschitz with respect to the parabolic distance:
	\begin{align}\label{1.1}
		|\psi(x,t)-\psi(y,s)|\leq b_1 \big(|x-y|+|t-s|^{1/2}\big),
		\qquad 
		\forall \, (x,t)\in\mathbb R^{n}, (y,s)\in\mathbb R^{n}.
	\end{align}
	Furthermore, let us denote by $D_{1/2}^t \psi  (x, t) $
	the half derivative in $ t $ of $ \psi ( x, \cdot )$ with  $x $ fixed, defined by the Fourier multiplier $|\tau|^{1/2}$ or 
	\begin{align*} 
		D_{1/2}^t  \psi (x, t):= \widehat c \; \mathrm{p.v.} \int_{ \mathbb R }
		\, \frac{ \psi ( x, s ) - \psi ( x, t ) }{ | s - t |^{3/2} } \, d s,
	\end{align*}
	for $ \widehat c$  properly chosen\footnote{Both formulations give rise to the same definition of regular Lip(1,1/2) function, as can be seen in \cite{HL}. Here one should view the half order time derivative as a distribution acting on the set of smooth functions with compact support with zero mean. The Lip(1,1/2) property is enough to ensure this distribution is well defined.}. We say that $\psi$ is a {regular} Lip(1,1/2)  function with parameters $b_1$ and $b_2$, if $\psi$ satisfies \eqref{1.1} and if
	\begin{align*} 
		D_{1/2}^t\psi\in \text{BMO}_{\text{par}}(\mathbb R^n), \; \text{and actually} \; \|D_{1/2}^t\psi\|_{\text{BMO}_{\text{par}}} \leq b_2<\infty,
	\end{align*}
	where $ \| \cdot \|_{\text{BMO}_{\text{par}}} $ denotes the semi-norm in parabolic $\text{BMO}(\mathbb R^{n})$ (that is, taking the standard definition of BMO, but replacing cubes by parabolic cylinders).
\end{definition}

Indeed, it turns out that if $\Omega$ is the open set lying above the graph of a function $\psi$, namely $\Omega = \big\{ (x_0, x, t) \in \R \times \R^{n-1} \times \R : x_0 > \psi(x, t) \big\}$, then $\pom$ is \textit{parabolic} uniformly rectifiable if and only if $\psi$ is regular Lip(1,1/2) (see \cite[Remark 2.14]{BFHH}). Thus, by the result of \cite{BHMN_graph}, regular Lip(1,1/2) graphs are the model of parabolic boundaries in which PDE are well-behaved. The extra regularity condition that $D_{1/2}^t\psi\in \text{BMO}_{\text{par}}$ is what makes the parabolic setting substantially more challenging than the elliptic counterpart.

The goal in our main result (Theorem~\ref{th:main}) is to show that $\Sigma=\pom$ is parabolic uniformly rectifiable. It will be very convenient to use the following criterion to prove p-UR, which basically reduces the proof to finding regular Lip(1,1/2) graphs which approximate $\Sigma$ at most scales, making use of the dyadic decomposition from Lemma~\ref{lem:cubes}. It is very important that in Theorem~\ref{th:corona_implies_UR} the constant $\eta$ need not be small, which provides flexibility for the proofs.

\begin{definition}[Corona decomposition by regular Lip(1,1/2) graphs] \label{def:unilateralcorona}
	Suppose that the closed set $\Sigma \subset \ree$ is parabolic ADR with constant $C_{\mathrm{ADR}}$. We say that \textit{$\Sigma$ has a Corona decomposition by regular $\Lip(1,1/2)$ graphs} if there exist constants
	$\eta > 0$ and  $K > 1$, and a disjoint decomposition
	$\mathbb D = \G\cup\B$, such that the following hold:
	\begin{list}{$(\theenumi)$}{\usecounter{enumi}\leftmargin=1cm \labelwidth=1cm \itemsep=0.2cm \topsep=.2cm \renewcommand{\theenumi}{\roman{enumi}}}
		
		\item  The collection $\G$ can be subdivided into a family of disjoint stopping time re\-gimes, $$\G = \bigsqcup\limits_{\S \in \mathcal{S}} \S,$$ such that each such tree $\S$ is semi-coherent, that is, each $\S$ has a maximal cube $Q(\S)$ (with respect to containment) and whenever $Q_1, Q_2, Q_3 \in \D$ satisfy $Q_1 \subset Q_2 \subset Q_3$ and $Q_1, Q_3 \in \S$, then it also holds $Q_2 \in \S$.
		
		\item The maximal cubes $\{Q({\S})\}$ and the cubes in $\B$ satisfy a Carleson
		packing condition,
		\begin{equation*}
			\sum_{\substack{\S \in \mathcal{S} \\ Q({\S})\subset R}} \sigma\big(Q({\S})\big)
			+
			\sum_{\substack{Q\in\B \\ Q\subset R}} \sigma(Q)
			\leq\, 
			C(n,C_{\mathrm{ADR}},\eta,K)\, \sigma(R),
			\quad \forall R\in \D.
		\end{equation*}
		
		\item For each $\S \in \mathcal{S}$, there exists a coordinate system, and a
		regular  Lip(1,1/2)  function $ \psi_{\S} = \psi_\S ( x, t ) : \mathbb R^{n-1}\times\mathbb R\to \mathbb R$, with
		parameters $b_1\leq\eta $ and $b_2$, with $b_2$ depending at most on $n$, $C_{\mathrm{ADR}}$, $\eta$ and $K$, such that if we define
		$\Gamma_{\S}:=\{(\psi_{\S}(x,t),x,t): (x,t)\in \mathbb R^{n-1}\times\mathbb R\}$ (in the coordinate system associated to $\S$), then
		\begin{equation*}
			\sup_{\mbf{x} \in KQ} \dist(\mbf{x},\Gamma_{{\S}} )
			\leq
			\eta\,\diam(Q),
			\qquad 
			\forall \, Q \in \S.
		\end{equation*}
	\end{list}
\end{definition}

The following criterion is the result of \cite[Theorems 1.2 and 4.16]{BHHLN_22a} (see also \cite[Theorem 4.15]{BHMN}).

\begin{theorem}[Corona by regular graphs implies p-UR]\label{th:corona_implies_UR}
	Suppose that $\Sigma \subset \ree$ is parabolic ADR with constant $C_{\mathrm{ADR}}$. Assume that $\Sigma$ has, for some $\eta>0$ and $K>1$, an  $(\eta,K)$-Corona decomposition by regular $\Lip(1,1/2)$ graphs with constants $b_1\leq\eta$ and $b_2$ (see Definition~\ref{def:rlip}), in the sense of Definition \ref{def:unilateralcorona}.
	Then $\Sigma$ is parabolic uniformly rectifiable with constants $(C_{\mathrm{ADR}},C_{\mathrm{UR}})$ for some $C_{\mathrm{UR}}=C_{\mathrm{UR}}(n,C_{\mathrm{ADR}},\eta,K,b_1,b_2)$.
\end{theorem}


\subsection{PDE framework} Let us recall some elementary properties of solutions of the heat equation (and its adjoint equation), and some finer ones making it explicit that the TBCDC is a sufficient condition in order for solutions to behave nicely close to the boundary of our domain. We start with some definitions.

\begin{definition}[Caloric and subcaloric functions]
	Let $\Omega \subset \R^{n+1}$ be an open set. We say that $u = u(X, t)$ is \textit{caloric} in $\Omega$ if 
	\begin{equation*}
		Hu(X, t)
		:=
		\partial_t u(X, t) - \Delta u(X, t)
		=
		0
		\qquad 
		\forall \, (X, t) \in \Omega.
	\end{equation*}
	Similarly, we say that $u = u(X, t)$ is \textit{adjoint caloric} in $\Omega$ if 
	\begin{equation*}
		H^*u(X, t)
		:=
		\partial_t u(X, t) + \Delta u(X, t)
		=
		0
		\qquad 
		\forall \, (X, t) \in \Omega.
	\end{equation*}
	Analogously, $u$ is subcaloric in $\Omega$ if $H u \leq 0$ in $\Omega$, and adjoint subcaloric if $H^* u \leq 0$ in $\Omega$. 
\end{definition}

\begin{definition}[Caloric measure]
	Let $\Omega \subset \R^{n+1}$ be an open set satisfying the TBCDC. Whenever $f \in C_c(\pom)$, there is a solution to the continuous Dirichlet problem 
	\begin{equation*}
		\begin{cases}
			(\partial_t - \Delta) u = 0 \quad & \text{in } \Omega, \\ 
			u = f & \text{on } \pom,
		\end{cases}
	\end{equation*}
	given by the Perron-Wiener-Brelot procedure (see \cite[Chapter 8]{Watson}). This construction actually verifies that, if we fix $\X \in \Omega$, the mapping $C_c(\pom) \ni f \mapsto u(\X) \in \R$ is positive and bounded (maximum principle), and also linear. Thus, the Riesz representation theorem yields the existence of a probability measure $\omega^\X$, supported on $\pom$, such that 
	\begin{equation*}
		u(\mbf{X}) = \iint_\pom f(\mbf{y}) \, d\omega^\X(\mbf{y}), 
		\qquad 
		\forall \, \mbf{X} \in \Omega, \, f \in C_c(\pom).
	\end{equation*}
	Similarly, if $\Omega$ satisfies the TFCDC, there is an associated adjoint caloric measure, satisfying the same properties above, but for solutions of $(\partial_t + \Delta) u = 0$.
\end{definition}

\begin{definition}[Green's function] \label{def:green}
	Let $\Omega \subset \R^{n+1}$ be an open set satisfying the TBCDC. We define the Green's function with pole at $\mbf{Y} \in \Omega$ by 
	\begin{equation*}
		G(\mbf{X}; \mbf{Y})
		:=
		\Gamma(\mbf{X}; \mbf{Y}) - \iint_\pom \Gamma(\mbf{z}, \mbf{Y}) \, d \omega^{\mbf{X}}(\mbf{z}),
		\qquad 
		\mbf{X} \in \Omega,
	\end{equation*}
	where we recall that $\Gamma$ is the heat kernel and $\omega^{\mbf{X}}$ the caloric measure with pole at $\mbf{X}$. Similarly, if $\Omega$ satisfies the TFCDC, we define the adjoint Green's function with pole at $\mbf{X} \in \Omega$ by 
	\begin{equation*}
		\widehat{G}(\mbf{X}; \mbf{Y})
		:=
		\Gamma(\mbf{X}; \mbf{Y}) - \iint_\pom \Gamma(\mbf{X}, \mbf{z}) \, d \widehat{\omega}^{\mbf{Y}}(\mbf{z}),
		\qquad 
		\mbf{Y} \in \Omega,
	\end{equation*}
	where $\widehat{\omega}$ denotes the adjoint caloric measure.
\end{definition}

\begin{remark}
	The TBCDC is actually not needed for the resolubility of the heat equation (same for TFCDC for the adjoint equation), and hence for the definitions of caloric measure and Green's functions, but we have preferred this exposition to avoid technicalities dealing with the essential boundary (see Remark~\ref{rem:boundaries}). In any case, let us mention that the TBCDC assumption allows for solvability of the continuous Dirichlet problem (and hence existence of the associated parabolic measure) for parabolic operators with bounded measurable coefficients, well beyond the heat equation (see \cite[Theorem 1.9]{HHK}), as a consequence of the parabolic version of the Wiener's criterion from \cite{EG_Wiener}.
\end{remark}

The following well-known interior estimates for caloric functions can be found, for instance, in \cite[Theorem H]{A} and \cite[Section 0]{FGS}.

\begin{lemma}[{Harnack's inequality}] \label{lem:harnack}
	Let $u$ be caloric and non-negative in $D \times (T_{\min}, T_{\max} )$ for a bounded Lipschitz domain $D \subset \Rn$. Let $D' \subset D$ be a convex subdomain of $D$. Then, for every $X, Y \in D'$, $T_{\min} < s < t \leq T_{\max}$ we have 
	\begin{equation*}
		u(Y, s)
		\leq 
		u(X, t) \exp \left[ C(n) \left( \frac{\abs{X-Y}^2}{t-s} + \frac{t-s}{\min\{1, s - T_{\min}, \dist(D', \partial D)^2\}} + 1 \right) \right].
	\end{equation*}
\end{lemma}

\begin{lemma}[{Caccioppoli's estimate}] \label{lem:caccioppoli}
	Let $u$ be non-negative and subcaloric in $\C_{2r}(\mbf{X}_0)$. Then 
	\begin{equation*}
		\iiint_{\C_r(\mbf{X}_0)} \abs{\nabla u(\mbf{X})}^2 d\mbf{X}
		\leq 
		\frac{C(n)}{r^2} \iiint_{\C_{2r}(\mbf{X}_0)} u(\mbf{X})^2 d\mbf{X}.
	\end{equation*}
\end{lemma}

The next results, about boundary behavior of caloric functions under the TBCDC, can be found in \cite[Lemmas 3.15 and 3.20]{MP} for the heat equation (see also \cite[Theorem 1.2]{HHK} for general parabolic operators and \cite[Lemmas 2.2 and 2.5]{GH_RH} under the TBADR).

\begin{lemma}[{Bourgain's estimate}] \label{lem:bourgain}
	If $\Omega \subset \R^{n+1}$ is an open set which satisfies the TBCDC, there exists $c_{\mathrm{B}} = c_{\mathrm{B}}(n, C_{\mathrm{CDC}}) > 0$ so that for every $\mbf{x}_0 \in \Sigma$ and $r > 0$, it holds 
	\begin{equation*}
		\omega^{\mbf{X}}(\mbf{C}_r(\mbf{x}_0) \cap \pom) \geq c_{\mathrm{B}}, 
		\qquad 
		\forall \, \mbf{X} \in C_{c_{\mathrm{B}} r}(\mbf{x}_0) \cap \Omega.
	\end{equation*}
\end{lemma}

\begin{lemma}[{H\"{o}lder continuity at the boundary}] \label{lem:holder}
	Assume that $\Omega \subset \R^{n+1}$ is an open set which satisfies the TBCDC. Then there exist some $\gamma > 0$ and $C > 0$, depending only on $n$ and $C_{\mathrm{CDC}}$, such that the following holds. Let $\mbf{x}_0 \in \Sigma$ and $r > 0$ such that $u$ is caloric in $\Omega \cap \C_{2r}(\mbf{x}_0)$ associated to some continuous (and compactly supported) boundary datum which vanishes on $\mbf{C}_{2r}(\mbf{x}_0) \cap \Sigma$. Then,
	\begin{equation*}
		u(\mbf{X}) 
		\leq 
		C \left( \frac{\delta(\mbf{X})}{r} \right)^{\gamma} \sup_{\mbf{C}_{2r}(\mbf{x}_0) \cap \Omega} u, 
		\qquad 
		\forall \, \X \in \mbf{C}_r(\mbf{x}_0) \cap \Omega.
	\end{equation*}
\end{lemma}

Actually, the result in \cite[Lemma 2.5]{GH_RH} is not exactly the one stated above. Nevertheless, it is easy to note that the integral in their RHS is controlled by our supremum. 

\begin{remark}
	All the results above are true also for adjoint solutions of the heat equation, just reversing appropriately the time variable, and using TFCDC instead of TBCDC.
\end{remark}

\begin{remark} \label{rem:estimates_BHMN}
	The fact that Bourgain's estimate holds (Lemma~\ref{lem:bourgain}) has a series of consequences, on top of Lemma~\ref{lem:holder}:
	\begin{itemize}
		\item Some of the finer estimates of Caffarelli-Fabes-Mortola-Salsa-type (concretely, the upper bounds for the Green's function in terms of some associated caloric measure) as in \cite[Lemma 5.6, Remark 5.7]{BHMN} also hold under the TBCDC assumption, and the same goes for the Caccioppoli inequality at the boundary from \cite[Lemma 5.10]{BHMN}. The reader can readily check that their proofs only use Lemma~\ref{lem:bourgain}, aside from elementary estimates. 
		
		\item Moreover, as explained in \cite[Remark 5.19]{BHMN}, the boundary H\"{o}lder continuity from Lemma~\ref{lem:holder} and the Riesz formula from \cite[Lemma 5.12]{BHMN} (which is true in just any open set, with no regularity needed) together imply some estimate as in Lemma~\ref{lem:holder}, but with averages in the right hand side, at least in the case when $u$ is an adjoint Green function (so we need to instead assume the TFCDC in order for the machinery to apply). This means that, under the TFCDC assumption, \cite[Corollary 5.17]{BHMN} is true.
		
		\item Lastly, \cite[Lemma 5.20]{BHMN} also holds under the TFCDC assumption, if the boundary is ADR, as can be readily verified from inspecting its proof (it uses the previous two points in this remark, applied to an adjoint Green's function).
	\end{itemize}
\end{remark}



\section{Construction of solutions oscillating at given scales} \label{sec:GMT}

To establish the packing condition for the low density cubes in the proof of Theorem~\ref{th:corona}, we construct caloric functions whose oscillations are localized at prescribed scales. While the construction is inspired by \cite[Lemma~3.3]{GMT}, its extension to the parabolic setting requires substantial new ideas. The main difficulty is that comparing the solutions at different points requires repeated applications of interior Harnack inequalities while simultaneously approaching the rough boundary at the desired scales. In the parabolic setting, the time-directed nature of Harnack's inequality makes these competing requirements significantly more delicate. Our weaker geometric assumptions also require several additional refinements of the construction.

\begin{proposition}[Existence of caloric functions oscillating at prescribed scales] \label{prop:GMT}
	Let $\Omega \subset \ree$ be an open set satisfying the TBCDC (see Definition~\ref{def:TBCDC}), so that $\Sigma := \pom$ is p-ADR (see Definition~\ref{def:ADR}). Then, there exists $\tau = \tau(n, C_{\mathrm{ADR}}, C_{\mathrm{CDC}}) > 0$ and $\varepsilon' = \varepsilon'(\tau) > 0$ small enough, such that the following holds (the implicit constants will be independent of $\tau$ and $\varepsilon'$):
	
	Let $Q \in \D$ (from Lemma~\ref{lem:cubes}) with $\ell(Q) \leq 1$, and assume that there is $\mbf{P}_Q \in \C_{c_{\mathrm{B}} \ell(Q)}(\mbf{x}_Q) \cap \Omega$ (recall Lemma~\ref{lem:bourgain}) such that
	\begin{equation}\label{eq:hyp_epsilon_1}
		\delta(\mbf{P}_Q) \approx \varepsilon' \ell(Q),
	\end{equation}
	and $E_Q \subset Q$ a Borel set such that 
	\begin{equation}\label{eq:hyp_epsilon_2}
		\omega^{\mbf{P}_Q} (E_Q) \geq (1-\varepsilon')\, \omega^{\mbf{P}_Q}(Q).
	\end{equation}
	Then there exists a non-negative measurable function $f_Q$ satisfying 
	\begin{equation*} 
		0 \leq f_Q \lesssim \mathbf{1}_{E_Q} 
	\end{equation*}
	such that if we denote $u_Q(\mbf{X}) := \iint_\pom f_Q(\mbf{y}) d\omega^{\mbf{X}}(\mbf{y})$ for $\X \in \Omega$, then $u_Q$ fulfills the bound 
	\begin{equation*}
		\iiint_{V_Q} \abs{\nabla u_Q}^2 \delta
		\geq 
		c(\tau, \varepsilon') \, \sigma(Q)
		,
	\end{equation*} 
	for $c(\tau, \varepsilon') \approx \tau^{n+5} (\varepsilon')^{n+1}$,
	where $\mbf{V}_Q \subset \Omega$ is some Borel set satisfying 
	\begin{equation} \label{eq:loc_V_Q}
		\mbf{V}_Q 
		\; \subset \; 
		\C_{\ell(Q)}(\mbf{x}_Q) \cap \big\{\mbf{X} \in \Omega : \delta(\mbf{X}) \geq \tau \delta(\mbf{P}_Q)\big\}.
	\end{equation}
\end{proposition}
\begin{proof}
	Write $\mbf{P}_Q = (P_Q, s_Q)$ and $\mbf{X}_Q = (X_Q, t_Q)$.	
	Denote $r := \delta(P_Q, s_Q)$, and note that $r \approx \varepsilon' \ell(Q) \ll \ell(Q)$ by \eqref{eq:hyp_epsilon_1}. We divide the proof into several steps for it to be easier to understand. Along the proof, it will be important to keep track of dependencies of constants on $\tau$ and $\varepsilon'$: otherwise we will leave them as implicit constants depending on $n, C_{\mathrm{ADR}}, C_{\mathrm{CDC}}$.
	
	\textbf{Step 1: geometrical construction.} In order to later apply Harnack inequalities, we need to develop a complicated construction to ensure that we stay in the interior at all times, and additionally move monotonically in time. The reader is advised to have Figure~\ref{fig:construction_initial} in mind during the whole proof.
		
	\textbf{Substep 1.1: some relevant points.}
	First, we construct a parabola within the domain that connects $(P_Q, s_Q)$ with a point $(y', s')$ which is on the boundary and, at the same time, sufficiently far backwards in time from $(P_Q, s_Q)$. This will be done in two cases, divided by the relationship between $(P_Q, s_Q)$ and any of its closest points at $\Sigma$, namely 
	\begin{equation*}
	\text{$(\hat{y},\hat{s}) \in \Sigma$ for which $\delta(P_Q, s_Q)=\norm{(P_Q, s_Q) - (\hat{y}, \hat{s})}$.}
	\end{equation*} 
	\begin{itemize}
		\item First, in the case that $\hat{s} < s_Q - (r/8)^2$, the point $(\hat{y},\hat{s})$ is sufficiently far backwards in time and we set  $(y', s') := (\hat{y}, \hat{s})$.
		
		\item Otherwise, we have that $\hat{s} \geq s_Q - (r/8)^2$ meaning that $\hat{s} - s_Q+(r/4)^2 > 0$, so it is possible to find $A > 0$ such that $\hat{s}-(aA r)^2=s_Q-(r/4)^2$, where $a > 0$ is the structural constant from the TBCDC condition (thus, $A \approx a^{-1}$, and we may treat it as a harmless constant). By such condition, applied at $(\hat{y}, \hat{s})$ with scale $A r$, we can find a point $(Z,\tau) \in \overline{C_{Ar}(\hat{y})} \times [\hat{s}-(A r)^2, \hat{s}-(aA r)^2] \cap \Omega^c$, i.e. $\tau \leq \hat{s}-(aA r)^2 = s_Q - (r/4)^2$. Additionally, the point $(P_Q, s_Q') := (P_Q, s_Q - (r/3)^2)$ resides in $\om$ (for its distance to $(P_Q, s_Q)$ is smaller than $r = \delta(P_Q, s_Q)$). Now, connecting the points $(P_Q, s_Q') \in \Omega$ and $(Z, \tau) \notin \Omega$ with a line segment, we necessarily find a point $(y', s') \in \pom$, and since both $s_Q'$ and $\tau$ are $\leq s_Q - (r/4)^2$, it must also be $s' \leq s_Q - (r/4)^2$. 
	\end{itemize}
	In either case the chosen point $(y', s')$ is sufficiently time-backward of $(P_Q, s_Q)$, namely
	\begin{equation} \label{def:y'}
		(y', s') \in \Sigma, \quad \norm{(P_Q, s_Q) - (y', s')} \lesssim r, \quad \text{and} \quad s' < s_Q - \left(\frac{r}{8}\right)^2.
	\end{equation}	
	
	Up to translation and rotation in only the spatial variables, we view
	\begin{equation*}
		(y', s') = (0, 0), 
		\qquad 
		(P_Q, s_Q) = (p_Q, 0, \ldots, 0, s_Q).
	\end{equation*}
	
	\begin{figure}[t!]
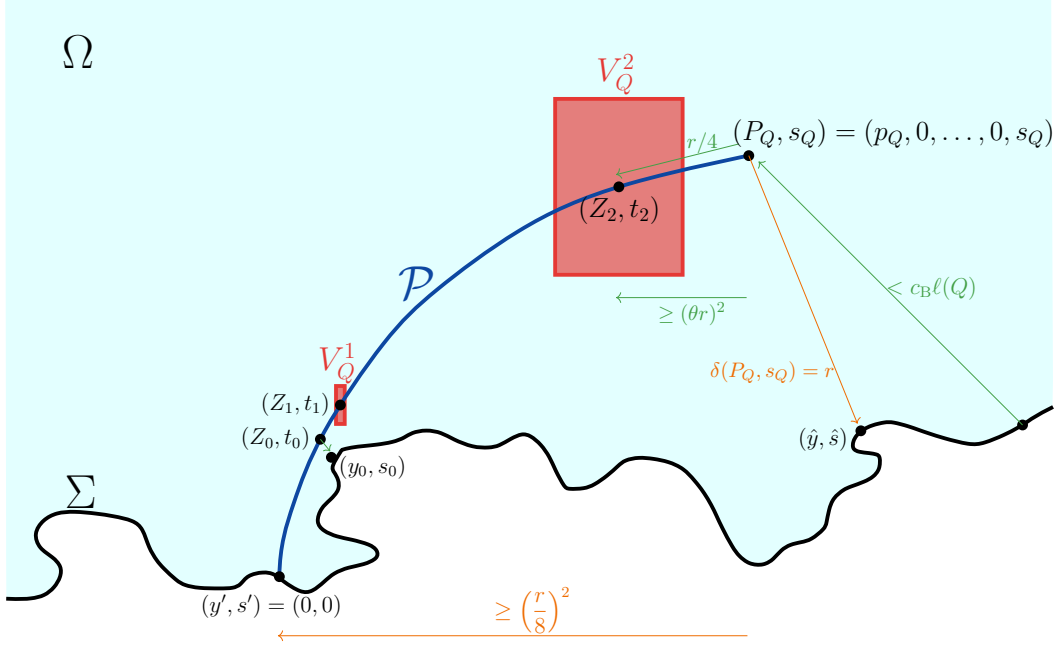

		\centering 
		\resizebox{0.9\textwidth}{!}{

		}
		\caption{Diagram of the construction. Remember that TBCDC enforces that every point in $\Sigma$ has a further point in $\Sigma$ in the near past, but we do not show this situation exactly in the picture because we are only showing 2 of the $n+1$ dimensions of the setting. Note that the TBCDC assumption does not rule out the case that $\Sigma$ may have lots of holes (the depicted situation is much better), hence the need for an argument to find $(y', s')$ as in \eqref{def:y'}.}
		\label{fig:construction_initial}
	\end{figure}
	
	Now consider the arc of parabola $\mathcal{P}$ (1D curve) with vertex at $(y', s') = (0, 0)$ passing through $(P_Q, s_Q)$ (which finishes there), defined by 
	\begin{equation*}
		\mathcal{P}
		:= 
		\Big\{ \bm{\gamma}(s) := (\gamma(s), 0, \ldots, 0, s) : 0 \leq s \leq s_Q \Big\}, 
		\quad 
		\text{where } 
		\gamma(s) 
		:= 
		\sqrt{\frac{p_Q^2}{s_Q}} \sqrt{s}
		=: 
		\alpha^{-1/2} \sqrt{s},
	\end{equation*}
	and it will be important for the forthcoming arguments that the ``aperture'' $\alpha$ of the parabola satisfies the following non-degeneracy estimate, which follows from \eqref{def:y'}, 
	\begin{equation} \label{eq:aperture}
		\alpha 
		=
		\frac{s_Q}{p_Q^2}
		=
		\frac{s_Q-s'}{\abs{P_Q - y'}^2}
		\gtrsim 
		\frac{(r/8)^2}{r^2} 
		\gtrsim
		1.
	\end{equation}

	Starting at $(P_Q, s_Q)$ and moving towards the past (in the pictures, to the left) along $\mathcal{P}$ towards $(y', s') = (0, 0)$, we find 
	\begin{equation*}
		\text{$(Z_0, t_0) \in \mathcal{P}$ the first point for which $\delta(Z_0, t_0) = 4\tau r$. }
	\end{equation*}
	Concretely, there must exist some $(y_0, s_0) \in \Sigma \cap \partial \C_{4\tau r}(Z_0, t_0)$.
	We also denote: 
	\begin{equation*}
		\text{$(Z_1, t_1)$ the point lying on $\mathcal{P}$ for which $t_1 := t_0 + (8\tau r)^2$,}
	\end{equation*}
	\begin{equation*}
		\text{$(Z_2, t_2) := \partial \C_{\frac14 r}(P_Q, s_Q) \cap \mathcal{P} \cap \{ t_2 \leq s_Q \}$.}
	\end{equation*}
	
	\textbf{Substep 1.2: some properties of the construction, and definition of $\mbf{V}_Q^1$ and $\mbf{V}_Q^2$.} 
	First, since $\delta(Z_0, t_0) = 4\tau r$, $(y', s') \in \Sigma$ and $(Z_0, t_0) \in \mathcal{P}$, we have
	\begin{equation*}
		4\tau r 
		\leq 
		\norm{(Z_0, t_0) - (y', s')}
		=
		\abs{Z_0} + t_0^{1/2}
		=
		\abs{\gamma(t_0)} + t_0^{1/2}
		=
		(1+\alpha^{-1/2}) t_0^{1/2},
	\end{equation*}
	which, along with \eqref{eq:aperture}, yields
	\begin{equation*} 
		t_0 \gtrsim (\tau r)^2.
	\end{equation*}
	This implies, remembering again \eqref{eq:aperture}, that for some uniform constant $M \geq 10$,
	\begin{equation*}
		\frac{d\gamma(t)}{dt}
		= 
		\alpha^{-1/2} \frac{1}{2\sqrt{t}}
		\leq 
		\frac{M}{\tau r} 
		\qquad 
		\text{whenever } t \geq t_0.
	\end{equation*}
	Therefore, the Lipschitz constant of $\mathcal{P}$ for $t \geq t_0$ is bounded above by $M / \tau r$ for some $M \geq 10$ which is independent of $\tau$. That is, 
	\begin{equation} \label{eq:lipschitz}
		\abs{X''-X'} \leq \frac{M}{\tau r} (t''-t'),
		\qquad 
		\forall \, (X', t'), (X'', t'') \in \mathcal{P} \text{ such that } t_0 \leq t' \leq t''
	\end{equation}
	
	Therefore, if we define 
	\begin{equation*}
		\mbf{V}_Q^1 
		:=
		\left\{ (X, t) \in \ree : \abs{X-Z_1} < \tau r \text{ and } \abs{t-t_1} < \frac{(\tau r)^2}{2M} \right\},
	\end{equation*}
	the parabola $\mathcal{P}$ escapes $\mbf{V}_Q^1$ through its right-most face (which will be useful later to apply Harnack's inequality). Indeed, otherwise, there would exist some point $(\widetilde{X}, \widetilde{t}) \in \mathcal{P} \cap \overline{\mbf{V}_Q^1}$ with $|\widetilde{X} - Z_1| = \tau r$ and $\widetilde{t} \geq t_1$. Then, by \eqref{eq:lipschitz}, we would have 
	\begin{equation*}
		\widetilde{t} - t_1 
		\geq 
		\frac{\tau r}{M} |\widetilde{X} - Z_1|
		=
		\frac{(\tau r)^2}{M}
		> 
		\frac{(\tau r)^2}{2M},
	\end{equation*}
	which contradicts $(\widetilde{X}, \widetilde{t}) \in \overline{\mbf{V}_Q^1}$. \\
	
	Another consequence of \eqref{eq:lipschitz}, for the points $(Z_0, t_0)$ and $(Z_1, t_1)$ defined in the previous substep, that will be useful later is the following
	\begin{equation} \label{eq:V_Q^1_is_close}
		\norm{(Z_0, t_0) - (Z_1, t_1)} 
		=
		\abs{Z_0 - Z_1} + \abs{t_0 - t_1}^{1/2}
		\leq 
		\frac{M}{\tau r} \abs{t_0 - t_1} + \abs{t_0 - t_1}^{1/2}
		\lesssim 
		\tau r.
	\end{equation}
	
	On the other hand, we claim that, using the definitions of $(Z_2, t_2)$ and $\mathcal{P}$, along with \eqref{eq:aperture}, there exists some $\theta \in (0, 1/10)$ such that
	\begin{equation} \label{eq:s_Q-t_2}
		s_Q - t_2 \geq (\theta r)^2,
	\end{equation}
	which allows us to define
	\begin{equation*}
		\mbf{V}_Q^2 := \C_{\frac12\theta r}(Z_2, t_2).
	\end{equation*}
	
	Let us show \eqref{eq:s_Q-t_2}: the proof is based on the simple fact that the slopes of the tangents to the parabola between $(Z_2, t_2)$ and $(P_Q, s_Q)$ are bounded by the slope at $(Z_2, t_2)$. If $(X, t) \in \mathcal{P}$ lies between $(Z_2, t_2)$ and $(P_Q, s_Q)$, we have $t \geq t_2 \geq \left( \frac34 r \right)^2$ (because $\delta(Z_2, t_2) \geq \frac34 r$ for $\delta(P_Q, s_Q) = r$ and $\norm{(P_Q, s_Q) - (Z_2, t_2)} = r/4$). Thus, for some uniform $C \geq 1$, it holds
	\begin{equation*}
		\frac{d\gamma(t)}{dt} 
		=
		\alpha^{-1/2} \frac{1}{2\sqrt{t}}
		\leq
		C \frac{1}{r},
	\end{equation*}
	where we have also used \eqref{eq:aperture}. Therefore, by the fundamental theorem of calculus for $\gamma$ (in 1 variable), $\abs{P_Q - Z_2} \leq C \frac{s_Q - t_2}{r}$. At this point, we claim that it holds $s_Q - t_2 \geq \left( \frac{r}{10C} \right)^2$, which clearly gives \eqref{eq:s_Q-t_2}. Indeed, if this were not true, we would have, putting the above together, 
	\begin{multline*}
		\frac14 r
		=
		\norm{(P_Q, s_Q) - (Z_2, t_2)}
		=
		\abs{P_Q - Z_2} + \abs{s_Q - t_2}^{1/2}
		\leq 
		C \frac{s_Q - t_2}{r} + \abs{s_Q - t_2}^{1/2}
		\\ \leq 
		C \frac{r^2}{100C^2 r} + \frac{r}{10C}
		\leq 
		\frac{r}{10} + \frac{r}{10}
		=
		\frac{r}{5},
	\end{multline*}
	which is clearly false.


	\textbf{Step 2: oscillation of a capacitary potential.} We are going to define in an elementary way a function which has the oscillation that we seek in this lemma, forgetting for the moment the fact that it needs to be supported only inside $E_Q$. We will amend this later.
	
	\textbf{Substep 2.1: definition of the capacitary potential.}
	Denote 
	\begin{equation*}
		\mbf{K}^-_{\tau r}(y_0, s_0)
		:=
		\left( \overline{C_{\tau r}(y_0)} \times [s_0 - (\tau r)^2, s_0 - (a\tau r)^2] \right) \cap \Omega^c,
	\end{equation*}
	so that by the TBCDC (Definition~\ref{def:TBCDC}) and Remark~\ref{rem:capac_cylinders}, there exists a positive measure $\mu$, supported in $\mbf{K}^-_{\tau r}(y_0,s_0)$, such that $\mu(\mbf{K}^-_{\tau r}(y_0,s_0))\approx (\tau r)^n$.
	Define then, using the heat kernel and recalling Definition~\ref{def:capacity},
	\begin{equation*}
		g_Q(X, t) 
		:=
		\Gamma \mu (X, t)
		=
		\iiint_{\mbf{K}^-_{\tau r}(y_0, s_0)} \Gamma(X, t; Y, s) \,d\mu(Y, s), 
		\qquad 
		(X, t) \in \R^{n+1}.
	\end{equation*}
	By definition of capacity, it holds
	\begin{equation} \label{eq:bound_g_Q}
		\norm{g_Q}_\infty \leq 1.
	\end{equation}

	\textbf{Substep 2.2: $g_Q$ is ``large'' close to the boundary.} Fix $(X, t) \in \mbf{V}_Q^1$. Then, for any $(Y, s) \in \mbf{K}^-_{\tau r}(y_0, s_0)$ it holds, using the definitions of $\mbf{V}_Q^1, (y_0, s_0)$ and $\mbf{K}^-_{\tau r}(y_0, s_0)$, and \eqref{eq:V_Q^1_is_close},
	\begin{multline*}
		\norm{(X, t) - (Y, s)} 
		\leq 
		\norm{(X, t) - (Z_1, t_1)} + \norm{(Z_1, t_1) - (Z_0, t_0)} 
		\\ + \norm{(Z_0, t_0) - (y_0, s_0)} + \norm{(y_0, s_0) - (Y, s)}
		\lesssim 
		\tau r,
	\end{multline*} 
	and also, using the definitions of $\mbf{K}^-_{\tau r}(y_0, s_0), (Z_1, t_1), (y_0, s_0)$ and $\mbf{V}_Q^1$, and that $M \geq 10$,
	\begin{equation*}
		t - s
		\geq 
		t - s_0 
		\geq 
		t_1 - t_0 - \abs{t_0 - s_0} - \abs{t - t_1}
		\geq 
		(8\tau r)^2 - (4\tau r)^2 - \frac{(\tau r)^2}{2M}
		\geq
		(\tau r)^2.
	\end{equation*}
	Moreover, by the definitions of $\mbf{V}_Q^1, (Z_1, t_1), (y_0, s_0)$ and $\mbf{K}^-_{\tau r}(y_0, s_0)$, it holds
	\begin{equation*}
		\abs{t-s}
		\leq 
		\abs{t-t_1} + \abs{t_1-t_0} + \abs{t_0-s_0} + \abs{s_0-s}
		\lesssim 
		(\tau r)^2.
	\end{equation*}
	All these provide a lower bound for the heat kernel (see  \eqref{eq:heat_kernel}), so that our choice of $\mu$ yields
	\begin{equation*}
		g_Q(X, t) 
		\gtrsim 
		\iiint_{\mbf{K}^-_{\tau r}(y_0, s_0)} (\tau r)^{-n} \, d\mu(Y, s) 
		=
		(\tau r)^{-n} \mu(\mbf{K}^-_{\tau r}(y_0, s_0))
		\approx
		1.
	\end{equation*}
	
	\textbf{Substep 2.3: $g_Q$ is small away from the boundary.} Fix $\mbf{X} \in \mbf{V}_Q^2$. Then, 
	\begin{equation*}
		\delta(\mbf{X})
		\geq 
		\delta(Z_2, t_2) - \norm{\mbf{X} - (Z_2, t_2)}
		\geq 
		\frac34 r - \frac12 \theta r
		\geq 
		\frac12 r,
	\end{equation*}
	which implies that for any $\mbf{Y} \in \mbf{K}^-_{\tau r} (y_0, s_0)$ we have 
	\begin{equation*}
		\norm{\mbf{X} - \mbf{Y}} 
		\geq 
		\norm{\mbf{X} - (y_0, s_0)} - \norm{(y_0, s_0) - \mbf{Y}}
		\geq 
		\frac{r}{2} - 2 \tau r
		\geq 
		\frac{r}{4}
	\end{equation*}
	because $\tau \leq \tau_0 \leq 1/8$. Then, the bound \eqref{eq:bound_gamma} for the heat kernel and the choice of $\mu$ imply 
	\begin{equation*}
		g_Q(\mbf{X}) 
		\lesssim 
		\iiint_{\mbf{K}^-_{\tau r}(y_0, s_0)} r^{-n} \, d{\mu}(\mbf{Y}) 
		=
		r^{-n} \mu(\mbf{K}^-_{\tau r}(y_0, s_0))
		\approx  
		\tau^n.
	\end{equation*} 
	
	\textbf{Substep 2.4: $g_Q$ oscillates.} As an immediate consequence of Substeps 2.2 and 2.3, if $\tau$ is small enough, then it holds
	\begin{equation} \label{eq:oscillation_g}
		g_Q(\mbf{Z}_1') - g_Q(\mbf{Z}_2') 
		\gtrsim 
		1, 
		\qquad 
		\forall \; \mbf{Z}_1' \in \mbf{V}_Q^1, \; \mbf{Z}_2' \in \mbf{V}_Q^2.
	\end{equation}
	

	\textbf{Step 3: restricting to $E_Q$.} Although the potential $g_Q$ shows the desired oscillations, it is supported on the whole $Q$. We need a function which behaves similarly, but is supported only on $E_Q \subset Q$.
	
	\textbf{Substep 3.1: definition of $f_Q$ and $u_Q$.}	
	Define now 
	\begin{equation*}
		f_Q := g_Q \, \mathbf{1}_{E_Q}, 
		\qquad 
		u_Q(\mbf{X}) 
		:=
		\iint_\Sigma f_Q \, d\omega^{\mbf{X}}
		=
		\iint_{E_Q} g_Q \, d\omega^{\mbf{X}}, 
		\quad \mbf{X} \in \Omega.
	\end{equation*}
	Since $g_Q$ is continuous and caloric in $\Omega$ (because $\text{spt}(\mu) \subset \Omega^c$ and $\Gamma$ is caloric in its first entry), and $\Omega$ is Wiener regular (for it satisfies the TBCDC, see \cite[Lemma 3.12]{MP}, \cite[Lemma 3.24]{HHK}), we can recover it from its boundary values via integration against the caloric measure:
	\begin{equation*}
		g_Q(\mbf{X}) 
		=
		\iint_\Sigma g_Q \, d\omega^{\mbf{X}}, 
		\quad \mbf{X} \in \Omega.
	\end{equation*} 
	These allow us to estimate the difference between $u_Q$ and $g_Q$ by using \eqref{eq:bound_g_Q}, for $\mbf{X} \in \Omega$,
	\begin{equation} \label{eq:diff_u_g}
		\abs{u_Q(\mbf{X}) - g_Q(\mbf{X})}
		=
		\bigg| \iint_{\Sigma \setminus E_Q} g_Q \, d\omega^{\mbf{X}} \bigg|
		\leq
		\omega^{\mbf{X}} (\Sigma \setminus E_Q)
		=
		\omega^{\mbf{X}} (\Sigma \setminus Q) + \omega^{\mbf{X}} (Q \setminus E_Q).
	\end{equation}
	
	\textbf{Substep 3.2: smallness of the first term in \eqref{eq:diff_u_g}.}
	First note that H\"{o}lder continuity at the boundary (Lemma~\ref{lem:holder}) and the hypothesis \eqref{eq:hyp_epsilon_1} imply
	\begin{equation} \label{eq:small_holder}
		\omega^{\mbf{X}} (\Sigma \setminus Q) 
		\lesssim 
		(\varepsilon')^\gamma.
	\end{equation} 
	As a technical note, since Lemma~\ref{lem:holder} is stated only for compactly supported boundary data, one should apply it to the truncated sets $\Sigma \cap C_N \setminus Q$, and then let $N \to \infty$ because constants are independent of $N$.
	
	\textbf{Substep 3.3: smallness of the second term in \eqref{eq:diff_u_g}.}
	On the other hand, we claim that, using a Harnack chain argument (see Lemma~\ref{lem:harnack}) along $\mathcal{P}$, and later the hypothesis \eqref{eq:hyp_epsilon_2}, it holds
	\begin{equation} \label{eq:small_size_E_Q}
		\omega^{(X, t)}(Q \setminus E_Q)
		\leq
		C(\tau) \, \omega^{(P_Q, s_Q)}(Q \setminus E_Q)
		\leq 
		C(\tau) \, \varepsilon',
		\qquad 
		(X, t) \in \mbf{V}_Q^1 \cup \mbf{V}_Q^2.
	\end{equation}
	
	Let us show how this Harnack argument along a parabola works. The idea is to work in small enough time increments to always stay close to the parabola, and that we only have to apply Harnack's inequality a limited amount of times. We abbreviate $u(\mbf{X}) := \omega^{\mbf{X}}(Q \setminus E_Q)$.
	
	Indeed, consider first the case that $(X, t) \in \mbf{V}_Q^1$. Let $(\widetilde{X}_1, \widetilde{t}_1) \in \mathcal{P}$ the point at which $\mathcal{P}$ exits $\C_{\tau r}(Z_1, t_1)$, namely $||(\widetilde{X}_1, \widetilde{t}_1) - (Z_1, t_1)|| = \tau r$ and $\widetilde{t}_1 \geq t_1$. Let us show in all detail that 
	\begin{equation} \label{eq:harnack_chain}
		u(X, t) \leq C_{\mathrm{H}} u(\widetilde{X}_1, \widetilde{t}_1), 
		\quad \text{and} \quad 
		\widetilde{t}_1 - t_1 \geq \frac23 \frac{(\tau r)^2}{M},
	\end{equation}
	where $C_{\mathrm{H}} > 0$ only depends on $n$. 
	\begin{itemize}
		\item Let us consider the cylinder and its (spatial) subset
		\begin{equation*}
			D \times (T_{\min}, T_{\max}) := \C_{\frac32 \tau r} (Z_1, t_1)
			\quad  \text{and} \quad  
			D' := C_{\tau r} (Z_1).
		\end{equation*} 
		Then it is clear that $\dist (D', \partial D) \geq \tau r / 2$. Moreover, the cylinder is contained in $\Omega$ because $\delta(Z_1, t_1) \geq 2 \tau r$ by definition of $(Z_1, t_1)$ and our choice of $(Z_0, t_0)$. Hence $\omega^{(\cdot)}(Q \setminus E_Q)$ is caloric in the cylinder.
		
		\item Since, by definition of $\mbf{V}_Q^1$ with $M \geq 10$, it holds $\abs{t - t_1} \leq \frac{(\tau r)^2}{2M} \leq \frac{(\tau r)^2}{20}$, we have also 
		\begin{equation*}
			t-T_{\min}
			= 
			t - \left(t_1 - \Big(\frac32 \tau r\Big)^2\right) 
			\geq 
			\left(\frac94 - \frac{1}{20}\right) (\tau r)^2 
			\geq 
			(\tau r)^2.
		\end{equation*}
		
		\item Moreover, by definition of $\mbf{V}_Q^1$ and our choice of $(\widetilde{X}_1, \widetilde{t}_1)$, we can estimate
		\begin{equation*}
			|X - \widetilde{X}_1| + |t - \widetilde{t}_1|^{1/2} 
			\leq 
			\norm{(X, t) - (Z_1, t_1)} + \|(Z_1, t_1) - (\widetilde{X}_1, \widetilde{t}_1)\|
			\leq 
			3\tau r.
		\end{equation*}
		
		\item Furthermore, it also holds $\widetilde{t}_1 - t_1 \geq \frac23 \frac{(\tau r)^2}{M}$, which is a quantitative reflection of the fact that $\mathcal{P}$ exits $\mbf{V}_Q^1$ through its right face. Indeed, if this were false, then by \eqref{eq:lipschitz}, 
		\begin{align*}
			\|(\widetilde{X}_1, \widetilde{t}_1) - (Z_1, t_1)\|
			& =
			|\widetilde{X}_1 - Z_1| + |\widetilde{t}_1 - t_1|^{1/2}
			\leq 
			\frac{M}{\tau r} |\widetilde{t}_1 - t_1| + |\widetilde{t}_1 - t_1|^{1/2}
			\\ & <
			\frac23 \tau r + \left( \frac{2}{3M} \right)^{1/2} \tau r
			< 
			\tau r
		\end{align*}
		noting that $M \geq 10$, which would be a contradiction with the choice of $(\widetilde{X}_1, \widetilde{t}_1)$.
	\end{itemize}
	
	With all the above in mind (and recalling that $\ell(Q) \leq 1$), the time-directed Harnack's inequality from Lemma~\ref{lem:harnack} yields
	\begin{equation*}
		u(X, t)
		\leq 
		u(\widetilde{X}_1, \widetilde{t}_1) \exp \left[ C(n) \left( \frac{(2\tau r)^2}{ \frac23 \frac{(\tau r)^2}{M}} + \frac{(2\tau r)^2}{\min\left\{1, (\tau r)^2, (\frac{\tau r}{2})^2\right\}} + 1 \right) \right]
		\lesssim_n 
		u(\widetilde{X}_1, \widetilde{t}_1),
	\end{equation*}
	which finishes the proof of \eqref{eq:harnack_chain}.
	
	Now that this first step is clear, let us continue in a similar fashion. Indeed, find $(\widetilde{X}_2, \widetilde{t}_2) \in \mathcal{P}$ with $\widetilde{t}_2 \geq \widetilde{t}_1$ such that $\|(\widetilde{X}_2, \widetilde{t}_2) - (\widetilde{X}_1, \widetilde{t}_1)\| = \tau r$. By the very same arguments leading to \eqref{eq:harnack_chain}, it holds $u(\widetilde{X}_1, \widetilde{t}_1) \leq C_{\mathrm{H}} u(\widetilde{X}_2, \widetilde{t}_2)$ and $\widetilde{t}_2 - \widetilde{t}_1 \geq \frac23 \frac{(\tau r)^2}{M}$. 
	
	Iterating this, we construct a sequence of points $(\widetilde{X}_j, \widetilde{t}_j)$ for which 
	\begin{equation*}
		u(\widetilde{X}_j, \widetilde{t}_j) 
		\leq 
		C_{\mathrm{H}} u(\widetilde{X}_{j+1}, \widetilde{t}_{j+1})
		\quad \text{and} \quad 
		\widetilde{t}_{j+1} - \widetilde{t}_j \geq \frac23 \frac{(\tau r)^2}{M}
	\end{equation*}
	Moreover, since by definition of $(Z_0, t_0)$ and $(Z_1, t_1)$, and \eqref{def:y'}, it holds 
	\begin{equation*}
		s_Q - t_1
		\leq 
		s_Q - t_0
		\leq 
		s_Q - s'
		\lesssim 
		r^2, 
	\end{equation*}
	we only need $j_0 \lesssim \tau^{-2}$ steps for our sequence of $\widetilde{t}_{j_0}$ to surpass $s_Q$, i.e. $\widetilde{t}_{j_0} \geq s_Q$. Therefore, 
	\begin{equation*}
		u(X, t) 
		\leq 
		C_{\mathrm{H}} u(\widetilde{X}_1, \widetilde{t}_1) 
		\leq
		\ldots 
		\leq 
		C_{\mathrm{H}}^{j_0-1} u(\widetilde{X}_{j_0-1}, \widetilde{t}_{j_0-1}) 
		\leq 
		C_{\mathrm{H}}^{j_0} u(P_Q, s_Q) 
		= 
		C(\tau) \, u(P_Q, s_Q),
	\end{equation*}
	which finishes the justification of \eqref{eq:small_size_E_Q} for $(X, t) \in \mbf{V}_Q^1$.
	
	Now let us turn to the case that $(X, t) \in \mbf{V}_Q^2$. In this case, the argument is easier because only one application of Harnack's inequality is needed. Let us be more precise.
	\begin{itemize}
		\item Consider the cylinder $D \times (T_{\min}, T_{\max}) := \C_{r/2}(Z_2, t_2)$ and $D' := C_{r/3}(Z_2)$. Then, clearly $\dist (D', \partial D) \geq r/6$. Moreover, the cylinder is contained in $\Omega$ because 
		\begin{equation*}
			\delta(Z_2, t_2) \geq \delta(P_Q, s_Q) - \norm{(Z_2, t_2) - (P_Q, s_Q)} = r - \frac{r}{4} = \frac34 r.
		\end{equation*}
		Also, $P_Q \in D'$ because $\norm{(Z_2, t_2) - (P_Q, s_Q)} = r/4 < r/3$. The same applies to the spatial component of any point of $\mbf{V}_Q^2$ by definition of this set, because $\theta r / 2 < r/3$. 
		
		\item We also have, since $\theta < 1/10$, 
		\begin{equation*}
			t - T_{\min} = t - \left(t_2 - \left(\frac{r}{2}\right)^2\right) \geq \frac{r^2}{4} - \abs{t-t_2} \geq \frac{r^2}{4} - \left( \frac{\theta r}{2} \right)^2 \geq \frac{r^2}{8}.
		\end{equation*}
		
		\item Moreover, by definition of $\mbf{V}_Q^2$ and $(Z_2, t_2)$, it holds 
		\begin{equation*}
			\norm{(X, t) - (P_Q, s_Q)} \leq \norm{(X, t) - (Z_2, t_2)} + \norm{(Z_2, t_2) - (P_Q, s_Q)} \leq r.
		\end{equation*}
		
		\item Lastly, \eqref{eq:s_Q-t_2} yields 
		\begin{equation*}
			s_Q - t \geq s_Q - t_2 - \abs{t-t_2} \geq (\theta r)^2 - \left( \frac{\theta}{2} r \right)^2 = \frac34 (\theta r)^2.
		\end{equation*}
		This is the key estimate motivating the construction of $\mbf{V}_Q^2$. 
	\end{itemize}
	With all these, Harnack's inequality from Lemma~\ref{lem:harnack} yields $u(X, t) \lesssim u(P_Q, s_Q)$ with uniform constants, which finishes the proof of \eqref{eq:small_size_E_Q} also for points in $\mbf{V}_Q^2$.
	
	\textbf{Substep 3.4: $u_Q$ oscillates.}
	Putting together \eqref{eq:oscillation_g}, \eqref{eq:diff_u_g}, \eqref{eq:small_holder} and \eqref{eq:small_size_E_Q}, we have shown that  if we choose $\tau$ small enough, and then $\varepsilon' = \varepsilon'(\tau)$ small enough, it holds
	\begin{equation} \label{eq:u_Q_oscillates}
		u_Q(\mbf{Z}_1') - u_Q(\mbf{Z}_2') 
		\gtrsim 
		1, 
		\qquad 
		\forall \; \mbf{Z}_1' \in \mbf{V}_Q^1, \; \mbf{Z}_2' \in \mbf{V}_Q^2.
	\end{equation}
	

	\textbf{Step 4: creating the gradient and definition of $\mbf{V}_Q$.} We have successfully constructed an oscillating function $u_Q$. It remains to check that its gradient reflects these oscillations, in a region $\mbf{V}_Q$ that stays away from the boundary, as suggested in the statement.
	
	\textbf{Substep 4.1: Poincaré along the parabola.} 
	We claim that the estimate \eqref{eq:u_Q_oscillates} and a Poincaré-like argument integrating along parabolas parallel to $\mathcal{P}$ imply that 
	\begin{multline} \label{eq:poincare_parabola}
		\left(\frac{\tau r}{M} \right)^{n+2} 
		\leq
		\abs{\mbf{V}_Q^1} 
		\lesssim 
		\iiint_{\mbf{V}_Q^1} \abs{u_Q(X, t) - u_Q(X + Z_2 - Z_1, t + t_2 - t_1)}^2 dX dt
		\\ \lesssim 
		\tau^{-2} r^2 \iiint_{\widetilde{\mathcal{P}} + \C_{\tau r}} \abs{\nabla u_Q(X, t)}^2 dX dt 
		+ r^4 \iiint_{\widetilde{\mathcal{P}} + \C_{\tau r}} \abs{\partial_t u_Q(X, t)}^2 dX dt,
	\end{multline}
	where we denote 
	\begin{equation*}
		\widetilde{\mathcal{P}} 
		:=
		\{ (Z, t) \in \mathcal{P} : t_1 \leq t \leq t_2 \}.
	\end{equation*}
	
	Let us add more details on the proof of \eqref{eq:poincare_parabola}.
	First note that the second inequality is true because if $(X, t) \in \mbf{V}_Q^1$, then $(X, t) + (Z_2, t_2) - (Z_1, t_1) \in \mbf{V}_Q^2$ because the ``radius'' of $\mbf{V}_Q^2$ is larger than that of $\mbf{V}_Q^1$ ($\tau r < \frac12 \theta r$ if we choose $\tau$ small enough).
	Next, let us justify the last inequality in \eqref{eq:poincare_parabola}. Integrating along parabolas parallel to $\mathcal{P}$ using the fundamental theorem of calculus in 1 variable,
	\begin{align*}
		\iiint_{\mbf{V}_Q^1} & \abs{u_Q(X, t) - u_Q(X + Z_2 - Z_1, t + t_2 - t_1)}^2 dX dt
		\\ & =
		\iiint_{\mbf{V}_Q^1 - (Z_1, t_1)} \abs{u_Q(X + Z_1, t + t_1) - u_Q(X + Z_2, t + t_2)}^2 dX dt
		\\ & \lesssim 
		r^2 \iiint_{\mbf{V}_Q^1 - (Z_1, t_1)} \int_{t_1}^{t_2} \bigg( \nabla u_Q((X, t) + \bm{\gamma}(s)) \cdot \dot{\gamma}(s) + \partial_t u_Q((X, t)+\bm{\gamma}(s)) \bigg)^2 ds \, dX dt 
	\end{align*}	
	where in the inequality, we have used Cauchy-Schwarz and the bound $t_2-t_1 \leq s_Q - s' \lesssim r^2$ using the definitions of $(Z_2, t_2), (Z_1, t_1)$ and \eqref{def:y'}. We also denote $\dot{\gamma}(s) = d\gamma(s) / ds$, which can be estimated by \eqref{eq:lipschitz} to obtain	
	\begin{align*}
		\iiint_{\mbf{V}_Q^1} & \abs{u_Q(X, t) - u_Q(X + Z_2 - Z_1, t + t_2 - t_1)}^2 dX dt
		\\ & \lesssim 
		r^2 \! \iiint_{\mbf{V}_Q^1 - (Z_1, t_1)} \int_{t_1}^{t_2} \!\! \left( \abs{\nabla u_Q((X, t) + \bm{\gamma}(s))}^2 \abs{\dot{\gamma}(s)}^2 + \abs{\partial_t u_Q((X, t)+\bm{\gamma}(s))}^2 \right) ds \, dX dt 
		\\ & \lesssim
		\tau^{-2} \iiint_{\mbf{V}_Q^1 - (Z_1, t_1)} \int_{t_1}^{t_2} \abs{\nabla u_Q((X, t) + \bm{\gamma}(s))}^2 ds \, dX dt 
		\\ & \quad + r^2 \iiint_{\mbf{V}_Q^1 - (Z_1, t_1)} \int_{t_1}^{t_2} \abs{\partial_t u_Q((X, t)+\bm{\gamma}(s))}^2 ds \, dX dt
		\\ & \lesssim 
		\tau^{-2 }r^2 \iiint_{\widetilde{\mathcal{P}} + \C_{\tau r}} \abs{\nabla u_Q(X', t')}^2 dX' dt'
		+ r^4 \iiint_{\widetilde{\mathcal{P}} + \C_{\tau r}} \abs{\partial_t u_Q(X', t')}^2 dX' dt',
	\end{align*}
	where the last step is justified as follows. First, we have used the fact that the bundle of parabolas satisfies 
	\begin{equation*}
		\Big\{ (X, t) + \bm{\gamma}(s) : (X, t) \in V_Q^1 - (Z_1, t_1) \text{ and } s \in [t_1, t_2] \Big\}
		\subset 
		\widetilde{\mathcal{P}} + \C_{\tau r}
	\end{equation*}
	just by the definitions of $\mbf{V}_Q^1$ and $\widetilde{\mathcal{P}}$. These parabolas may overlap (their curvature changes with time), but we can control the integrals using Fubini and a change of variables
	\begin{align*}
		\iiint_{\mbf{V}_Q^1 - (Z_1, t_1)} \int_{t_1}^{t_2} & \abs{\nabla u_Q((X, t) + \bm{\gamma}(s))}^2 ds \, dX dt
		\\ & =
		\int_{t_1}^{t_2} \iiint_{\mbf{V}_Q^1 - (Z_1, t_1)} \abs{\nabla u_Q((X, t) + \bm{\gamma}(s))}^2  dX dt ds
		\\ & =
		\int_{t_1}^{t_2} \iiint_{\mbf{V}_Q^1 - (Z_1, t_1)  + \bm{\gamma}(s)} \abs{\nabla u_Q(X', t')}^2  dX' dt' ds
		\\ & \leq 
		\int_{t_1}^{t_2} \iiint_{\widetilde{\mathcal{P}} + \C_{\tau r}} \abs{\nabla u_Q(X', t')}^2  dX' dt' ds
		\\ & \lesssim 
		r^2 \iiint_{\widetilde{\mathcal{P}} + \C_{\tau r}} \abs{\nabla u_Q(X', t')}^2  dX' dt',
	\end{align*}
	where in the last step we have used $t_2-t_1 \lesssim r^2$ as before. Note that the key point is to use Fubini appropriately so that the changes of variable are mere translations. The term with time derivatives goes similarly, and this finishes the justification of \eqref{eq:poincare_parabola}. 
	
	On the other hand, using that both $u_Q$ and $\partial_{x_i} u_Q$ solve the heat equation, we can use Caccioppoli's estimate (Lemma~\ref{lem:caccioppoli}) to get, covering $\widetilde{\mathcal{P}} + \C_{\tau r}$ by a family of boxes $\{\mbf{I}\}_{\mbf{I} \in \mathcal{I}}$ with sidelength $\ell(\mbf{I}) = c(n) \tau r$, where $c(n) > 0$ is small enough so that $2 \mbf{I} \subset \widetilde{\mathcal{P}} + \C_{2 \tau r}$ and these dilated boxes have bounded overlap, so that we can compute
	\begin{align} \label{eq:caccioppoli_parabola}
		\iiint_{\widetilde{\mathcal{P}} + \C_{\tau r}} \abs{\partial_t u_Q(\mbf{X})}^2 d\mbf{X}
		& \leq 
		\sum_{\substack{\mbf{I} \in \mathcal{I} \\ \mbf{I} \cap (\widetilde{\mathcal{P}} + \C_{\tau r}) \neq \emptyset }} \iiint_{\mbf{I}} \abs{\partial_t u_Q(\mbf{X})}^2 d\mbf{X}
		\nonumber
		\\ & \lesssim
		\sum_{i=1}^n \sum_{\substack{\mbf{I} \in \mathcal{I} \\ \mbf{I} \cap (\widetilde{\mathcal{P}} + \C_{\tau r}) \neq \emptyset }} \iiint_{\mbf{I}} \abs{\partial_{x_i} \partial_{x_i} u_Q(\mbf{X})}^2 d\mbf{X}
		\nonumber
		\\ & \lesssim 
		\sum_{i=1}^n \sum_{\substack{\mbf{I} \in \mathcal{I} \\ \mbf{I} \cap (\widetilde{\mathcal{P}} + \C_{\tau r}) \neq \emptyset }} \ell(\mbf{I})^{-2} \iiint_{2\mbf{I}} \abs{\partial_{x_i} u_Q(\mbf{X})}^2 d\mbf{X}
		\nonumber
		\\ & \lesssim 
		(\tau r)^{-2} \iiint_{\widetilde{\mathcal{P}} + \C_{2 \tau r}} \abs{\nabla u_Q(\mbf{X})}^2 d\mbf{X}.
	\end{align}
	
	The estimates \eqref{eq:poincare_parabola} and \eqref{eq:caccioppoli_parabola} yield 
	\begin{equation} \label{eq:GMT_without_delta}
		\iiint_{\widetilde{\mathcal{P}} + \C_{2 \tau r}} \abs{\nabla u_Q(\mbf{X})}^2 d\mbf{X}
		\gtrsim 
		\tau^{n+4} r^n.
	\end{equation}
	
	\textbf{Substep 4.2: definition of $\mbf{V}_Q$.} We finally define (see Figure~\ref{fig:V_Q})
	\begin{equation*}
		\mbf{V}_Q := \widetilde{\mathcal{P}} + \C_{2 \tau r}.
	\end{equation*}
	
	\begin{figure}[h!]
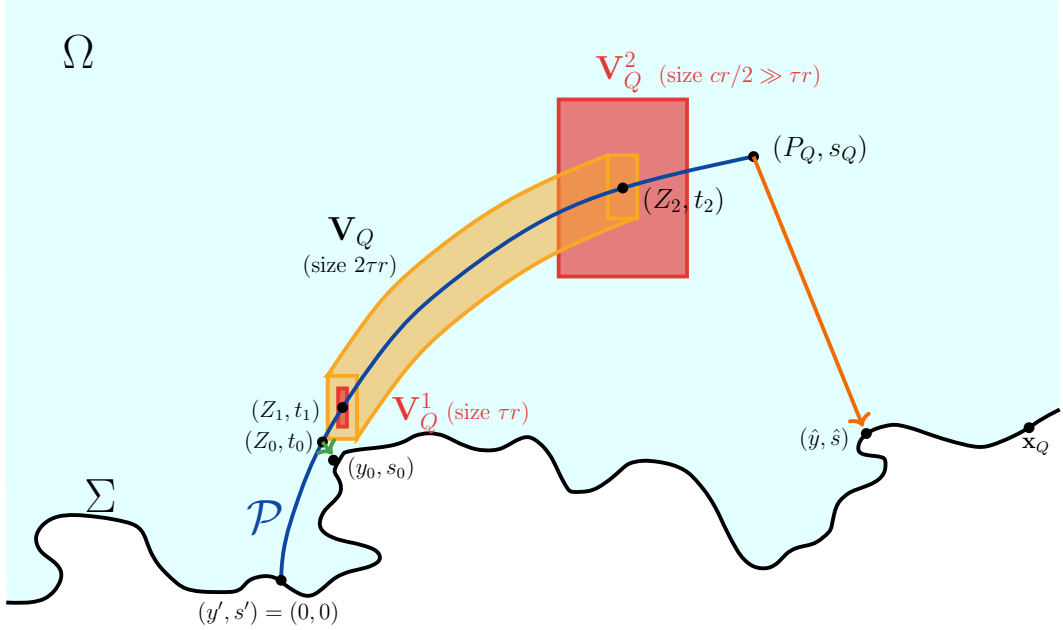

		\centering 
		\resizebox{0.9\textwidth}{!}{

		}
		\caption{Diagram showing the shape of $\mbf{V}_Q$, adapted to the shape of the parabola $\mathcal{P}$ that goes from $\mbf{V}_Q^1$ to $\mbf{V}_Q^2$.}
		\label{fig:V_Q}
	\end{figure}

	Lastly, we are left to show the bounds for $\mbf{V}_Q$ asserted in \eqref{eq:loc_V_Q}. First, let us note that $\delta(\mbf{Z}) \geq 4 \tau r$ for each $\mbf{Z} \in \widetilde{\mathcal{P}}$ by the definition of $(Z_0, t_0)$, which trivially implies that $\delta(\mbf{X}) \geq  \tau r$ for every $\mbf{X} \in \mbf{V}_Q$. Moreover, this allows us to insert the distance to the boundary in the estimate \eqref{eq:GMT_without_delta}, obtaining
	\begin{equation*}
		\iiint_{\mbf{V}_Q} \abs{\nabla u_Q(\mbf{X})}^2 \delta(\mbf{X}) \, d\mbf{X}
		\gtrsim 
		\tau^{n+5} r^{n+1}
		=
		\tau^{n+5} (\delta(\mbf{P}_Q))^{n+1}
		\approx
		\tau^{n+5} (\varepsilon')^{n+1} \sigma(Q)
	\end{equation*}
	where we have also used the definition of $r$, \eqref{eq:hyp_epsilon_1} and p-ADR. This shows the main estimate claimed in the statement.
	
	On the other hand, for any $(Z, s) \in \widetilde{\mathcal{P}}$, we have, by \eqref{def:y'} and \eqref{eq:aperture}, recalling $(y', s') = (0, 0)$,
	\begin{equation*}
		\norm{(Z, s) - (y', s')} 
		=
		\abs{\gamma(s)} + s^{1/2}
		= 
		(1+\alpha^{-1/2}) \, s^{1/2} 
		\lesssim 
		s^{1/2}
		\leq
		s_Q^{1/2} 
		\leq 
		\norm{(P_Q, s_Q) - (y', s')} 
		\leq
		Cr.
	\end{equation*} 
	Therefore, for any $\mbf{X} \in \mbf{V}_Q$, we may find some $\mbf{Z} \in \widetilde{\mathcal{P}}$ with $\norm{\mbf{X} - \mbf{Z}} \leq 2 \tau r$ (by definition of $\mbf{V}_Q$), so that recalling also \eqref{eq:hyp_epsilon_1} and \eqref{def:y'}, it holds
	\begin{multline*}
		\norm{\mbf{X} - \mbf{x}_Q} 
		\leq 
		\norm{\mbf{X} - \mbf{Z}} + \norm{\mbf{Z} - (y', s')} + \norm{(y', s') - \mbf{P}_Q} + \norm{\mbf{P}_Q - \mbf{x}_Q}
		\\ \leq 
		2 \tau r + Cr + Cr + c_{\mathrm{B}} \ell(Q)
		\leq 
		C \varepsilon' \ell(Q) + c_{\mathrm{B}} \ell(Q)
		\leq 
		\ell(Q)
	\end{multline*}
	once we consider $\varepsilon'$ and $c_{\mathrm{B}}$ small enough. This finishes the proof of \eqref{eq:loc_V_Q}.
\end{proof}

\begin{remark}
	Under more favorable background hypotheses (see \cite[Section 3]{HHK}), like TBADR (time-backwards ADR, see \cite{GH_RH}) or TBHCC (time-backward Hausdorff content condition, see \cite{HHK}), the proof of Proposition~\ref{prop:GMT} may use layer potentials $g_Q$ in Step 2 that are maybe more recognizable for the reader:
	\begin{itemize}
		\item In presence of the TBADR assumption, one can define 
		\begin{equation*}
			g_Q(X, t) 
			:=
			\frac{1}{\tau r} \iint_{\C_{\tau r}(y_0, s_0) \cap \Sigma \cap \T_{<s_0}} \Gamma(X, t; Y, s) \, d\sigma(Y, s),
			\qquad 
			(X, t) \in \R^{n+1},
		\end{equation*}
		and the estimates in Substeps 2.2, 2.3 follow because $\sigma(\C_{\tau r}^-(y_0, s_0) \cap \Sigma \cap \T_{<s_0}) \approx (\tau r)^{n+1}$ by TBADR, and the estimate \eqref{eq:bound_g_Q} can be verified combining p-ADR (only the upper bound) and a decomposition in annuli (similar to \cite[p. 1558]{GH_RH}).
		
		\item In turn, in presence of the TBHCC with exponent $\varepsilon_0$ (as in \cite[Definition 3.5]{HHK}), one can define 
		\begin{equation*}
			g_Q(X, t) 
			:=
			\frac{1}{(\tau r)^{\varepsilon_0}} \iiint_{\mbf{K}_{\tau r}^-(y_0, s_0)} \Gamma(X, t; Y, s) \, d\widetilde{\mu}(Y, s),
			\qquad 
			(X, t) \in \R^{n+1},
		\end{equation*}
		where $\widetilde{\mu}$ is the Frostman measure naturally associated to the TBHCC condition, as in \cite[Lemma 3.13]{HHK}, so that it holds $\widetilde{\mu}(\mbf{K}_{\tau r}^-(y_0, s_0)) \approx (\tau r)^{n+\varepsilon_0}$, and therefore all the estimates in Step 2 are still true.
	\end{itemize}
\end{remark}



\section{The corona decomposition for the caloric measure} \label{sec:corona}

In this section, we show that if p-CME holds in $\Omega$, then we can construct a corona decomposition for $\omega$ with respect to $\sigma$. This is a multiscale condition asserting, roughly speaking, that caloric measure and surface measure are quantitatively comparable on a large collection of dyadic cubes.

The p-CME assumption enters precisely in the proof of the packing condition for the low density cubes. Once Proposition~\ref{prop:GMT} is available, the argument follows the combinatorial argument introduced in \cite{GMT}, where the oscillating caloric functions control the contribution of each individual cube, and p-CME allows one to sum across these cubes.

At this point, one might hope to proceed exactly as in \cite{GMT}. However, a fundamental new difficulty arises in the parabolic setting. For the subsequent construction of approximating domains in Proposition~\ref{prop:WHSA}, it is crucial that the poles of caloric measure lie far in the future, so that the associated adjoint Green function controls the entire stopping-time region. On the other hand, it is also essential for Proposition~\ref{prop:GMT} that poles are very close to the boundary. Consequently, the one-pole argument of \cite{GMT} cannot be adapted directly.

In Theorem~\ref{th:corona}, we solve this difficulty by running the stopping-time argument with two poles simultaneously. The two-pole construction can subsequently be reduced to a one-pole formulation through applications of the time-directed Harnack inequality.

\begin{definition}[Corona decomposition for caloric measure]\label{def:corona}
Let $\Omega \subset \ree$ be an open set such that $\Sigma := \pom$ is p-ADR (see Definition~\ref{def:ADR}). We say that the \textit{caloric measure admits a corona decomposition}\footnote{The reader should not confuse this corona decomposition with that from Definition~\ref{def:unilateralcorona}: both are corona decompositions, but one measures how frequently $\Sigma = \pom$ can be well approximated by regular Lip(1,1/2) graphs, and the other checks the scales in which, in some appropriate sense, $\omega$ and $\sigma$ are comparable.} if 
there exists $C_{\mathrm{cor}} > 0$ such that:
	\begin{enumerate}
		\item The dyadic grid $\D$ (from Lemma~\ref{lem:cubes}) can be decomposed as a disjoint union $\D = \G \sqcup \B$, where the good cubes are organized into subtrees $\G = \bigsqcup_{\S \in \mathcal{S}} \S$, and each tree $\S$ is semi-coherent (see Definition~\ref{def:unilateralcorona}),
		\item The maximal cubes $Q(\S)$ of the subtrees $\S$, and the bad cubes $\B$, satisfy the Carleson packing condition
		\begin{equation*}
			\sum_{\substack{\S \in \mathcal{S} \\ Q(\S) \subset Q}} \sigma(Q(\S)) + \sum_{\substack{Q' \in \B \\ Q' \subset Q}} \sigma(Q')
			\leq 
			C_{\mathrm{cor}} \, \sigma(Q), 
			\qquad 
			\forall \, Q \in \D.
		\end{equation*}
		\item For each tree $\S \in \mathcal{S}$ there exists a pole $\mbf{Y}_\S = (Y_\S, s_\S)$ such that 
		\begin{enumerate}
			\item The pole is far in the future
			\begin{equation*}
				s_\S \geq T_{\max}(Q(\S)) + (500\diam(Q(\S)))^2 ,
			\end{equation*}
			and accordingly far from the boundary
			\begin{equation*}
				\delta(\mbf{Y}_\S) \geq 100\diam(Q(\S)).
			\end{equation*}
			\item $\omega^{\mbf{Y}_\S}(Q(\S)) \geq C_{\mathrm{cor}}^{-1}$.
			\item The normalized measure 
			\begin{equation*}
				\mu := \frac{\sigma(Q(\S))}{\omega^{\mbf{Y}_\S}(Q(\S))} \omega^{\mbf{Y}_\S}
			\end{equation*}
			satisfies 
			\begin{equation*}
				\frac{1}{C_{\mathrm{cor}}}
				\leq 
				\frac{\mu(Q)}{\sigma(Q)} 
				\leq 
				\left( \iint_Q (\mathcal{M}_\sigma(\restr{\mu}{2Q(\S)})(\mbf{x}))^{1/2} d\sigma(\mbf{x}) \right)^2
				\leq 
				C_{\mathrm{cor}},
				\qquad 
				\forall Q \in \S,
			\end{equation*}
			where $\mathcal{M}_\sigma$ denotes the Hardy-Littlewood maximal operator with respect to the underlying measure $\sigma$, namely 
			\begin{equation*}
				\mathcal{M}_\sigma (\nu)(\mbf{x})
				:=
				\sup_{\substack{\mbf{y} \in \Sigma, r > 0 : \\ \mbf{x} \in \C_r(\mbf{y})}} \frac{\nu(\C_r(y))}{\sigma(\C_r(y))}
			\end{equation*}
			for a measure $\nu$ over $\Sigma$, and $\mbf{x} \in \Sigma$.
		\end{enumerate}
	\end{enumerate} 
\end{definition}

\begin{theorem}[p-CME implies corona decomposition for caloric measure] \label{th:corona}
	Let $\Omega \subset \ree$ be an open set satisfying the interior corkscrew condition (see Definition~\ref{def:corkscrew}) and the TBCDC (see Definition~\ref{def:TBCDC}), so that $\Sigma := \pom$ is p-ADR (see Definition~\ref{def:ADR}). If p-CME holds in $\Omega$ (see Definition~\ref{def:CME}), then the caloric measure admits a corona decomposition with constant $C_{\mathrm{cor}} = C_{\mathrm{cor}}(n, C_{\mathrm{CKS}}, C_{\mathrm{ADR}}, C_{\mathrm{CDC}}, C_{\mathrm{CME}})$. 
\end{theorem}
\begin{proof}
	Fix $\tau, \varepsilon' > 0$ from Proposition~\ref{prop:GMT}, and we will allow implicit constants to depend on them. We divide the proof into a series of steps.
	
	\textbf{Step 1: creation of a single stopping-time regime, with 2 poles.} Let us first explain how to run a single stopping-time argument, which we will later iterate. For that purpose, fix $Q_0 \in \D$ with $\ell(Q_0) \leq 1$. We consider a pole $\X_0 \in \C_{c_{\mathrm{B}}r_{Q_0}}(\mbf{x}_{Q_0})$ with $\delta(\X_0) \approx \varepsilon' \ell(Q_0)$; indeed, one may simply choose the corkscrew point associated to a cube $Q_0' \in \D$ containing the center of $Q_0$, namely $\mbf{x}_{Q_0} \in Q_0'$, satisfying $\ell(Q_0') \approx \varepsilon' \ell(Q_0)$ (note that $\varepsilon' \ll c_{\mathrm{B}}$). Then, Bourgain's estimate (Lemma~\ref{lem:bourgain}) yields $\omega^{\X_0}(Q_0) \gtrsim 1$. We also fix another pole a little forward in time:
	\begin{equation*}
		\X_1 := \X_0 + (0, \lambda \ell(Q_0)^2), 
	\end{equation*}
	where $\lambda = \lambda(\varepsilon') > 0$ is a small fixed number, to be determined later. The translation happens only in the time direction.
	
	Now fix a large constant $N = N(\tau, \varepsilon') \gg 1$. Subdivide $Q_0$ dyadically, and stop at any cube $Q$ whenever it holds
	\begin{equation*}
		\frac{\omega^{\X_0}(Q)}{\sigma(Q)} < \frac{1}{N} \frac{\omega^{\X_0}(Q_0)}{\sigma(Q_0)}, 
		\qquad \text{or} \qquad 
		\left( \frac{1}{\sigma(Q)} \iint_Q \left( \mathcal{M}_\sigma  \omega^{\X_1}  \right)^{1/2} \, d\sigma \right)^2 > N^2 \, \frac{\omega^{\X_1}(Q_0)}{\sigma(Q_0)}.
	\end{equation*}
	Cubes satisfying the first condition will be called $\mathrm{LD}(Q_0)$ (``low density''), and the ones satisfying the second but not the first one will be called $\mathrm{HD}(Q_0)$ (``high density''). Both families are comprised of disjoint cubes, because of maximality when choosing them in a stopping time.
	
	It is elementary to check that $\mathrm{HD}$ cubes pack: 
	\begin{multline} \label{eq:packing_HD}
		\sum_{Q \in \mathrm{HD}(Q_0)} \sigma(Q) 
		< 
		\frac{1}{N} \sqrt{\frac{\sigma(Q_0)}{\omega^{\X_1}(Q_0)}} \sum_{Q \in \mathrm{HD}(Q_0)} \iint_Q (\mathcal{M}_\sigma \omega^{\X_1})^{1/2} \, d\sigma
		\\ 
		\leq
		\frac{1}{N} \sqrt{\frac{\sigma(Q_0)}{\omega^{\X_1}(Q_0)}} \iint_{Q_0} (\mathcal{M}_\sigma \omega^{\X_1})^{1/2} \, d\sigma
		\lesssim 
		\frac{1}{N} \sqrt{\frac{\sigma(Q_0)}{\omega^{\X_1}(Q_0)}} \sqrt{\sigma(Q_0)\omega^{\X_1}(Q_0)}
		= 
		\frac{1}{N} \sigma(Q_0),
	\end{multline}
	where we have used the weak-(1, 1) boundedness of the maximal function $\mathcal{M}_\sigma$, which holds since the underlying measure $\sigma$ is doubling because of the p-ADR assumption; see the Appendix for more details on this computation.
	
	On the other hand, $\mathrm{LD}$ cubes naturally satisfy a packing condition with respect to the measure $\omega^{\X_0}$. Indeed, denoting 
	\begin{equation*}
		\mathcal{L}(Q_0)
		:= 
		\bigcup_{Q\in \LD(Q_0)} \!\!\! Q 
		\;\; \subset \pom,
	\end{equation*}
	which is the subset of the boundary covered by these $\LD$ cubes, it holds
	\begin{equation*}
		\omega^{\X_0}(\mathcal{L}(Q_0))
		=
		\sum_{Q \in \mathrm{LD}(Q_0)} \omega^{\X_0}(Q) 
		< 
		\frac{1}{N} \frac{\omega^{\X_0}(Q_0)}{\sigma(Q_0)} \sum_{Q \in \mathrm{LD}(Q_0)} \sigma(Q) 
		\leq 
		\frac{1}{N} \, \omega^{\X_0}(Q_0).
	\end{equation*}
	Concretely, this implies that $\omega^{\X_0}(Q_0 \setminus \mathcal{L}(Q_0)) \geq (1-\varepsilon') \, \omega^{\X_0}(Q_0)$ if $N$ is chosen large enough. Recalling our choice of $\X_0$, we can apply Proposition~\ref{prop:GMT} to deduce the existence of a function $0 \leq f_{Q_0} \lesssim \mathbf{1}_{Q_0 \setminus \mathcal{L}(Q_0)}$ such that $u_{Q_0}(\X) := \iint_\Sigma f_{Q_0} \, d\omega^\X$, $\X \in \Omega$, satisfies 
	\begin{equation} \label{eq:GMT_lemma_applied}
		\iiint_{\mbf{V}_{Q_0}} \abs{\nabla u_{Q_0}}^2 \delta \gtrsim_{\tau, \varepsilon'} \sigma(Q_0).
	\end{equation}


	\textbf{Step 2: iteration to create a full corona decomposition, using p-CME.}
	
	Now fix $Q^0 \in \D$ with $\ell(Q^0) \leq 1$. We will iterate the construction from Step 1 below $Q^0$. Let us start by defining $\Stop_0 := \{Q^0\}$, and inductively
	\begin{equation*}
		\Stop_{k+1} 
		:=
		\bigsqcup_{R \in \Stop_k} \big( \!\HD(R) \cup \LD(R) \big), 
		\qquad k \geq 0,
	\end{equation*}
	where $\HD$ and $\LD$ are defined by the procedure from Step 1 (using $Q_0 = R$).
	We also denote 
	\begin{equation*}
		\Stop := \bigsqcup_{k \geq 0} \Stop_k, 
		\quad \text{and} \quad 
		\Tree(R) := \D(R) \setminus \big(\!\HD(R) \cup \LD(R)\big)
		\quad \text{for $R \in \Stop$},
	\end{equation*}
	where $\D(Q) := \{ Q' \in \D : Q' \subset Q\}$ are the cubes below $Q$.
	This creates a structure of subcoherent trees, rooted at $\HD$ and $\LD$ cubes, that covers all the cubes below $Q^0$. We need to show that the family $\Stop$ satisfies a Carleson packing condition in order for this construction to constitute a meaningful corona decomposition. For that, we follow the original idea of \cite[Section 3]{GMT} (see also \cite[Section 3]{AGMT} or \cite[Section 3.3]{CHM}), relying strongly on Proposition~\ref{prop:GMT} through \eqref{eq:GMT_lemma_applied}. The argument follows closely these references (although we will modify the exposition), and we refer the reader to these papers for further information.
	
	\textbf{Substep 2.1: packing $\LD$ cubes.} The key idea to pack the $\Stop$ cubes is to arrange them into subtrees of $\LD$ cubes rooted at $\HD$ cubes. Indeed, fix any $R \in \Stop$ which was a $\HD$ cube for the previous $\Stop$ generation, i.e., $R \in \Stop_{k+1}$ is contained in some $\widehat{R} \in \Stop_k$ and $R \in \HD(\widehat{R})$. Consider all the $\LD$ cubes that lie below $R$ before encountering any $\HD$ cube:
	\begin{equation*}
		\LD_0(R) := \{R\}, 
		\quad 
		\LD_{k+1}(R) := \bigsqcup_{Q \in \LD_k(R)} \LD(Q) 
		\quad \text{for $k \geq 0$,}
		\quad 
		\LD_{\all}(R) := \bigsqcup_{k \geq 0} \LD_k(R),
	\end{equation*}
	where we are again using the construction of $\LD$ cubes below any given cube from Step 1. As argued in \eqref{eq:GMT_lemma_applied}, it holds 
	\begin{equation} \label{eq:LD_all_sigma}
		\iiint_{\mbf{V}_{Q}} \abs{\nabla u_{Q}}^2 \delta \gtrsim_{\tau, \varepsilon'} \sigma(Q), 
		\qquad 
		\forall \, Q \in \LD_{\all},
	\end{equation}
	where $u_{Q}(\X) := \iint_\Sigma f_{Q} \, d\omega^\X$ for $\X \in \Omega$, and $0 \leq f_{Q} \lesssim \mathbf{1}_{Q \setminus \mathcal{L}(Q)}$, where we recall that we defined $\mathcal{L}(Q) = \bigcup_{Q' \in \LD(Q)} Q'$.
	
	Now let $K \geq 1$, and consider $\{r_j(t)\}_{j \in \N}$ the Rademacher system in $[0, 1)$. For $t \in [0, 1)$, let
	\begin{equation*}
		f_t(\mbf{y})
		:=
		\sum_{k = 1}^K \sum_{Q \in \LD_k(R)} r_{\mathrm{idx}(Q)}(t) \, f_Q(\mbf{y}), 
		\qquad \mbf{y} \in R, 
	\end{equation*}
	where $\mathrm{idx} : \LD_\all(R) \to \N$ is an enumeration of the cubes in $\LD_\all(R)$, and the solutions
	\begin{equation*}
		u_t(\X) := \iint_\Sigma f_t(\mbf{y}) \, d\omega^\X(\mbf{y})
		=
		\sum_{k = 1}^K \sum_{Q \in \LD_k(R)} r_{\mathrm{idx}(Q)}(t) \, u_Q(\X), 
		\qquad \X \in \Omega.
	\end{equation*}
	It is elementary to check that $\{Q \setminus \mathcal{L}(Q)\}_{Q \in \LD_\all(R)}$ are pairwise disjoint\footnote{
		Indeed, assume that $Q, Q' \in \LD_\all(R)$ are distinct cubes satisfying $(Q \setminus \mathcal{L}(Q)) \cap (Q' \setminus \mathcal{L}(Q')) \neq \emptyset$. Concretely, $Q \cap Q' \neq \emptyset$, so without loss of generality $Q' \subsetneq Q$. Then pick $Q'' \in \LD_\all(R)$ maximal such that $Q' \subset Q'' \subsetneq Q$, such that in fact $Q'' \in \LD(Q)$. This means that $Q' \subset Q'' \subset \mathcal{L}(Q)$, which implies that $Q' \cap (Q \setminus \mathcal{L}(Q)) = \emptyset$, which is a contradiction.
	}, which implies that $\norm{f_t}_\infty \lesssim 1$ for every $t \in [0, 1)$, where the constant is independent of $t$ and $K$ (this just comes from the individual bound $0 \leq f_Q \lesssim \mathbf{1}_{Q\setminus \mathcal{L}(Q)}$). Then the maximum principle yields $\norm{u_t}_\infty \lesssim 1$ independently of $t$ and $K$, which implies that we can meaningfully use the p-CME assumption with $u_t$, along with the orthogonality of the Rademacher family and the crucial estimate \eqref{eq:LD_all_sigma}, to obtain
	\begin{align*}
		\sum_{k=1}^K \sum_{Q \in \LD_k(R)} \sigma(Q)
		& \lesssim_{\tau, \varepsilon'}
		\sum_{k=1}^K \sum_{Q \in \LD_k(R)} \iiint_{\mbf{V}_{Q}} \abs{\nabla u_{Q}(\mbf{X})}^2 \delta(\mbf{X}) 
		\\ & \leq 
		\iiint_{\bigcup\limits_{k=1}^K \bigcup\limits_{Q \in \LD_k(R)} \!\!\!\!\! \mbf{V}_Q} \;\;	\sum_{k=1}^K \sum_{Q \in \LD_k(R)} \abs{\nabla u_{Q}(\mbf{X})}^2 \delta(\mbf{X})
		\\ & =
		\iiint_{\bigcup\limits_{k=1}^K \bigcup\limits_{Q \in \LD_k(R)} \!\!\!\!\! \mbf{V}_Q} \;\;	\left( \int_0^1 \abs{\sum_{k=1}^K \sum_{Q \in \LD_k(R)} r_{\mathrm{idx}(Q)}(t) \nabla u_{Q}(\X)}^2 dt \right)\delta (\X) \, d\X
		\\ & =
		\iiint_{\bigcup\limits_{k=1}^K \bigcup\limits_{Q \in \LD_k(R)} \!\!\!\!\! \mbf{V}_Q} \;\;	\left( \int_0^1 \abs{\nabla u_t(\X)}^2 dt \right)\delta (\X) \, d\X
		\\ & \leq 
		\int_0^1 \left( \sum_{k=1}^K \sum_{Q \in \LD_k(R)} \iiint_{\mbf{V}_Q} \abs{\nabla u_t(\X)}^2 \delta (\X) \, d\X \right) \, dt
		\\ & \lesssim_{\tau, \varepsilon'} 
		\int_0^1 \iiint_{\Omega \cap \C_{2C_{\mathrm{cub}}\diam(R)}(\mbf{x}_R)} \abs{\nabla u_t(\X)}^2 \delta (\X) \, d\X \, dt
		\\ & \lesssim 
		\int_0^1 \norm{u_t}_\infty^2 \diam(R)^{n+1} \, dt
		\\ & \lesssim 
		\sigma(R),
	\end{align*}
	where in the third-to-last line we have also used the bounded overlap of the $\mbf{V}_Q$ regions because of their localization property from \eqref{eq:loc_V_Q}\footnote{
		By bounded overlap we mean that, given a fixed $Q \in \LD_\all(R)$, there is only a finite amount (depending on $\tau$ and $\varepsilon'$) of $Q' \in \LD_\all(R)$ such that $\mbf{V}_Q \cap \mbf{V}_{Q'} \neq \emptyset$. Indeed, if that is true, then \eqref{eq:loc_V_Q} and \eqref{eq:hyp_epsilon_1} imply that $\big(\C_{\ell(Q)}(\mbf{x}_Q) \cap \{ \delta \geq c\tau\varepsilon'\ell(Q) \}\big) \cap \big(\C_{\ell(Q')}(\mbf{x}_{Q'}) \cap \{ \delta \geq c\tau\varepsilon'\ell(Q') \}\big) \neq \emptyset$. This clearly implies that $\ell(Q') \approx_{\tau, \varepsilon'} \ell(Q)$ and that $\dist(Q, Q') \lesssim_{\tau, \varepsilon'} \ell(Q)$. The p-ADR property implies that, given $Q$, there is only finitely many (depending on $\tau, \varepsilon'$) cubes $Q'$ satisfying these properties.
	} and also the fact that for any $Q \in \LD_\all(R)$ it holds $\mbf{V}_Q \subset \C_{\ell(Q)}(\mbf{x}_Q) \subset \C_{\ell(R)}(\mbf{x}_Q) \subset \C_{C_{\mathrm{cub}}\diam(R)}(\mbf{x}_Q) \subset \C_{2C_{\mathrm{cub}}\diam(R)}(\mbf{x}_R)$ because of \eqref{eq:loc_V_Q} and $\mbf{x}_Q \in Q \subset R$, and the last estimate is due to the p-ADR. Since constants are uniform on $K$,
	\begin{equation} \label{eq:packing_LD_all}
		\sum_{Q\in \LD_\all(R)} \sigma(Q) 
		\lesssim_{\tau, \varepsilon'} 
		\sigma(R).
	\end{equation}
	
	\begin{figure}[h!]
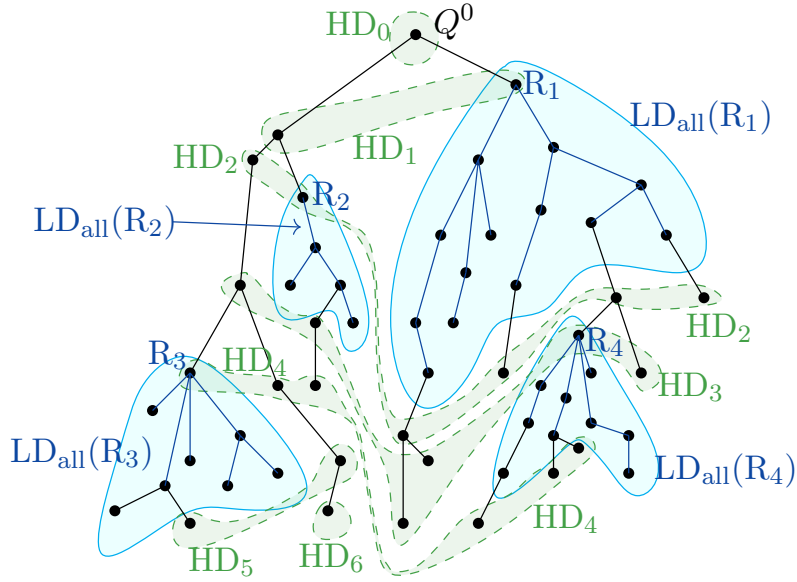

		\centering 
		\resizebox{0.7\textwidth}{!}{
		
		}
		\caption{A sketch of the decomposition of $\D$ carried out during the argument of Step 2. We only draw the cubes below $Q^0$ which belong to the family $\Stop$ (so there are in fact many omitted cubes in the picture: all in $\D \setminus \Stop$): lines represent containment (size decreases as we go down), and blue lines mean that the small cube is a $\LD$ cube of its ancestor, as defined in Step 1 (resp. black lines represent $\HD$). The idea is to create subtrees fully comprised of $\LD$ cubes (the light blue ones) below each $\Stop$ cube (sometimes these subtrees are trivial because there are no $\LD$ cubes directly below the given $\Stop$ cube, so we do not represent them in the picture), which we call $\LD_\all(R)$ for a given $R \in \Stop$, and apply the argument in Substep 2.1 to estimate their size in terms of that of the root $R$. Then, we are left with $\HD$ cubes, which we stratify in layers $\HD_k$ (the green ones; note that these layers may not be at all in correspondence with the dyadic generations in $\D$, but are still stratified as we go down the grid) in order to apply the geometric series argument from Substep 2.2 to pack them in terms of the size of $Q^0$.}
		\label{fig:corona}
	\end{figure} 
	
	\textbf{Substep 2.2: packing $\HD$ cubes.} Denote by $\HD_\all$ all the $\HD$ cubes below $Q^0$, namely 
	\begin{equation*}
		\HD_\all
		:=
		\{Q^0\} \cup
 		\bigsqcup_{k \geq 0} \bigsqcup_{R \in \Stop_k} \HD(R).
	\end{equation*}
	It is easy to check that these constructions yield (see Figure~\ref{fig:corona} for the intuition behind this)
	\begin{equation} \label{eq:stop_decomp}
		\Stop 
		=
		\bigsqcup_{R \in \HD_\all(Q^0)} \LD_\all(R),
	\end{equation}
	that is, all the $\Stop$ cubes are stratified into cubes rooted at some $\HD_\all$ cube, which contain only $\LD$ cubes. Let us stratify the $\HD_\all$ cubes by defining $\HD_0 := \{Q^0\}$ and then
	\begin{equation*}
		\HD_{k+1} := \left\{ R \in \HD_\all : 
		\begin{array}{c}
			\text{there are exactly $k$ cubes $R' \in \HD_\all$}
			\\ \text{satisfying $R \subsetneq R' \subset Q^0$} 
		\end{array}
		\right\}, 
		\quad \text{for $k \geq 0$.}
	\end{equation*}
	Then, it is easy to note that, in fact, it holds 
	\begin{equation*}
		\HD_{k+1} 
		=
		\bigsqcup_{R' \in \HD_k}
		\bigsqcup_{Q \in \LD_\all(R')}
		\HD(Q),
	\end{equation*}
	that is, each cube $R$ in the new $\HD_{k+1}$ generation is in the $\HD$ set (defined in Step 1) of some $\LD_\all$ cube from the previous generation.\footnote{Recall that $\LD_\all(R)$ contains the cube $R$, too, which covers the case that there are no intermediate $\LD$ cubes between $R$ and $R'$.} This implies, along with the fact that the layers $\HD_k$ are comprised of disjoint cubes, \eqref{eq:packing_HD} and \eqref{eq:packing_LD_all}, that
	\begin{multline*}
		\sum_{R \in \HD_{k+1}} \sigma(R)
		=
		\sum_{R' \in \HD_k} \sum_{Q \in \LD_\all(R')} \sum_{R \in \HD(Q)} \sigma(R)
		\leq 
		\frac{C}{N} \sum_{R' \in \HD_k} \sum_{Q \in \LD_\all(R')} \sigma(Q)
		\\ \leq 
		\frac{C'}{N} \sum_{R' \in \HD_k} \sigma(R')
		\leq 
		\frac12 \sum_{R' \in \HD_k} \sigma(R'),
	\end{multline*}
	where we have chosen $N$ large enough (depending on $\tau, \varepsilon'$).
	This allows us to sum a geometric series to obtain, using also \eqref{eq:stop_decomp} and then \eqref{eq:packing_LD_all},
	\begin{equation*}
		\sum_{Q \in \Stop} \sigma(Q)
		=
		\sum_{k \geq 0} \sum_{R \in \HD_k} \sum_{Q \in \LD_\all(R)} \sigma(Q)
		\lesssim
		\sum_{k \geq 0} \sum_{R \in \HD_k} \sigma(R)
		\leq 
		2 \sigma(Q^0).
	\end{equation*} 
	
	\textbf{Substep 2.3: the whole corona.}
	By covering the whole $\Sigma$ by cubes $Q^0$ of the same generation with sidelength smaller than 1, and the arguments from the previous substeps, we have effectively found\footnote{
		This only produces a decomposition of the cubes in $\D$ whose sidelength is smaller than $\ell(Q^0)$, which are the relevant cubes for our problem because they take care of the small scales, so of the fine behavior near the boundary. One can easily complete the corona decomposition (see e.g. \cite[Proposition 2.39]{CHM}).
	} a decomposition $\D = \bigsqcup \S$, where each $\S$ is a semi-coherent stopping-time regime (the so-called $\Tree$ previously), such that the maximal cubes $Q(\S)$ of these subregimes (which correspond to $\Stop$ cubes with the notation above) satisfy 
	\begin{equation*}
		\sum_{Q(\S) \subset \widetilde{Q}} \sigma(Q(\S)) 
		\lesssim 
		\sigma(\widetilde{Q}), 
		\qquad 
		\forall \, \widetilde{Q} \in \D
	\end{equation*}
	because of the estimates just obtained in Substep 2.2. Moreover, by the construction from Step 1 (see the definition of the $\Tree$ at the beginning of Step 2, too), it holds, for every $Q \in \S$,
	\begin{equation*}
		\frac{\omega^{\X_0(\S)}(Q)}{\sigma(Q)} \geq \frac{1}{N} \frac{\omega^{\X_0(\S)}(Q(\S))}{\sigma(Q(\S))}, 
		\;\; \text{and} \;\; 
		\left( \frac{1}{\sigma(Q)} \iint_Q \left( \mathcal{M}_\sigma \omega^{\X_1(\S)}  \right)^{1/2} \, d\sigma \right)^2 \leq N^2 \, \frac{\omega^{\X_1(\S)}(Q(\S))}{\sigma(Q(\S))},
	\end{equation*} 
	where $\X_0(\S), \X_1(\S)$ are the poles associated to $Q(\S)$ by the mechanism of Step 1.

	  
	\textbf{Step 3: removing the pole in the past.} The corona decomposition obtained in Substep 2.3 is not yet fully satisfactory because we have two poles at each regime: $\X_0(\S)$ and $\X_1(\S)$. In order for the corona decomposition to be meaningful and for the constructions in the forthcoming arguments of the paper, we prefer to use $\X_1(\S)$ because it is more at the future. So let us get rid of $\X_0(\S)$. 
	
	Fix $\S$. First, by the interior Harnack inequality (Lemma~\ref{lem:harnack}, use that $\lambda$ is small), it holds 
	\begin{equation} \label{eq:harnack_two_poles}
		\omega^{\X_0(\S)}(Q) \leq C_{\mathrm{H}} \omega^{\X_1(\S)}(Q),
	\end{equation}
	for any $Q \in \S$. Moreover, by our choice of $\X_0$ in Step 1 using Lemma~\ref{lem:bourgain}, it also holds 
	\begin{equation} \label{eq:bourgain_two_poles}
		\omega^{\X_0(\S)}(Q(\S))
		\geq 
		c_{\mathrm{B}} 
		\geq 
		c_{\mathrm{B}} \omega^{\X_1(\S)}(Q(\S)).
	\end{equation}
	Putting these together back in the conclusion of Step 2, we have
	\begin{multline*}
		\frac{c_{\mathrm{B}}}{C_{\mathrm{H}} N} \frac{\omega^{\X_1(\S)}(Q(\S))}{\sigma(Q(\S))}
		\leq 
		\frac{1}{C_{\mathrm{H}} N} \frac{\omega^{\X_0(\S)}(Q(\S))}{\sigma(Q(\S))}
		\leq 
		\frac{1}{C_{\mathrm{H}}} \frac{\omega^{\X_0(\S)}(Q)}{\sigma(Q)}
		\\ \leq 
		\frac{\omega^{\X_1(\S)}(Q)}{\sigma(Q)}
		\leq 
		\left( \bariint_Q \left( \mathcal{M}_\sigma \omega^{\X_1(\S)} \right)^{1/2} \, d\sigma \right)^2 \leq 
		N^2 \, \frac{\omega^{\X_1(\S)}(Q(\S))}{\sigma(Q(\S))}, 
		\qquad \forall \, Q \in \S.	
	\end{multline*}
	 
	On top of the above, for any $Q \in \S$, it holds $\omega^{\X_0(\S)}(Q) > 0$ (see the last display in Step 2). Denoting $t_0(\S)$ and $t_1(\S)$ the time coordinates of $\X_0(\S)$ and $\X_1(\S)$, this implies that $T_{\min}(Q) \leq t_0(\S)$, which yields (recalling also Lemma~\ref{lem:cubes})
	\begin{equation*}
		t_1(\S) - T_{\max}(Q) 
		\geq 
		(t_0(\S) + \lambda \ell(Q(\S))^2) - (T_{\min}(Q) + \diam(Q)^2) 
		\geq 
		\lambda \ell(Q(\S))^2 - \diam(Q)^2.
	\end{equation*}
	This of course resembles (a) in the statement, but (a) is not yet true because $Q$ could be of a comparable size of $Q(\S)$. We need to fix that.
	
	
	\textbf{Step 4: refinement of the corona regimes, so that poles are in the future.} The idea to fix the problem just exposed is simple: within each $\S$, we will \textit{discard} the larger cubes. Concretely, fix $J = J(\varepsilon', \lambda) \gg 1$. Given $\S$, we denote\footnote{$\D_J(Q(\S)) := \{ Q \in \D : Q \subset Q(\S) \text{ and } \ell(Q) = 2^{-J} \ell(Q(\S)) \}$ is the $J$-th generation of descendants of $Q(\S)$.} $\F(\S) := \D_J(Q(\S)) \cap \S$: these will be the new maximal cubes. Every cube in $\S \cap \D_{Q(\S), \F(\S)}$, i.e., contained between $Q(\S)$ and any cube of $\F(\S)$, is put into $\B$: these cubes pack because they are comparable in size to $Q(\S)$ (constants depending on $J$), and the family $\{Q(\S)\}_{\S}$ already packed (Step 2). 
	
	Associated to each $Q_j \in \F(\S)$, we create a new semicoherent regime $\S_j := \S \cap \D(Q_j)$, where we recall that $\D(Q_j) := \{Q' \in \D : Q' \subset Q_j\}$ are the cubes inside $Q_j$. With these new regimes, it is clear that (1) and (2) in the statement hold.
	
	Given any such $\S_j$ (associated to $\S$), we define its pole by $\X_1(\S_j) := \X_1(\S)$, then it still holds (Step 3), simply because $\S_j \subset \S$, that for every $Q \in \S_j$,
	\begin{equation*}
		\frac{c_{\mathrm{B}}}{C_{\mathrm{H}} N} \frac{\omega^{\X_1(\S_j)}(Q(\S))}{\sigma(Q(\S))}
		\leq 
		\frac{\omega^{\X_1(\S_j)}(Q)}{\sigma(Q)}
		\leq 
		\left( \frac{1}{\sigma(Q)} \iint_Q \left( \mathcal{M}_\sigma \omega^{\X_1(\S_j)} \right)^{1/2} \, d\sigma \right)^2 
		\leq 
		N^2 \, \frac{\omega^{\X_1(\S_j)}(Q(\S))}{\sigma(Q(\S))}.
	\end{equation*}

	Note now that $\sigma(Q_j) \approx_J \sigma(Q(\S))$ by the p-ADR condition, and since $Q_j \in \S$ it also holds
	\begin{multline*}
		\omega^{\X_1(\S_j)}(Q(\S_j)) 
		= 
		\omega^{\X_1(\S)}(Q_j) 
		\geq 
		\frac{c_{\mathrm{B}}}{C_{\mathrm{H}} N} \, \frac{\omega^{\X_1(\S)}(Q(\S))}{\sigma(Q(\S))} \sigma(Q_j)
		\geq 
		\frac{c_{\mathrm{B}}}{C(J) C_{\mathrm{H}} N} \, \omega^{\X_1(\S)}(Q(\S))
		\\ \geq 
		\frac{c_{\mathrm{B}}}{C(J) C_{\mathrm{H}}^2 N} \, \omega^{\X_0(\S)}(Q(\S))
		\geq 
		\frac{c_{\mathrm{B}}^2}{C(J) C_{\mathrm{H}}^2 N},
	\end{multline*}
	where we have used \eqref{eq:harnack_two_poles} and \eqref{eq:bourgain_two_poles}. This establishes (b). Moreover, this implies that 
	\begin{equation*}
		\omega^{\X_1(\S_j)}(Q(\S_j)) \approx 1 \approx \omega^{\X_1(\S)}(Q(\S)),
	\end{equation*} 
	which back in the display above implies 
	\begin{equation*}
		\frac{\omega^{\X_1(\S_j)}(Q(\S_j))}{\sigma(Q(\S_j))}
		\lesssim 
		\frac{\omega^{\X_1(\S_j)}(Q)}{\sigma(Q)}
		\leq 
		\left( \iint_Q \mathcal{M}_\sigma \left( \frac{\omega^{\X_1(\S_j)}}{\sigma(Q)} \right)^{1/2} \, d\sigma \right)^2
		\lesssim 
		\frac{\omega^{\X_1(\S_j)}(Q(\S_j))}{\sigma(Q(\S_j))}, 
	\end{equation*}
	for any $Q \in \S_j$, which is (c).
	
	Lastly, let us check (a). As in Step 3, if $Q \in \S_j$,
	\begin{multline*}
		t_1(\S_j) - T_{\max}(Q)
		\geq 
		\lambda \ell(Q(\S))^2 - \diam(Q)^2
		\\ \geq 
		\lambda (C_{\mathrm{cub}}^{-1} 2^J \diam(Q(\S_j)))^2 - \diam(Q(\S_j))^2
		\geq 
		(500 \diam(Q(\S_j)))^2, 
	\end{multline*}
	simply by choosing any $\lambda > 0$ small (recall the use of Harnack's inequality in Step 3), and then $J$ large enough. Moreover, 
	\begin{align*}
		\delta(\X_1(\S)) 
		& \geq 
		\delta(\X_0(\S)) - \norm{\X_1(\S) - \X_0(\S)}
		\geq 
		c \varepsilon' \ell(Q(\S)) - \sqrt{\lambda} \ell(Q(\S))
		\\ & \geq 
		c \varepsilon' 2^J \diam(Q(\S_j)) - C_{\mathrm{cub}} 2^J \sqrt{\lambda} \diam(Q(\S_j))
		\\ & \geq 
		\frac{c}{2} \varepsilon' 2^J \diam(Q(\S_j))
		\geq 
		100 \diam(Q(\S_j)),
	\end{align*}
	again by simply choosing $\lambda$ small enough (depending on $\varepsilon'$), and later $J$ large enough (depending on $\lambda$ and $\varepsilon'$). This establishes (a), hence finishing the proof.
\end{proof}



\section{Construction of approximating graphical subdomains} \label{sec:WHSA}

By Theorem~\ref{th:corona}, the p-CME assumption yields a corona decomposition for caloric measure. Thus, to complete the proof of Theorem~\ref{th:main}, it remains to show that such a corona decomposition forces $\Sigma$ to be parabolic uniformly rectifiable.

As a first step, we construct interior approximating parabolic Lipschitz graphs at many scales (more precisely, in the sense of the corona decomposition). The one-sided nature of this approximation, namely that the graphs lie strictly inside $\Omega$ after truncation to the relevant scale (see \eqref{eq:containment}), is essential for the arguments of Section~\ref{sec:pushing_CME}. This construction is the parabolic analogue of the elliptic weak half-space approximation (WHSA). Unlike in the elliptic setting, however, the existence of such approximating graphs does not by itself imply parabolic uniform rectifiability, see \cite[Proposition 1.17]{HLMN} and \cite[Observation 4.19]{BHHLN_22a}. Rather, the graphs must also satisfy the regularity condition from Section~\ref{subsec:UR}, which will be established in the next section.

The construction in this section follows closely the approach of Martell, Nystr\"om and the first and third authors \cite{BHMN}. Although our corona decomposition is formulated slightly differently, the arguments of \cite[Sections~6--9 and 11]{BHMN} adapt with only minor modifications\footnote{
	Indeed, the difference of our corona decomposition with respect to that in \cite{BHMN} is that our maximal function term is measured in $L^{1/2}$ instead of $L^1$. Nevertheless, one can check that the only place in which this affects the arguments of \cite{BHMN} is precisely in (6.33), and in that computation, one can easily introduce $L^{1/2}$ averages instead of $L^1$ ones.
	\\ 
	Moreover, we are changing the background assumption of \cite{BHMN}, namely TSADR (time-symmetric ADR), by TSCDC and p-ADR. The impact of this is that the PDE estimates from \cite[Section 5]{BHMN} have to be verified under this new set of hypotheses, but we already did that in Lemmas~\ref{lem:bourgain} and \ref{lem:holder}, and Remark~\ref{rem:estimates_BHMN}. Moreover, in \cite[Section 7.4]{BHMN}, the authors invoke \cite[Lemma 2.6]{BHMN}, but we already took care of it in Lemma~\ref{lem:2.6BHMN}. The rest of \cite[Sections 6-9]{BHMN} only needs the non-time-directed version of p-ADR.
} to give the main result of this section, which is the following coronization by interior Lip(1,1/2) subdomains. In \cite{BHMN}, this does not appear as an isolated result, but rather as a long construction, so we include enough citations for the reader to follow the argument.

\begin{proposition}[Approximation by very flat Lipschitz subdomains] \label{prop:WHSA}
	Let $\Omega \subset \ree$ be an open set satisfying the TSCDC (see Definition~\ref{def:TBCDC}) so that $\Sigma := \pom$ is p-ADR (see Definition~\ref{def:ADR}). Assume also that caloric measure admits a corona decomposition (see Definition~\ref{def:corona}). Then, there exist $K_0, C_{\mathrm{LA}} > 0$ ($\mathrm{LA} =$ ``Lipschitz approximation'') large enough,
	\begin{equation} \label{eq:epsilon_K_0}
		0 < \varepsilon \leq K_0^{-100}
	\end{equation} 
	small enough (for this non-optimal, yet sufficient, bound, see \cite[(6.34)]{BHMN}), all depending only on $n, C_{\mathrm{ADR}}, C_{\mathrm{CDC}}, C_{\mathrm{cor}}$, such that the following holds.	
	
	Fix a stopping time regime $\S \in \mathcal{S}$ provided by the corona decomposition for the caloric measure. 
	Then $\S$ can be divided into a collection of disjoint subregimes $\S^* \in \mathcal{S}^*(\S)$, the so-called ``graph trees'' ($\S^*$ denotes a single graph tree, and $\mathcal{S}^*(\S)$ is the set of all graph trees contained in the tree $\S$), up to some new bad cubes $\B_\S$, that is, $\S = \left(\bigsqcup_{\mathcal{S}^*(\S)} \S^*\right) \sqcup \B_\S$ (see Lemma 6.45 and Proposition 8.2 in \cite{BHMN}), so that it holds:
	\begin{itemize}
		
		\item For any $\S^* \in \mathcal{S}^*(\S)$, it holds 
		\begin{equation} \label{eq:tops_small}
			\ell(Q(\S^*)) \leq \varepsilon^{10} \ell(Q(\S)).
		\end{equation}
		See pp. 60 and 61 in \cite{BHMN}.
		
		\item If we consider the following normalized version of the adjoint Green function in $\Omega$ (see Definition~\ref{def:green}, and Theorem~\ref{th:corona} for the choice of the pole $\mbf{Y}_\S = (Y_\S, s_\S)$)
		\begin{equation} \label{eq:def_hatu}
			\widehat{u}(Y, s)
			:=
			\frac{\sigma(Q(\S))}{\omega^{\mbf{Y}_\S}(Q(\S))} \widehat{G}_\Omega (Y_\S, s_\S; Y, s), 
			\qquad (Y, s) \in \Omega.
 		\end{equation}
		then, for each $Q \in \S^* \in \mathcal{S}^*(\S)$, there exists $(Y^*_Q, s^*_Q) \in \mbf{U}_Q^\good$ with
		\begin{equation} \label{eq:good_gradient}
			C_{\mathrm{LA}}^{-1} \leq \abs{\nabla \widehat{u}(Y^*_Q, s^*_Q)} \leq C_{\mathrm{LA}},
		\end{equation}
		where $\mbf{U}_Q^\good$ is some component of the Whitney region associated to $Q$ (see Subsection~\ref{subsec:whitney}). Moreover, 
		it also holds 
		\begin{equation} \label{eq:upper_bound_u_enlarged}
			\delta_\Omega(\mbf{X}) \lesssim_\varepsilon \widehat{u} (\mbf{X}), 
			\qquad 
			\forall \, \mbf{X} \in \widetilde{\mbf{U}}_Q^\good,
		\end{equation}
		where $\widetilde{\mbf{U}}_Q^\good$ is the component of the enlarged Whitney region associated to $Q$ constructed from $\mbf{U}_Q^\good$, which contains $(Y^*_Q, s^*_Q)$ (see Subsection~\ref{subsec:whitney}).
		See (6.28), (6.54) (also (8.6)) and Corollary 6.52 in \cite{BHMN}.
		
		\item For every fixed $\S^*$:  
		\begin{itemize}[label=$\circ$]
			\item Let us denote 
			\begin{equation*}
				\kappa := C_{\mathrm{LA}} K_0, 
				\qquad \text{and} \qquad 
				R_{\S^*} := \diam(Q(\S^*)).
			\end{equation*}
			See (9.11) and (9.22) in \cite{BHMN}.
			
			\item Up to rotation only in the spatial variables, there exists a function $\psi_{\S^*} : \Rn \to \R$, whose graph is
			\begin{equation*}
				\Gamma_{\S^*} 
				:= 
				\big\{ (\psi_{\S^*}(x, t), x, t) : \; (x, t) \in \R^{n-1} \times \R \big\}.
			\end{equation*}
			We denote the projection towards the hyperplane defining $\psi_{\S^*}$ by 
			\begin{equation*}
				\pi: \R^{n+1} \longrightarrow \Rn, \qquad \pi(x_0, x, t) = (x, t).
			\end{equation*}
			See pp. 66, 70 and 78 in \cite{BHMN}.
			
			\item  
			If we define the stopping-time distance
			\begin{equation*}
				d(X, t) := d_{\S^*}(X, t) :=
				\inf_{Q \in \S^*} \big[\!\dist((X, t), Q) + \diam(Q) \, \big],
				\quad 
				(X, t) \in \R^{n+1}, 
			\end{equation*} 
			then it holds $\Gamma_{\S^*} \equiv \Sigma$ over\footnote{
				Indeed, $\psi_{\S^*}$ is defined to agree with $\Sigma$ over $F$, and then extended outside. See \cite[(9.30)]{BHMN}.
			}
			\begin{equation*}
				F := F_{\S^*} := \{ (X, t) \in \Sigma : \; d(X, t) = 0 \}.
			\end{equation*}
			See (9.6), (9.8), Lemma 9.12 and (9.30) in \cite{BHMN}.
			
			\item Moreover, we also define a projected stopping-time distance
			\begin{equation*}
				D(x, t)
				:=
				\inf_{x_0 \in \R} d(x_0, x, t)
				=
				\inf_{Q \in \S^*} \big[\! \dist((x, t), \pi(Q)) + \diam(Q) \, \big],
				\quad 
				(x, t) \in \Rn.
			\end{equation*}
			See \cite[(9.7)]{BHMN}.
		
			\item The set $\Rn \setminus \pi(F)$, where $\psi_{\S^*}$ does not agree with $\Sigma$, is partitioned in dyadic cubes\footnote{These are actually the usual dyadic cubes in $\R^n$, not the ones adapted to the geometry of $\Sigma$ from Lemma~\ref{lem:cubes}.} $\{I_i\}_{i \in \mathcal{I}}$, for which we have
			\begin{equation} \label{eq:diam_I_i}
				10 \diam(I_i) 
				\leq 
				D(x, t)
				\leq 
				60 \diam(I_i),
				\qquad 
				(x, t) \in 10I_i.
			\end{equation}
			We also denote 
			\begin{equation} \label{eq:def_Lambda}
				\Lambda 
				:=
				\big\{ i \in \mathcal{I} : I_i \cap C'_{2\kappa R_{\S^*}}(x_{Q(\S^*)}, t_{Q(\S^*)}) \neq \emptyset \big\}.
			\end{equation}
			See Section 9.2, Lemma 9.19 and (9.23) in \cite{BHMN}.
			
			\item The graph $\psi_{\S^*}$ is localized to the scale of $\S^*$, in the sense that 
			\begin{equation} \label{eq:psi_zero_far}
				\psi_{\S^*} \equiv 0 \qquad \text{ in $\Rn \setminus C'_{4\kappa R_{\S^*}}(x_{Q(\S^*)}, t_{Q(\S^*)})$}.
			\end{equation}
			See \cite[(9.31)]{BHMN}.
		
			\item The graph $\psi_{\S^*}$ is $\Lip(1,1/2)$, and actually very flat: 
			\begin{equation} \label{eq:psi_lipschitz}
				\norm{\psi_{\S^*}}_{\Lip(1,1/2)} \leq C_{\mathrm{LA}} \varepsilon^{1/2}.
			\end{equation}
			See \cite[Lemma 9.72]{BHMN}.
		
			\item The graph of $\psi_{\S^*}$ approximates the boundary $\Sigma$ at the scales of $\S^*$, namely 
			\begin{equation} \label{eq:Gamma_approximates} 
				\sup_{(X, t) \in 2Q} \dist((X, t), \Gamma_{\S^*}) \leq C_{\mathrm{LA}} K_0 \diam(Q), 
				\qquad 
				\forall \, Q \in \S^*.
			\end{equation}
			See \cite[Lemma 9.76]{BHMN}.
			
			\item Reciprocally, the approximation is very tight at the scales of the $I_i$ ($i \in \mathcal{I}$), namely 
			\begin{equation} \label{eq:approx_I_i}
				\delta_\Omega(\psi_{\S^*}(x, t), x, t) 
				\leq 
				\varepsilon^{1/2} \diam(I_i), 
				\qquad 
				\forall \, (x, t) \in 10I_i, \; i \in \Lambda.
			\end{equation}
			See \cite[Lemma 9.63]{BHMN}.
		
			\item Define the domain above the graph of $\psi_{\S^*}$ by 
			\begin{equation*}
				\Omega_{\S^*} := \big\{ (x_0, x, t) \in \R \times \R^{n-1} \times \R : x_0 > \psi_{\S^*}(x, t) \big\}.
			\end{equation*} 
			Moreover, let
			$(x_{Q(\S^*)}, t_{Q(\S^*)}) = \pi(\mbf{x}_{Q(\S^*)})$ be the projection (i.e. removing the first coordinate) of the center of $Q(\S^*)$. Define the following truncation of $\Omega_{\S^*}$:
			\begin{align} \label{def:subdomain_truncated}
				\Omega'_{\S^*} := \Big \{ (x_0, x, t) \in \R \times \R^{n-1} \times \R : \; & \;\; \psi_{\S^*}(x, t) < x_0 < \varepsilon^{-1/4} \ell(Q(\S^*)), 
				\\ & \text{ and } (x, t) \in C'_{\kappa R_{\S^*}}(x_{Q(\S^*)}, t_{Q(\S^*)}) \Big \}.
				\nonumber
			\end{align}
			Then it holds 
			\begin{equation} \label{eq:containment}
				\Omega'_{\S^*} \subset \bigcup_{Q \in \S^*} \widetilde{\mbf{U}}_Q^{\good} \subset \Omega.
			\end{equation}
			See (9.78), (9.79) and Lemma 9.81 in \cite{BHMN}.
		\end{itemize} 
		\item  Moreover, {\bf if the graphs $\Gamma_{\S^*}$ are regular} (with uniform constants) then the bad and maximal cubes of the graph trees pack, namely 
		\begin{equation*}
			\sum_{\substack{\S^* \in \mathcal{S}^*(\S) \\ Q(\S^*) \subset R}} \sigma(Q(\S^*))
			+ \sum_{\substack{Q \in \B_{\S} \\ Q \subset R}} \sigma(Q)
			\lesssim 
			\sigma(R), 
			\qquad
			\forall R \in \D.
		\end{equation*}
		See Section 11 in \cite{BHMN}.\footnote{
			As explained in \cite[Remark 8.5]{BHMN}, the packing condition in that paper is stated in Lemma 8.20, and proved in Section 11. The proof in \cite[Section 11]{BHMN} relies strongly on the regularity of the approximating graphs, which is shown in \cite[Section 10]{BHMN} (concretely, Proposition 10.4). Although our strategy to show regularity of the graphs will be different than that of \cite{BHMN}, because we need to use the p-CME (or the corona for $\omega$) assumption instead of the $A_\infty$ condition from \cite{BHMN} (see Theorem~\ref{th:push_cme}), it is easy to check that the arguments of \cite[Section 11]{BHMN} work well as a \textit{black box} to show the packing condition under the regularity of the graphs. Indeed, the $A_\infty$ assumption is never used directly in \cite[Section 11]{BHMN}, but only through the regularity of the approximating graphs, and this allows us to import the result to our setting.
		}
	\end{itemize}
\end{proposition}

It is important to note that the last item is conditioned on the fact that the graphs $\Gamma_{\S^*}$ are regular. Indeed, it is known in this setting that the local $N \lesssim S$ (non-tangential maximal function bounded by square function) estimates hold in the domain \textit{above} the graph $\Gamma_{\S^*}$. Roughly speaking, in \cite{BHMN} the corona decomposition yields a local square function estimate in each caloric measure tree $\S$ involving $|\nabla^2 \widehat{u}|$ so that the small oscillation of the gradient in appropriate Whitney regions as in \cite[(8.23)]{BHMN}, along with the local $N \lesssim S$ estimates guarantee that the graph trees inside of $\S$ contribute to the (finite) local square function estimate. In truth, the graph trees are broken up into different cases, and this argument (using the small oscillations of the gradient of $\widehat{u}$ along with the local $N \lesssim S$ estimate) is used to handle the difficult case.

A standard consequence of the corona decomposition for the caloric measure $\omega$ from Theorem~\ref{th:corona} is the fact that one can get upper bounds for Green functions in the Whitney regions associated to the good cubes in the corona decomposition (lower bounds as in \eqref{eq:good_gradient} are always trickier, and that is why they are only obtained in \textit{some} regions, not everywhere). In the next lemma, which does not rely on Proposition~\ref{prop:WHSA}, we obtain these bounds for the adjoint Green function that we have used to construct the approximating graphs. We remark that since $\widehat{u}$ is an adjoint solution, we need to use the TFCDC in order to obtain meaningful estimates.

\begin{lemma}[Interior estimates for Green's function] \label{lem:u_upper_bounds}
	Let $\Omega \subset \ree$ be an open set satisfying the TFCDC so that $\Sigma := \pom$ is p-ADR. Assume also that the caloric measure in $\Omega$ admits a corona decomposition.
	Fix a stopping time regime $\S$ provided by the corona decomposition, and let $\varepsilon > 0$ small enough (from Proposition~\ref{prop:WHSA}). If $Q \in \S$ satisfies $\ell(Q) \leq \varepsilon^{10} \ell(Q(\S))$, then the normalized adjoint Green function defined in \eqref{eq:def_hatu} satisfies the upper bounds
	\begin{equation*}
		\abs{\nabla \widehat{u} (\mbf{Y})} + \frac{\widehat{u}(\mbf{Y})}{\delta_\Omega(\mbf{Y})} 
		\leq 
		C_{\mathrm{int}}, 
		\qquad 
		\forall \, \mbf{Y} \in \widetilde{\mbf{U}}_Q,
	\end{equation*}
	where $C_{\mathrm{int}}$ depends on $\varepsilon$ (along with $n, C_{\mathrm{ADR}}, C_{\mathrm{CDC}}, C_{\mathrm{cor}}$). Concretely, the estimate is true for any $Q \in \S* \in \mathcal{S}^*(\S)$ from Proposition~\ref{prop:WHSA} because of \eqref{eq:tops_small}. 
\end{lemma} 

These upper bounds are proved in \cite[Lemma 6.85]{BHMN}\footnote{
	Actually, let us check that the set $\widetilde{\mbf{U}}_Q$ is contained in $\Omega_{\mathcal{F}, Q_0}^{***}$ (this is notation in \cite{BHMN}, $Q_0$ in \cite{BHMN} plays the role of $Q(\S)$ in our text), which is the set where the estimates are proved in \cite[Lemma 6.85]{BHMN}. Indeed, unraveling the definitions in pp. 40, 45 and 46 in \cite{BHMN}, one ends up checking that $\bigcup_{Q \in \D_{\mathcal{F}, Q_0}^*} \widetilde{\mbf{U}}_Q \subset \Omega_{\mathcal{F}, Q_0}^{***}$, where $\D_{\mathcal{F}, Q_0}^* = \{ Q \in \S : \ell(Q) \leq \varepsilon^{10} \ell(Q(\S)) \}$ by \cite[(6.35)]{BHMN}, which already gives the result. 
}. Indeed, the proof is standard in the area. To prove the upper bound $\widehat{u} \lesssim \delta_\Omega$ in the simpler case that $\mbf{Y} \in \mbf{U}_Q$ (or rather, a slightly fattened version of it, to later be able to apply standard interior estimates to get the gradient bound), we proceed as in \cite[Lemma 6.53]{BHMN}: first use the CFMS-like upper bound from \cite[Lemma 5.20]{BHMN} (which is applicable under the TFCDC and p-ADR assumptions, as explained in Remark~\ref{rem:estimates_BHMN}) to control $\widehat{u}/\delta_\Omega$ by the caloric measure of corresponding surface balls, and then control those using the corona decomposition from Theorem~\ref{th:corona}. To extend the proof for the enlarged regions $\mbf{Y} \in \widetilde{\mbf{U}}_Q$, one needs to work a bit more because of the cumbersome definition of $\widetilde{\mbf{U}}_Q$ (see Subsection~\ref{subsec:whitney}), but the idea is the same as before, using the corona decomposition for the caloric measure with suitable dilations of $Q$ (indeed, it suffices to use cubes of sidelength $\varepsilon^{-5}\ell(Q)$, which are still within $\S$ because $\ell(Q) \leq \varepsilon^{10} \ell(Q(\S))$). For full details, see \cite[Lemma 6.85]{BHMN}.



\section{p-CME holds in the approximating graphical subdomains} \label{sec:pushing_CME}

In the previous section, we showed that a corona decomposition for $\omega$ implies the existence of interior approximating parabolic Lipschitz subdomains. By Theorem~\ref{th:corona_implies_UR}, it remains to prove that the boundaries of these subdomains, which are Lipschitz graphs, are regular. The key step of this section is to establish that the p-CME property is inherited by these approximating subdomains. Once this is achieved, the recent characterization of \cite{HW}, together with \cite{BHMN_graph}, immediately implies that the approximating graphs are regular.

\begin{theorem}[p-CME holds in the approximating subdomains] \label{th:push_cme}
	Let $\Omega \subset \ree$ be an open set satisfying the TSCDC (see Definition~\ref{def:TBCDC}) so that $\Sigma := \pom$ is p-ADR (see Definition~\ref{def:ADR}). Assume also that $\omega$ admits a corona decomposition (see Definition~\ref{def:corona}). 
	Fix a stopping time regime $\S$ provided by Definition~\ref{def:corona}, and a ``graph tree'' $\S^* \in \mathcal{S}*(\S)$ associated to it as in Proposition~\ref{prop:WHSA}. Then, p-CME holds in $\Omega_{\S^*}$, that is, for every $\overline{\mbf{y}} \in \partial \Omega_{\S^*}$, $r > 0$, and every bounded caloric function $u$ in $\Omega_{\S^*}$, it holds 
	\begin{equation} \label{goal:cme}
		\iiint_{\Omega_{\S^*} \cap \mbf{C}_r(\overline{\mbf{y}})} \abs{\nabla u(\mbf{X})}^2 \delta_{\Omega_{\S^*}}(\mbf{X}) \, d\mbf{X}
		\lesssim 
		\norm{u}_{L^\infty(\Omega_{\S^*})}^2 r^{n+1},
	\end{equation}
	with implicit constant depending only on $n, C_{\mathrm{ADR}}, C_{\mathrm{CDC}}$ and $C_{\mathrm{cor}}$.
	
	In particular, since $\Omega_{\S^*}$ are graphical domains with $\Lip(1,1/2)$ boundary (see Proposition~\ref{prop:WHSA}), \cite[Theorem 4.1]{HW} and \cite[Theorem 1.1]{BHMN_graph} imply that the graphs $\psi_{\S^*}$ are \textbf{regular} $\Lip(1,1/2)$, or equivalently, that $\partial \Omega_{\S^*}$ is parabolic uniformly rectifiable (see Subsection~\ref{subsec:UR}).
\end{theorem}

At a high level, the proof combines two complementary arguments: local estimates near the boundary, obtained from the regularity of the approximating graphs; and a global integration by parts argument in the interior, which exploits the stopping-time geometry and the Green function estimates coming from the corona decomposition. More precisely:
\begin{itemize}
	\item We first estimate the contribution to \eqref{goal:cme} coming from points close to $\partial\Omega_{\S^*}$. The boundary naturally decomposes into the contact set $\pi(F)$ (recall Proposition~\ref{prop:WHSA}), where $\partial\Omega_{\S^*} = \Gamma_{\S^*}$ coincides with $\partial\Omega = \Sigma$; and its complement, where the boundary is locally given by parabolic Lipschitz graphs above Whitney cubes $I_i$. Near the contact set, the corona decomposition allows us to compare caloric measure and surface measure, yielding the desired Carleson measure estimates.	Away from the contact set, the boundary is locally regular (Lemma~\ref{lem:DDH}),	so we may instead invoke the local implication p-UR $\implies$ p-CME from Lemma~\ref{lem:ur_implies_cme}.
	
	\item Then, we are left with estimating the part of $\Omega_{\S^*} \cap \C_r(\overline{\mbf{y}})$ in \eqref{goal:cme} which is outside of regions surrounding $\Gamma_{\S^*} = \partial \Omega_{\S^*}$. In particular, we will work inside a sawtooth region of the original domain $\Omega$ associated to $\S^*$, that is, one only needs to consider Whitney regions of $\Omega$ associated to cubes in $\S^*$ (this is the purpose of the stopping time distances).  Moreover, inside these Whitney regions, upper bounds for Green's functions (namely Lemma~\ref{lem:u_upper_bounds}) hold because of Theorem~\ref{th:corona} (and $\S^* \subset \S$), which will allow us to carry out an integration by parts scheme to finish the proof of \eqref{goal:cme}.
\end{itemize}

As discussed above, away from the contact set $\pi(F)$ from Proposition~\ref{prop:WHSA}, the graph $\psi_{\S^*}$ enjoys additional regularity. More precisely, at the scales of the Whitney decomposition from Proposition~\ref{prop:WHSA}, it agrees locally with regular Lip(1,1/2) graphs.

\begin{lemma}[{\cite[Lemma 10.3]{BHMN}}] \label{lem:DDH} 
	Under the hypotheses of Theorem~\ref{th:push_cme}, there exists $k_0 \in \N$ and $b$, depending only on $n$, such that the following holds: for any $I_i$ provided by Proposition~\ref{prop:WHSA}, if $I$ is a dyadic subcube of $I_i$ with sidelength $\ell(I) = 2^{-k_0}\ell(I_i)$, then $200n I \subset 3I_i$ and there exists a \textbf{regular} $\Lip(1,1/2)$ function $\psi_I$, with constants $b_1 + b_2 \leq b$, such that $\psi_I = \psi_{\S^*}$ on $100n I$. 
\end{lemma}

The second ingredient is a local version of the implication p-UR $\implies$ p-CME. Roughly speaking, it asserts that bounded caloric functions satisfy Carleson measure estimates in sufficiently small cylinders of regular Lip(1,1/2) domains. This local estimate ultimately goes back to \cite{HL05}.

\begin{lemma}[{p-UR implies p-CME, local version, \cite[Lemma 4.1]{BHHLN}}] \label{lem:ur_implies_cme}
	Let $\psi: \R^n \to\R$ be a regular $\Lip(1,1/2)$ function, with constants bounded by $b$, and let 
	\begin{equation*}
		\mathcal{U}
		:=
		\big\{ (x_0, x, t) \in \R \times \R^{n-1} \times \R : \; x_0 > \psi(x, t) \big\}.
	\end{equation*}
	Then there exist $M_0, C > 0$, depending only on $n$ and $b$, such that if $u$ is a bounded solution to $\partial_t u - \Delta u = 0$ in $\mathcal{U} \cap \mbf{C}_{M_0 r}(\mbf{x})$, for any $\mbf{x} \in \partial \mathcal{U}$ and $r > 0$, then it holds 
	\begin{equation*}
		\iiint_{\mbf{C}_r(\mbf{x}) \cap \mathcal{U}} \abs{\nabla u(\mbf{X})}^2 \delta_{\mathcal{U}}(\mbf{X}) \, d\mbf{X}
		\leq 
		C \norm{u}_{L^\infty(\mbf{C}_{M_0r}(\mbf{x}) \cap \mathcal{U})}^2 r^{n+1}.
	\end{equation*}
\end{lemma}


The rest of the section is devoted to the \textbf{proof of Theorem~\ref{th:push_cme}}. Fix $(\overline{y}, \overline{s}) \in \Rn$ and $r > 0$, and denote 
\begin{equation*}
	\overline{\mbf{y}} := (\psi_{\S^*}(\overline{y}, \overline{s}), \overline{y}, \overline{s}).
\end{equation*} 
This clearly plays the role of a generic point of $\partial \Omega_{\S^*}$, since $\Omega_{\S^*}$ is a graph domain.
Fix also $u$ a bounded caloric function in $\Omega_{\S^*}$, i.e. $H u = (\partial_t - \Delta) u = 0$ in $\Omega_{\S^*}$. Our goal is to show \eqref{goal:cme}.

To fix ideas, let us stay localized inside the region where $\psi_{\S^*}$ is nontrivial (see \eqref{eq:psi_zero_far}). Indeed, let us assume that 
\begin{equation} \label{eq:assumption_local}
	C'_r(\overline{y}, \overline{s}) \subset C'_{\frac12 \kappa R_{\S^*}}(x_{Q(\S^*)}, t_{Q(\S^*)}),
\end{equation}
and we will remove this assumption in Subsection~\ref{subsec:gral_case}.
Here, we recall that $(x_{Q(\S^*)}, t_{Q(\S^*)})$ is the projection $ \pi(\mbf{x}_{Q(\S^*)})$ (i.e. removing the first coordinate) of the ``center'' of $Q(\S^*)$. Recall also that $\kappa = C_{\mathrm{LA}} K_0$ and $R_{\S^*} = \diam(Q(\S^*))$ were defined in Proposition~\ref{prop:WHSA}.

Moreover, fix $0 < \mu < 1/2$ small enough, whose value will be determined later depending only on $n$. We divide the proof in two cases:
\begin{itemize}
	\item Case 1: $r \geq \mu D(\overline{y}, \overline{s})$.
	\item Case 2: $r < \mu D(\overline{y}, \overline{s})$.
\end{itemize}


\subsection{Proof in Case 1} 

Let us first illustrate how to deal with Case 1, which incorporates the key new ideas. Later, in Subsection~\ref{subsec:case_2} we will explain how Case 2 essentially reduces to Step 1 of Case 1, which we introduce now.

\textbf{Step 1: p-CME holds in strips surrounding $\partial \Omega_{\S^*}$.} First we deal with the fraction of $\Omega_{\S^*} \cap \C_r(\overline{\mbf{y}})$ lying close to $\partial \Omega_{\S^*}$. We will use Lemmas~\ref{lem:DDH} and \ref{lem:ur_implies_cme} to show that p-CME holds in small balls centered at $\partial \Omega_{\S^*}$, and these balls together create a neighborhood of $\partial \Omega_{\S^*}$ within $\Omega_{\S^*}$ (our so-called \textit{strips}). Thus, we will be left to show the p-CME estimates only above these strips. It will actually be beneficial that these strips are graphical, and we will enforce this during the construction.

\textbf{Substep 1.1: p-CME is true in small balls over the graph of $\psi_{\S^*}$.}
Fix any $I_i$ for $i \in \mathcal{I}$ from Proposition~\ref{prop:WHSA}. Using $k_0$ from Lemma~\ref{lem:DDH}, cover $I_i$ by dyadic cubes $\{I_{i,j}\}_{j \in \mathcal{J}_i}$ in $\R^n$ of sidelength $\ell(I_{i,j}) = 2^{-k_0}\ell(I_i)$. Note that, since $k_0 = k_0(n)$, it holds, for some $C_0 = C_0(n) > 0$ independent of $i$,
\begin{equation} \label{eq:amount_j}
	\abs{\mathcal{J}_i}
	\leq 
	C_0.
\end{equation} 
Moreover, for each $j \in \mathcal{J}_i$, by Lemma~\ref{lem:DDH}, there exists some regular Lip(1,1/2) function $\psi_{i, j}$ which agrees with $\psi_{\S^*}$ in $100nI_{i, j}$. Let us denote 
\begin{equation*}
	\Omega_{i, j} := \big\{(x_0, x, t) \in \R \times \R^{n-1} \times \R: x_0 > \psi_{i, j}(x, t)\big\}, 	
\end{equation*}
and $\mbf{x}_{i, j} := (\psi_{i, j}(x_{i, j}, t_{i, j}), x_{i, j}, t_{i, j}) \in \partial \Omega_{i, j}$, the point whose projection $\pi(\mbf{x}_{i, j}) = (x_{i, j}, t_{i, j})$ is the center of $I_{i, j}$. We note that, since $\psi_{i, j} = \psi_{\S^*}$ in $100nI_{i, j}$, it holds\footnote{
Given a Whitney box $I \subset \R^n = \R^{n-1} \times \R$, we define its dilation by a factor of $\lambda > 1$ as follows: $\lambda I := \{ (x, t) \in \R^{n-1} \times \R : \dist((x, t), I) < (\lambda - 1) \diam(I) \}$. In this way, it is clear that $C'_{30n\diam(I_{i,j})}(x_{i,j}, t_{i,j})$ definitely lies inside $100nI_{i,j}$.
}
\begin{equation*}
	\Omega_{i, j} \cap \mbf{C}_{30n\diam(I_{i, j})}(\mbf{x}_{i, j})
	=
	\Omega_{\S^*} \cap \mbf{C}_{30n\diam(I_{i, j})}(\mbf{x}_{i, j})
\end{equation*}
which implies
\begin{equation*}
	\delta_{\Omega_{i, j}}(\mbf{X})
	=
	\delta_{\Omega_{\S^*}}(\mbf{X}), 
	\qquad 
	\forall \, \mbf{X} \in \Omega_{\S^*} \cap \mbf{C}_{10n\diam(I_{i, j})}(\mbf{x}_{i, j}).
\end{equation*}
Indeed, let $\widehat{\mbf{x}}_{i, j} \in \pom_{i, j}$ satisfy $\norm{\mbf{X} - \widehat{\mbf{x}}_{i, j}} = \delta_{\Omega_{i, j}}(\mbf{X})$. Then the fact that $\mbf{x}_{i, j} \in \partial \Omega_{i, j}$ implies $\norm{\mbf{X} - \widehat{\mbf{x}}_{i, j}} \leq \norm{\mbf{X} - \mbf{x}_{i, j}} \leq 10n\diam(I_{i, j})$, and $\mbf{X} \in \C_{10n\diam(I_{i, j})}(\mbf{x}_{i, j})$ then yields $\widehat{\mbf{x}}_{i, j} \in \mbf{C}_{20n\diam(I_{i, j})}(\mbf{x}_{i, j})$. Inside this cylinder, $\Omega_{i, j}$ and $\Omega_{\S^*}$ coincide, so $\widehat{\mbf{x}}_{i, j} \in \pom_{\S^*}$, whence $\delta_{\Omega_{\S^*}}(\X) \leq \norm{\mbf{X} - \widehat{\mbf{x}}_{i, j}} = \delta_{\Omega_{i, j}}(\mbf{X})$. The $\geq$ is exactly the same.

Then, using Lemma~\ref{lem:ur_implies_cme} (because regular Lip(1,1/2) graphs are parabolic uniformly rectifiable, see Subsection~\ref{subsec:UR}) and also \eqref{eq:amount_j}, we deduce
\begin{multline} \label{eq:local_cme}
	\iiint_{\mbf{C}_{10n\diam(I_{i, j})}(\mbf{x}_{i, j}) \cap \Omega_{\S^*}} \abs{\nabla u}^2 \delta_{\Omega_{\S^*}}
	=
	\iiint_{\mbf{C}_{10n\diam(I_{i, j})}(\mbf{x}_{i, j}) \cap \Omega_{i, j}} \abs{\nabla u}^2 \delta_{\Omega_{i, j}}
	\\ \lesssim 
	\norm{u}_{L^\infty(\Omega_{\S^*})}^2 \diam(I_{i, j})^{n+1}
	\approx 
	\norm{u}_{L^\infty(\Omega_{\S^*})}^2 \ell(I_i)^{n+1}.
\end{multline} 


\textbf{Substep 1.2: strips covered by these small balls.}
The reader is encouraged to have Figure~\ref{fig:strips} handy while constructing the following objects. Let us now define, for $c_1 = c_1(n) > 0$, 
\begin{equation*}
	\Strip (I_i) 
	:=
	\left\{ (x_0, x, t) \in \R^{n+1} : \;  (x, t) \in I_i \; \text{ and } \; \psi_{\S^*}(x, t) < x_0 \leq  \psi_{\S^*}(x, t)+ c_1 2^{-k_0} \ell(I_i) \right\},
\end{equation*}
where we recall that $\psi_{\S^*}$ parametrizes $\pom_{\S^*}$ (Proposition~\ref{prop:WHSA}).
We claim that if $c_1$ is small,
\begin{equation} \label{eq:strip_1}
	\Strip (I_i) 
	\subset
	\bigcup_{j \in \mathcal{J}_i} \mbf{C}_{10n\diam(I_{i, j})}(\mbf{x}_{i, j}) \cap \Omega_{\S^*}.
\end{equation}
Indeed, take $(x_0, x, t) \in \Strip(I_i)$. Then $(x, t) \in I_i$, so $(x, t) \in I_{i, j}$ for some $j\in \mathcal{J}_i$. Therefore
\begin{multline*}
	\norm{(x_0, x, t) - \mbf{x}_{i, j}} 
	= 
	\norm{(x_0, x, t) - (\psi_{\S^*}(x_{i, j}), x_{i, j}, t_{i, j})} 
	\\ =
	\abs{x_0 - \psi_{\S^*}(x_{i, j})}
	+ \norm{(x, t) - (x_{i, j}, t_{i, j})} 
	\leq 
	c_1 2^{-k_0} \ell(I_i) + \diam(I_{i, j})
	\leq 
	10 n \diam(I_{i, j}).
\end{multline*} 
where in the last inequality we have just taken $c_1 = c_1(n)$ small enough. This shows \eqref{eq:strip_1}.

\begin{figure}[t!]
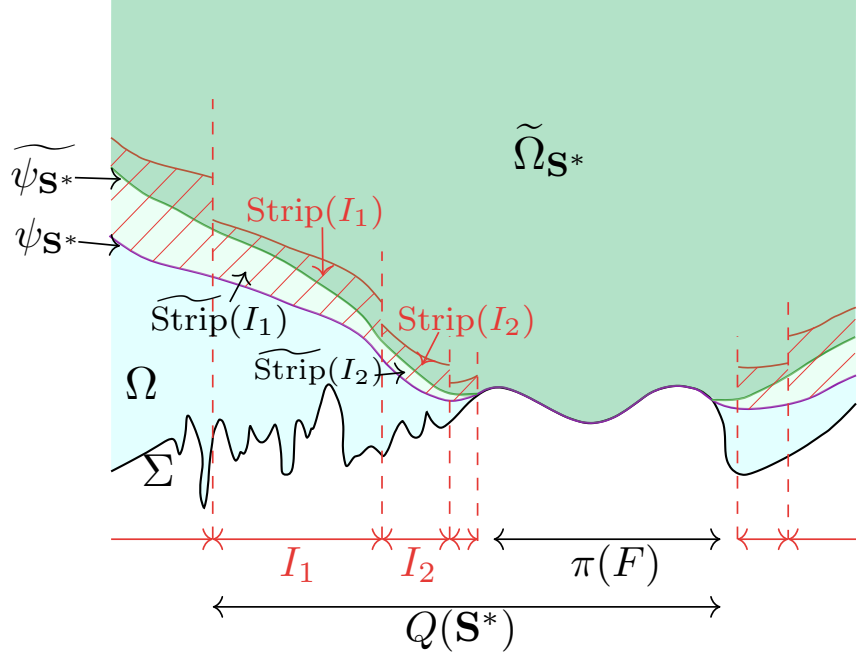

	\centering 
	\resizebox{0.75\textwidth}{!}{

	}
	\caption{Diagram of the construction of the graphs $\psi_{\S^*}$ (purple) and $\widetilde{\psi}_{\S^*}$ (green), the $\Strip(I_i)$ (hatched in red) and $\SStrip(I_i)$ (colored yellow) regions above each dyadic cube $I_i$ outside $\pi(F)$ (the set where $\psi_{\S^*}$ coincides with $\Sigma$), and $\widetilde{\Omega}_{\S^*}$ (colored green).}
	\label{fig:strips}
\end{figure}

To smoothen the shape of these strips, let us define, for small enough $c_2 > 0$ depending on $n, k_0$ and $c_1$ (so ultimately only on $n$),
\begin{equation*}
	\SStrip(I_i) 
	:=
	\Big\{ (x_0, x, t) \in \R^{n+1} : \; (x, t) \in I_i \; \text{ and } \; \psi_{\S^*}(x, t) < x_0 \leq \psi_{\S^*}(x, t) + c_2 D(x, t) \Big\},
\end{equation*}
so that it holds, simply because $D(x, t) \lesssim \diam(I_i)$ for $(x, t) \in I_i$ by \eqref{eq:diam_I_i},
\begin{equation} \label{eq:strip_2}
	\SStrip(I_i) 
	\subset 
	\Strip(I_i).
\end{equation} 
Therefore, if we define another graph by 
\begin{equation*}
	\widetilde{\psi}_{\S^*}(x, t) 
	:=
	\psi_{\S^*}(x, t) + c_2 D(x, t), 
	\qquad 
	(x, t) \in \R^{n-1} \times \R,
\end{equation*}
it turns out that $\widetilde{\psi}_{\S^*}$ is a Lip(1,1/2) graph because so is $\psi_{\S^*}$ by \eqref{eq:psi_lipschitz}, and $D$ by definition as a stopping-time distance. Define the domain above this graph by
\begin{equation*} 
	\widetilde{\Omega}_{\S^*}
	:=
	\big\{(x_0, x, t) \in \R \times \R^{n-1} \times \R: \;  x_0 > \widetilde{\psi}_{\S^*}(x, t) \big\}.
\end{equation*}
Then, it clearly holds
\begin{equation*}
	\Omega_{\S^*} 
	=
	\widetilde{\Omega}_{\S^*}
	\cup 
	\bigcup_{i \in \mathcal{I}} \SStrip(I_i).
\end{equation*} 
In fact, a localized version of the same decomposition readily holds inside $\C_r(\overline{\mbf{y}})$, the key point being that we may restrict indices to $\Lambda$ (from \eqref{eq:def_Lambda}) as a consequence of \eqref{eq:assumption_local}:
\begin{equation} \label{eq:above_below}
	\Omega_{\S^*} \cap \C_r(\overline{\mbf{y}})
	=
	\left( \widetilde{\Omega}_{\S^*}
	\cup 
	\bigcup_{i \in \Lambda} \SStrip(I_i) \right) \cap \C_r(\overline{\mbf{y}}).
\end{equation} 

As a result of the local p-CME estimate that we got in \eqref{eq:local_cme}, we obtain a p-CME estimate in the region covered by these strips:
\begin{multline} \label{eq:union_strips}
	\iiint_{\bigcup\limits_{i \in \Lambda} \SStrip(I_i) \cap \mbf{C}_r(\overline{\mbf{y}})} \abs{\nabla u}^2 \delta_{\Omega_{\S^*}} 
	\leq 
	\!\!\!\!
	\sum_{\substack{i \in \Lambda \\ \SStrip(I_i) \cap \mbf{C}_r(\overline{\mbf{y}}) \neq \emptyset}} 
	\!\!\!\!
	\iiint_{\SStrip(I_i)} \abs{\nabla u}^2 \delta_{\Omega_{\S^*}} 
	\\ \leq 
	\!\!\!\!
	\sum_{\substack{i \in \Lambda \\ \SStrip(I_i) \cap \mbf{C}_r(\overline{\mbf{y}}) \neq \emptyset}}
	\sum_{j \in \mathcal{J}_i} 
	\iiint_{\mbf{C}_{10n\diam(I_{i, j})}(\mbf{x}_{i, j}) \cap \Omega_{\S^*}} \abs{\nabla u}^2 \delta_{\Omega_{\S^*}} 
	\\ \lesssim 
	\norm{u}_{L^\infty(\Omega_{\S^*})}^2 \!\!\!\! \sum_{\substack{i \in \Lambda \\ \SStrip(I_i) \cap \mbf{C}_r(\overline{\mbf{y}}) \neq \emptyset}} \!\!\!\! \ell(I_i)^{n+1},
\end{multline}
where we have also used \eqref{eq:strip_1} and \eqref{eq:strip_2} in the second inequality, and \eqref{eq:amount_j} in the last one.


\textbf{Substep 1.3: pasting the strips.}
As already mentioned when obtaining \eqref{eq:above_below}, as a consequence of our localization assumption \eqref{eq:assumption_local}, if $(x_0, x, t) \in \SStrip(I_i) \cap \mbf{C}_r(\overline{\mbf{y}})$, then $(x, t) \in I_i$ for some $i \in \Lambda$. Moreover, since $(x_0, x, t) \in \C_r(\mbf{\overline{y}})$,
\begin{equation*}
	\abs{(x, t) - (\overline{y}, \overline{s})} \leq \norm{(x_0, x, t) - (\psi_{\S^*}(\overline{y}, \overline{s}), \overline{y}, \overline{s})} < r,
\end{equation*}
whence $\abs{D(x, t) - D(\overline{y}, \overline{s})} < r$ because $D \in \Lip(1,1/2)$ with constant 1. Therefore, since $D(\overline{y}, \overline{s}) \leq r/\mu$ (this is the assumption in Case 1) for $\mu$ small, we have that $D(x, t) \leq 2r/\mu$. Thus, \eqref{eq:diam_I_i} yields $\diam(I_i) \lesssim r$, from which it follows that 
\begin{equation*}
	I_i \subset C'_{M_1r}(\overline{y}, \overline{s})
\end{equation*}
for $M_1 > 0$ depending on $\mu$; but we will omit this dependence because $\mu = \mu(n)$. 

Therefore, noting also that $\SStrip(I_i)$ are disjoint because so are their projections $I_i$,
\begin{multline} \label{eq:disjoint_strips}
	\sum_{\substack{i \in \Lambda \\ \SStrip(I_i) \cap \mbf{C}_r(\overline{\mbf{y}}) \neq \emptyset}} \ell(I_i)^{n+1}
	\approx 
	\sum_{\substack{i \in \Lambda \\ \SStrip(I_i) \cap \mbf{C}_r(\overline{\mbf{y}}) \neq \emptyset}} \mathcal{H}^{n+1}_\parab(I_i)
	\\ =
	\mathcal{H}^{n+1}_\parab \bigg( \bigcup_{\substack{i \in \Lambda \\ \SStrip(I_i) \cap \mbf{C}_r(\overline{\mbf{y}}) \neq \emptyset}} I_i \bigg)
	\leq 
	\mathcal{H}^{n+1}_\parab (C'_{M_1r}(\overline{y}, \overline{s}))
	\lesssim 
	r^{n+1}.
\end{multline}
In this computation, $\mathcal{H}^{n+1}_\parab$ is the Hausdorff measure in space-time $\R^{n-1} \times \R$ (i.e. removing the first coordinate from the whole space-time), as explained in Subsection~\ref{subsec:geometry}. Although $\mathcal{H}^{n+1}_\parab$ is a measure on a space of $n$ dimensions, its homogeneity is $n+1$ because time is one of the variables involved (recall the inhomogeneous metric that we use in space-time, see Subsection~\ref{subsec:geometry}).

Hence, recalling \eqref{eq:union_strips}, the CME estimate holds in the union of the strips, that is, below the graph of $\widetilde{\psi}_{\S^*}$. Therefore taking into account \eqref{eq:above_below}, to obtain \eqref{goal:cme}, we are left to prove
\begin{equation} \label{goal:above_sawtooth}
	\iiint_{\widetilde{\Omega}_{\S^*} \cap \mbf{C}_r(\overline{\mbf{y}})} \abs{\nabla u}^2 \delta_{\Omega_{\S^*}} 
	\lesssim 
	\norm{u}_{L^\infty(\Omega_{\S^*})}^2 r^{n+1}.
\end{equation}


\textbf{Step 2: CME holds above the graph $\widetilde{\psi}_{\S^*}$, via integration by parts.} We are left to show \eqref{goal:above_sawtooth}, so we only need to work above the graph $\widetilde{\psi}_{\S^*}$ constructed in Step 1. The advantage is that here, the Green's function enjoys desirable properties because $\widetilde{\Omega}_{\S^*}$ contains Whitney regions associated to good cubes in the corona decompositions. Concretely, $\delta_{\Omega_{\S^*}} \lesssim \widehat{u}$ will be true because of Proposition~\ref{prop:WHSA}, which substituted back in \eqref{goal:above_sawtooth} will allow us to carry out an integration by parts, and later finish using the upper bounds from Lemma~\ref{lem:u_upper_bounds}. On the way, one needs to be careful when comparing $\delta_\Omega$ and $\delta_{\Omega_{\S^*}}$, because $u$ is a solution in $\Omega_{\S^*}$, but $\widehat{u}$ is an (adjoint) solution in $\Omega$.

Recall that by \eqref{eq:assumption_local}, we have $r \leq \frac12 \kappa R_{\S^*}$, where $\kappa$ and $R_{\S^*}$ are defined in Proposition~\ref{prop:WHSA}.


\textbf{Substep 2.1: $\delta_{\Omega_{\S^*}} \lesssim_\varepsilon \widehat{u}$ in $\widetilde{\Omega}_{\S^*} \cap \mbf{C}_r(\overline{\mbf{y}})$.} We start by claiming that, with $\Omega'_{\S^*}$ from \eqref{def:subdomain_truncated},
\begin{equation} \label{claim:contained}
	(x_0, x, t) \in \widetilde{\Omega}_{\S^*} \cap \mbf{C}_r(\overline{\mbf{y}})
	\quad \implies \quad 
	(x_0, x, t) \in \Omega'_{\S^*}.
\end{equation}
Indeed, by \eqref{eq:assumption_local}, $(x, t) \in C'_r (\overline{y}, \overline{s}) \subset C'_{2\kappa R_{\S^*}}(x_{Q(\S^*)}, t_{Q(\S^*)})$. Therefore, if we pick any point $(z, \tau) \in \partial C'_{8\kappa R_{\S^*}}(x_{Q(\S^*)}, t_{Q(\S^*)})$, it turns out that \eqref{eq:psi_lipschitz} and \eqref{eq:psi_zero_far} imply 
\begin{multline} \label{eq:psi_not_huge}
	\psi_{\S^*}(x, t)
	=
	\psi_{\S^*}(x, t) - \psi_{\S^*}(z, \tau)
	\leq 
	C_{\mathrm{LA}} \varepsilon^{1/2} \abs{(x, t) - (z, \tau)}
	\\ \leq 
	C_{\mathrm{LA}} \varepsilon^{1/2} \left( \abs{(x, t) - (x_{Q(\S^*)}, t_{Q(\S^*)})} + \abs{(x_{Q(\S^*)}, t_{Q(\S^*)}) - (z, \tau)} \right)
	\\ \leq 
	C_{\mathrm{LA}} \varepsilon^{1/2} 10 \kappa R_{\S^*}
	=
	10 C_{\mathrm{LA}}^2 K_0 \, \varepsilon^{1/2} \diam(Q(\S^*))
	\leq 
	\varepsilon^{1/4} \ell(Q(\S^*))
\end{multline} 
because $\varepsilon$ was chosen small enough (recall \eqref{eq:epsilon_K_0}). This readily implies \eqref{claim:contained}.

Thus, if $(x_0, x, t) \in \widetilde{\Omega}_{\S^*} \cap \mbf{C}_r(\overline{\mbf{y}})$, \eqref{eq:containment} and \eqref{claim:contained} imply that $(x_0, x, t) \in \widetilde{\mbf{U}}_Q^\good$ for some $Q \in \S^*$. Therefore, \eqref{eq:upper_bound_u_enlarged} implies that, for $\widehat{u}$ defined in \eqref{eq:def_hatu},
\begin{equation*}
	\delta_\Omega(x_0, x, t) \lesssim_\varepsilon \widehat{u} (x_0, x, t).
\end{equation*} 

To finish this substep, we claim that it is also true that, if $(x_0, x, t) \in \widetilde{\Omega}_{\S^*} \cap \mbf{C}_r(\overline{\mbf{y}})$, then
\begin{equation} \label{eq:dist_leq_dist}
	\delta_{\Omega_{\S^*}}(x_0, x, t) \leq \delta_\Omega(x_0, x, t),
\end{equation}
so that, putting together these last two estimates, to show \eqref{goal:above_sawtooth} it is enough to verify
\begin{equation} \label{goal:int_by_parts_2}
	\iiint_{\widetilde{\Omega}_{\S^*} \cap \mbf{C}_r(\overline{\mbf{y}})} \abs{\nabla u}^2 \widehat{u}
	\lesssim 
	\norm{u}_{L^\infty(\Omega_{\S^*})}^2 r^{n+1}.
\end{equation}

Let us prove \eqref{eq:dist_leq_dist}. 
By \eqref{eq:containment}, we have $\Omega'_{\S^*} \subset \Omega$, so $\delta_{\Omega'_{\S^*}}(x_0, x, t) \leq \delta_\Omega(x_0, x, t)$. Now we will show that $\delta_{\Omega'_{\S^*}}(x_0, x, t) = \delta_{\Omega_{\S^*}}(x_0, x, t)$. For that purpose, denote $(z_0, z, \tau) \in \partial \Omega'_{\S^*}$ a point for which $\delta_{\Omega'_{\S^*}}(x_0, x, t) = \norm{(x_0, x, t) - (z_0, z, \tau)}$. 

First, since $(\psi_{\S^*}(\overline{y}, \overline{s}), \overline{y}, \overline{s}) \in \pom_{\S^*} \not\subset \Omega'_{\S^*}$ and $(x_0, x, t) \in \widetilde{\Omega}_{\S^*} \cap \C_r(\overline{\mbf{y}})$, it definitely holds $\delta_{\Omega'_{\S^*}}(x_0, x, t) \leq \norm{(x_0, x, t) - (\psi_{\S^*}(\overline{y}, \overline{s}), \overline{y}, \overline{s})} < r$. Thus $(z_0, z, \tau) \in \C_{2r}(\psi_{\S^*}(\overline{y}, \overline{s}), \overline{y}, \overline{s})$ simply by the triangle inequality, whence $(z, \tau) \in C'_{2r}(\overline{y}, \overline{s}) \subset C'_{\kappa R_{\S^*}}(x_{Q(\S^*)}, t_{Q(\S^*)})$ by \eqref{eq:assumption_local}. Moreover, as in \eqref{eq:psi_not_huge}, since $\varepsilon$ is small, it holds $\psi_{\S^*}(\overline{y}, \overline{s}) \leq \frac12 \varepsilon^{-1/4} \ell(Q(\S^*))$,
so we have (recalling again our localization assumption \eqref{eq:assumption_local})
\begin{multline*}
	z_0 
	\leq 
	\psi_{\S^*}(\overline{y}, \overline{s}) +	\norm{(z_0, z, \tau) - (\psi_{\S^*}(\overline{y}, \overline{s}), \overline{y}, \overline{s})}
	\\ \leq 
	\frac12 \varepsilon^{-1/4}\ell(Q(\S^*)) + 2r
	\leq 
	\frac12 \varepsilon^{-1/4}\ell(Q(\S^*)) + \kappa R_{\S^*}
	< 
	\varepsilon^{-1/4} \ell(Q(\S^*)).
\end{multline*}
This shows that the truncation from $\Omega_{\S^*}$ to $\Omega'_{\S^*}$ does not affect $(z_0, z, \tau)$, i.e., $(z_0, z, \tau) \in \partial \Omega_{\S^*}$. This confirms that $\delta_{\Omega'_{\S^*}}(x_0, x, t) = \delta_{\Omega_{\S^*}}(x_0, x, t)$ and finishes the proof of \eqref{eq:dist_leq_dist}.


\textbf{Substep 2.2: cut-off function adapted to the region.}
We will show \eqref{goal:int_by_parts_2} by an integration by parts argument. We will localize the estimate to the region $\widetilde{\Omega}_{\S^*} \cap \mbf{C}_r(\overline{\mbf{y}})$ by choosing an appropriate cutoff function.

Indeed, using a standard argument,\footnote{
	Since $\Omega_{\S^*}$ is a graphical domain, this can be achieved with an elementary argument using a Whitney decomposition and partitions of unity, like in \cite[Subsection 5.3]{BHMN_graph}. Indeed, one considers Whitney boxes $\mbf{I}$ that intersect $\widetilde{\Omega}_{\S^*} \cap \mbf{C}_r(\overline{\mbf{y}})$, and a dilation parameter $\vartheta > 0$ small enough so that $(1+\vartheta) \mbf{I}$ fit inside $\widetilde{\Omega}_{\S^*} \cap \mbf{C}_{2r}(\overline{\mbf{y}})$. Then, given $N \in \N$, one discards the boxes which are smaller than $1/N$ to stay away from the boundary $\partial \widetilde{\Omega}_{\S^*}$, and still cover the whole domain as $N \to \infty$. Using a partition of unity associated to this family of Whitney boxes, it is easy to construct the desired cutoff $\Psi_N$ by combining cutoffs $\Psi_{\mbf{I}}$ associated to each box $\mbf{I}$. The fact that the estimate \eqref{eq:cutoff} holds is intuitive from the (parabolic) scaling of the cutoffs $\Psi_{\mbf{I}}$, and that the only regions in which $\Psi_N$ is non-constant are the Whitney boxes which are close to the boundary of $\supp \Psi_N$, which somehow covers the ``perimeter'' of this region (hence the homogeneity $n+1$ in \eqref{eq:cutoff}, which is codimension 1 in parabolic $\R^{n+1}$). For more details, see \cite[Subsection 5.3]{BHMN_graph}. Similar constructions in more challenging geometrical frameworks can be found in detail in \cite[Lemma 6.69]{BHMN} or \cite[Appendix A]{BFHH}. 
} we can find, for every $N \in \N$, some cut-off function $\Psi_N \in C^\infty_c(\widetilde{\Omega}_{\S^*} \cap \mbf{C}_{2r}(\overline{\mbf{y}}))$ which approximates the domain of integration in the sense that
\begin{equation*} 
	0 \leq \Psi_N \leq 1,
	\qquad 
	\Psi_N \text{ is increasing in } N,
	\qquad 
	\lim_{N \to \infty} \Psi_N(\mbf{X}) = 1  \quad\; \forall \, \mbf{X} \in \widetilde{\Omega}_{\S^*} \cap \mbf{C}_r(\overline{\mbf{y}}),
\end{equation*}
such that the following estimate holds\footnote{
	The fact that we are able to also control second derivatives of $\Psi_N$ in \eqref{eq:cutoff}, which is not needed in \cite[Subsection 5.3]{BHMN_graph}, is easy to check just from the elementary construction of the $\Psi_{\mbf{I}}$ with the appropriate parabolic scaling for them to be associated to the box $\mbf{I}$.
}
\begin{equation} \label{eq:cutoff}
	\iiint_{\widetilde{\Omega}_{\S^*}} \left( \abs{\nabla \Psi_N} + \abs{\nabla^2 \Psi_N} \delta_{\Omega_{\S^*}} + \abs{\partial_t \Psi_N} \delta_{\Omega_{\S^*}} \right) 
	\lesssim 
	r^{n+1}.
\end{equation}

With this in mind, by monotone convergence, \eqref{goal:int_by_parts_2} reduces to showing
\begin{equation} \label{goal:int_by_parts_3}
	\iiint_{\widetilde{\Omega}_{\S^*}} \abs{\nabla u}^2 \widehat{u} \,\Psi_N
	\lesssim 
	\norm{u}_{L^\infty(\Omega_{\S^*})}^2 r^{n+1},
\end{equation}
with an implicit constant that is independent of $N$.


\textbf{Substep 2.3: integration by parts (part I).}
Since $u$ is caloric, we can compute
\begin{equation*}
	H(u^2) = 2 u \, \partial_t u - 2u \Delta u - 2 \abs{\nabla u}^2 = - 2 \abs{\nabla u}^2, 
\end{equation*} 
which allows us to integrate by parts using $\widehat{u} \Psi_N, \partial_{x_i} (\widehat{u} \Psi_N) \in C^\infty_c(\widetilde{\Omega}_{\S^*} \cap \mbf{C}_{2r}(\overline{\mbf{y}})) \subset C^\infty_c(\Omega'_{\S^*})$:\footnote{The set containment follows by a simple modification of \eqref{claim:contained} to allow for $\mbf{C}_{2r}(\overline{\mbf{y}})$ instead of $\mbf{C}_r(\overline{\mbf{y}})$.} 
\begin{align*}
	\iiint_{\widetilde{\Omega}_{\S^*}} \abs{\nabla u}^2 \widehat{u} \, \Psi_N
	=
	-\frac12 \iiint_{\widetilde{\Omega}_{\S^*}} H(u^2) \, \widehat{u} \, \Psi_N 
	=
	\frac12 \iiint_{\widetilde{\Omega}_{\S^*}} u^2 H^* (\widehat{u} \Psi_N),
\end{align*}
where $H^* := \partial_t + \Delta$ is the adjoint heat operator.
Moreover, the pole $(Y_\S, s_\S)$ is far from our region of integration\footnote{
	This is an elementary consequence of our constructions and the choice of $\varepsilon$ very small. For more details, see the computations in the Appendix.
}, so it holds $H^*\widehat{u} = 0$ in $\Omega'_{\S^*}$. Therefore
\begin{align} \label{eq:ibp_part1}
	\iiint_{\widetilde{\Omega}_{\S^*}} u^2 H^* (\widehat{u} \Psi_N)
	& =
	\iiint_{\widetilde{\Omega}_{\S^*}} u^2 \, \partial_t \widehat{u} \, \Psi_N
	+ \iiint_{\widetilde{\Omega}_{\S^*}} u^2 \, \widehat{u} \, \partial_t \Psi_N
	\nonumber
	\\ & \qquad + \iiint_{\widetilde{\Omega}_{\S^*}} u^2 \Delta \widehat{u} \, \Psi_N
	+ 2 \iiint_{\widetilde{\Omega}_{\S^*}} u^2 \, \nabla \widehat{u} \cdot \nabla \Psi_N
	+ \iiint_{\widetilde{\Omega}_{\S^*}} u^2 \, \widehat{u} \, \Delta \Psi_N
	\nonumber
	\\ & =
	\iiint_{\widetilde{\Omega}_{\S^*}} u^2 \, \widehat{u} \, \partial_t \Psi_N
	+ 2 \iiint_{\widetilde{\Omega}_{\S^*}} u^2 \, \nabla \widehat{u} \cdot \nabla \Psi_N
	+ \iiint_{\widetilde{\Omega}_{\S^*}} u^2 \, \widehat{u} \, \Delta \Psi_N
	\\ & \lesssim_\varepsilon
	\norm{u}_{L^\infty(\Omega_{\S^*})}^2 
	\left( 
	\iiint_{\widetilde{\Omega}_{\S^*}} \delta_\Omega \abs{\partial_t \Psi_N}
	+ \iiint_{\widetilde{\Omega}_{\S^*}} \abs{\nabla \Psi_N}
	+ \iiint_{\widetilde{\Omega}_{\S^*}} \delta_\Omega \abs{\nabla^2 \Psi_N}
	\right) 
	\nonumber
\end{align}
where in the last step we have used Lemma~\ref{lem:u_upper_bounds}, because by choice of $\Psi_N$, \eqref{claim:contained} and \eqref{eq:containment},
\begin{equation*}
	\supp \Psi_N 
	\subset 
	\widetilde{\Omega}_{\S^*} \cap \mbf{C}_{2r}(\overline{\mbf{y}}) 
	\subset 
	\Omega'_{\S^*} \subset \bigcup_{Q \in \S^*} \widetilde{\mbf{U}}_Q^\good.
\end{equation*}


\textbf{Substep 2.4: $\delta_\Omega \lesssim \delta_{\Omega_{\S^*}}$ in $\widetilde{\Omega}_{\S^*}$.}\footnote{This step is actually the place where we are required to stay away from $\Gamma_{\S^*} = \pom_{\S^*}$, hence the need to create the second graph via $\widetilde{\psi}_{\S^*}$ in Step 1.2.} 
Pick $(x_0, x, t) \in \widetilde{\Omega}_{\S^*}$. Then, by definition of $\widetilde{\Omega}_{\S^*}$, 
\begin{equation*}
	\rho 
	:=
	x_0 - \psi_{\S^*}(x, t)
	>
	\widetilde{\psi}_{\S^*}(x, t) - \psi_{\S^*}(x, t)
	= 
	c_2 D(x, t).
\end{equation*}

Let us first consider the case that $D(x, t) \neq 0$. Then, $(x, t) \notin \pi(F)$, so $(x, t) \in I_i$ for some $i \in \Lambda$ (check Proposition~\ref{prop:WHSA} and \eqref{eq:assumption_local}). Then the inequality in the last display tells us that $\diam(I_i) \lesssim \rho$ by \eqref{eq:diam_I_i}, so that \eqref{eq:approx_I_i} yields
\begin{multline*}
	\delta_\Omega(x_0, x, t) 
	\leq 
	\delta_\Omega (\psi_{\S^*}(x, t), x, t) + \norm{(x_0, x, t) - (\psi_{\S^*}(x, t), x, t)}
	\\ \leq 
	\varepsilon^{1/2} \diam(I_i) + \abs{x_0 - \psi_{\S^*}(x, t)}
	\leq
	2\rho
\end{multline*}
because $\varepsilon$ is small. Now, since $\Gamma_{\S^*}$ is a Lip(1,1/2) graph with Lipschitz constant smaller than 1/2 (actually, see \eqref{eq:psi_lipschitz}), it is elementary that we can measure, up to constants, the distance to it by just moving in the vertical direction $x_0$ (see \cite[Lemma 3.2]{BHHLN}), so 
\begin{equation*}
	\delta_\Omega(x_0, x, t) 
	\leq
	2\rho
	=
	2 \big(x_0 - \psi_{\S^*}(x, t)\big)
	\leq 
	4 \dist\big((x_0, x, t), \Gamma_{\S^*}\big)
	=
	4 \delta_{\Omega_{\S^*}}(x_0, x, t).
\end{equation*}

In the case that $D(x, t) = 0$, then $(\psi_{\S^*}(x, t), x, t) \in \pom$ (see Proposition~\ref{prop:WHSA}). This allows to compute, again measuring distances moving in the vertical direction, that
\begin{equation*}
	\delta_\Omega(x_0, x, t) 
	\leq 
	\abs{(x_0, x, t) - (\psi_{\S^*}(x, t), x, t)}
	=
	\abs{x_0 - \psi_{\S^*}(x, t)}
	\leq
	2 \delta_{\Omega_{\S^*}}(x_0, x, t).
\end{equation*}

Thus, $\delta_\Omega \leq 4 \delta_{\Omega_{\S^*}}$ is true in $\widetilde{\Omega}_{\S^*}$.


\textbf{Substep 2.5: integration by parts (part II).}
Now we can finish the computation started in Step 2.3. Indeed, continuing \eqref{eq:ibp_part1} with the result of Step 2.4 and \eqref{eq:cutoff}, we obtain
\begin{align*}
	\iiint_{\widetilde{\Omega}_{\S^*}} u^2 H^* (\widehat{u} \Psi_N)
	& \lesssim_\varepsilon
	\norm{u}_{L^\infty(\Omega_{\S^*})}^2 
	\left( 
	\iiint_{\widetilde{\Omega}_{\S^*}} \delta_{\Omega_{\S^*}} \abs{\partial_t \Psi_N}
	+ \iiint_{\widetilde{\Omega}_{\S^*}} \abs{\nabla \Psi_N}
	+ \iiint_{\widetilde{\Omega}_{\S^*}} \delta_{\Omega_{\S^*}} \abs{\nabla^2 \Psi_N}
	\right) 
	\\ & \lesssim 
	\norm{u}_{L^\infty(\Omega_{\S^*})}^2 
	r^{n+1},
\end{align*}
This establishes \eqref{goal:int_by_parts_3}, hence finishing the proof of \eqref{goal:above_sawtooth}. Overall, this finishes the proof of Theorem~\ref{th:push_cme} under the additional assumption \eqref{eq:assumption_local} in Case 1.


\subsection{Modifications for Case 2} \label{subsec:case_2} 

In Case 2, namely when $r < \mu D(\overline{y}, \overline{s})$, we can even simplify the arguments because $r$ is so small that it suffices to consider the strips to finish the problem (there is no need for the Step 2 from the previous subsection, where we got the estimate \eqref{goal:above_sawtooth} above the sawtooth via integration by parts). However, we need to refine our argument before because $r$ can be very tiny compared to $\ell(I_i)$, so we should rather consider scales comparable to $r$ instead of $\ell(I_i)$. Let us do this more precisely.

Indeed, in this case 2, find $i_0$ for which $(\overline{y}, \overline{s}) \in I_{i_0}$, so $i_0 \in \Lambda$ because of \eqref{eq:assumption_local}. Consider now the dyadic subcube $I_{i_0}' \subset I_{i_0}$ for which $(\overline{y}, \overline{s}) \in I_{i_0}'$ and $\ell(I_{i_0}') \in (r, 2r]$. This will be the correct scale to work at. 

Consider the following family of dyadic cubes in $\R^n$:
\begin{equation*}
	\mathcal{D}_{i_0}
	:=
	\left\{ I' \text{ dyadic cube in $\R^n$} : 
	\begin{array}{c}
		\ell(I') = \ell(I'_{i_0}), \quad I' \subset 4nI_{i_0}',
		\\  I' \subset I_i \text{ for some $I_i$ with $i \in \Lambda$}
	\end{array}  
	\right\}.
\end{equation*}
Then we claim that
\begin{equation} \label{eq:boxes_cover}
	C'_r(\overline{y}, \overline{s}) 
	\subset 
	\bigcup_{I' \in \mathcal{D}_{i_0}} I' 
	\subset 
	4n I_{i_0}' 
	\subset 
	2I_{i_0}.
\end{equation} 
Indeed, for the first inclusion, note that any $(x, t) \in C'_r(\overline{y}, \overline{s})$ satisfies $\abs{(x, t) - (\overline{y}, \overline{s})} < r$, whence $|(x, t) - (x_{I_{i_0}'}, t_{I_{i_0}'})| < r + \diam(I_{i_0}') \leq \ell(I_{i_0}') + \diam(I_{i_0}') \leq 2n \diam(I_{i_0}')$ because $(\overline{y}, \overline{s}) \in I_{i_0}'$. Now pick the dyadic cube $I'$ in the generation of $I_{i_0}'$ for which $(x, t) \in I'$. The previous computation shows that $I' \cap 2nI_{i_0}' \neq \emptyset$. Since $\ell(I') = \ell(I_{i_0}')$, we have $I' \subset 4nI_{i_0}'$. Moreover, \eqref{eq:assumption_local} easily implies that $I' \subset I_i$ for some $i \in \Lambda$. Altogether, this means that $I' \in \mathcal{D}_{i_0}$.
In turn, the second inclusion follows from definition of $\mathcal{D}_{i_0}$. The last inclusion is true because $I_{i_0}' \subset I_{i_0}$ and 
\begin{equation*}
	\diam(I_{i_0}') \leq n \ell(I_{i_0}') \leq 2nr \leq 2n\mu D(\overline{y}, \overline{s}) \leq 120n\mu\diam(I_{i_0})
\end{equation*}
(here we used the fact that we are in Case 2, and \eqref{eq:diam_I_i} because $(\overline{y}, \overline{s}) \in I_{i_0}$), which imply that 
\begin{align*}
	4nI_{i_0}'
	& =
	\big\{ (x, t) \in \R^n : \dist((x, t), I_{i_0}') \leq (4n-1) \diam(I_{i_0}') \big\}
	\\ & \subset 
	\big\{ (x, t) \in \R^n : \dist((x, t), I_{i_0}) \leq 4n \cdot 120n\mu \diam(I_{i_0}) \big\} 
	\subset 
	2I_{i_0}
\end{align*} 
simply choosing $\mu$ small enough only depending on $n$ (indeed, $480n^2\mu \leq 1$ suffices).

Now that \eqref{eq:boxes_cover} is clear, take each $I' \in \mathcal{D}_{i_0}$ and run a similar argument as in Step 1.1 before to obtain a covering by $\{I_j'\}_j$ using Lemma~\ref{lem:DDH}, for which \eqref{eq:local_cme} holds. To finish the proof, it suffices to show that in this case, with the definitions of Step 1.2 before, we have
\begin{equation} \label{eq:strips_cover}
	\mbf{C}_r(\overline{\mbf{y}}) \cap \Omega_{\S^*} \, \subset \, \bigcup_{I' \in \mathcal{D}_{i_0}} \SStrip(I'),
\end{equation}
and then, with \eqref{eq:union_strips}, \eqref{eq:disjoint_strips} and \eqref{eq:boxes_cover} we obtain 
\begin{multline*}
	\iiint_{\mbf{C}_r(\overline{\mbf{y}}) \cap \Omega_{\S^*}} \abs{\nabla u}^2 \delta_{\Omega_{\S^*}} 
	\leq 
	\iiint_{\bigcup\limits_{I' \in \mathcal{D}_{i_0}} \SStrip(I')} \abs{\nabla u}^2 \delta_{\Omega_{\S^*}} 
	\lesssim 
	\norm{u}_{L^\infty(\Omega_{\S^*})}^2 \sum_{I' \in \mathcal{D}_{i_0}} \ell(I')^{n+1}
	\\ \approx 
	\norm{u}_{L^\infty(\Omega_{\S^*})}^2\mathcal{H}^{n+1}_\parab \Big( \bigcup_{I' \in \mathcal{D}_{i_0}} I' \Big)
	\leq 
	\norm{u}_{L^\infty(\Omega_{\S^*})}^2 \mathcal{H}^{n+1}_\parab (4nI_{i_0}')
	\lesssim 
	\norm{u}_{L^\infty(\Omega_{\S^*})}^2 r^{n+1}.
\end{multline*} 

Therefore, let us finish with the proof of \eqref{eq:strips_cover}. Pick $(x_0, x, t) \in \mbf{C}_r(\overline{\mbf{y}}) \cap \Omega_{\S^*}$, so that clearly $\abs{(x, t) - (\overline{y}, \overline{s})} < r$. Then, by the first inclusion in \eqref{eq:boxes_cover}, $(x, t) \in I'$ for some $I' \in \mathcal{D}_{i_0}$. 
Then we can compute, since $\norm{D}_{\Lip(1,1/2)} \leq 1$ by definition as a stopping-time distance,
\begin{multline*}
	D(\overline{y}, \overline{s})
	\leq 
	\abs{D(\overline{y}, \overline{s}) - D(x, t)} + D(x, t) 
	\leq 
	\norm{(\overline{y}, \overline{s}) - (x, t)} + D(x, t)
	\\ \leq 
	r + D(x, t)
	< 
	\mu D(\overline{y}, \overline{s}) + D(x, t),
\end{multline*}
and since $\mu < 1/2$ we can hide the first term in the right hand side to obtain
\begin{equation*} 
	D(\overline{y}, \overline{s}) \leq 2 D(x, t).
\end{equation*} 
This allows us to estimate
\begin{align*}
	x_0 
	& < 
	\psi_{\S^*}(\overline{y}, \overline{s}) + r
	\hspace{4.65cm}
	\text{(because $(x_0, x, t) \in \mbf{C}_r(\overline{\mbf{y}})$)}
	\\ & \leq 
	\psi_{\S^*}(\overline{y}, \overline{s}) + \mu D(\overline{y}, \overline{s})
	\hspace{3.45cm}
	\text{(by the hypothesis of Case 2)} 
	\\ & \leq 
	\psi_{\S^*}(x, t) + \abs{\psi_{\S^*}(\overline{y}, \overline{s}) - \psi_{\S^*}(x, t)} + \mu D(x, t) + \mu \abs{D(\overline{y}, \overline{s}) - D(x, t)} 
	\\ & \leq 
	\psi_{\S^*}(x, t) + \norm{(\overline{y}, \overline{s}) - (x, t)} + \mu D(x, t) + \norm{(\overline{y}, \overline{s}) - (x, t)}
	\\ & 
	\hspace{1cm}
	\text{(because $\norm{\psi_{\S^*}}_{\Lip(1,1/2)} \leq 1$ recalling \eqref{eq:psi_lipschitz}, and $\norm{D}_{\Lip(1,1/2)} \leq 1$)}
	\\ & \leq 
	\psi_{\S^*}(x, t) + \mu D(x, t) + 2r
	\hspace{2.75cm} 
	\text{(because $(x_0, x, t) \in \mbf{C}_r(\overline{\mbf{y}})$)}
	\\ & \leq 
	\psi_{\S^*}(x, t) + \mu D(x, t) + 2 \mu D(\overline{y}, \overline{s})
	\hspace{1.5cm} 
	\text{(by the hypothesis of Case 2)}
	\\ & \leq 
	\psi_{\S^*}(x, t) + 5\mu D(x, t)
	\hspace{3.35cm} 
	\text{(by the last display)}
	\\ & \leq 
	\psi_{\S^*}(x, t) + c_2 D(x, t)
	\hspace{3.35cm} 
	\text{ (by choosing $\mu \leq c_2 / 5$)}.
\end{align*}
This finishes the proof of \eqref{eq:strips_cover}, hence of Theorem~\ref{th:push_cme} under the assumption \eqref{eq:assumption_local} in Case 2.


\subsection{Removing the extra hypothesis} \label{subsec:gral_case}

Lastly, now that we have understood the core of the proof in the two previous subsections, let us remove the hypothesis \eqref{eq:assumption_local}. The equation \eqref{eq:psi_zero_far} says that the function $\psi_{\S^*}$ is zero outside a large cylinder, so this implies that we have already considered the most relevant cases. However, some small modifications are needed to cover the case when $r$ is extremely large, and concretely covers a large portion where $\psi_{\S^*} \equiv 0$.

The reader may want to see Figure~\ref{fig:decomposition_A} for intuition about the subsequent arguments.

\begin{figure}[h]
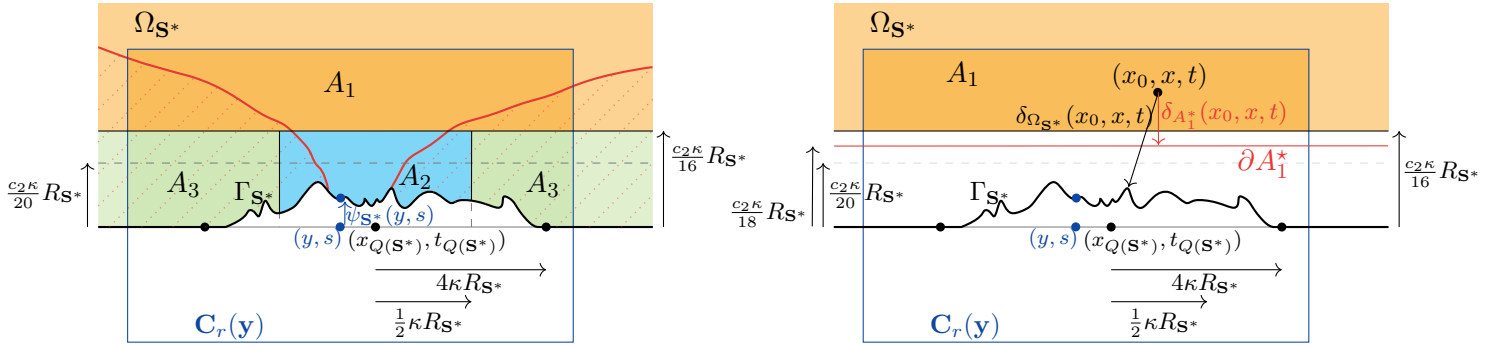

	\centering	
	\makebox[\textwidth][c]{
		\begin{subfigure}{0.6\textwidth}

		\end{subfigure}
	}
	
	\caption{In the left figure, the decomposition \eqref{eq:A1A2A3} is shown: $A_1$ is flat and strictly above $\Gamma_{\S^*}$, $A_2$ is near the center of $\C_r(\overline{\mbf{y}})$ (so studied by the previous cases), and $A_3$ is close to the region where $\psi_{\S^*} \equiv 0$ by \eqref{eq:psi_zero_far} and hence is covered by the $\Strip(I_i)$ (depicted in red), as shown in Step 1. In the right figure, it is sketched why $\delta_{\Omega_{\S^*}} \lesssim \delta_{A_1^\star}$ inside $A_1$: $\Gamma_{\S^*}$ lies below the threshold $\frac{c_2\kappa}{20} R_{\S^*}$, $A_1$ lies above $\frac{c_2\kappa}{16} R_{\S^*}$, and $\delta_{A_1^\star}$ lies right in between (at height $\frac{c_2\kappa}{18} R_{\S^*}$). }
	\label{fig:decomposition_A}
\end{figure}

\textbf{Step 1: Thick flat strips away from $Q(\S^*)$.}
First, we cover a neighborhood of the set $\Rn \setminus C'_{\frac14 \kappa R_{\S^*}}(x_{Q(\S^*)}, t_{Q(\S^*)})$ with a thick flat strip. For that purpose, first note that for any given $(x, t) \in \Rn \setminus C'_{\frac14 \kappa R_{\S^*}}(x_{Q(\S^*)}, t_{Q(\S^*)})$ we have, by the definitions of $D$ and $R_{\S^*}$ from Proposition~\ref{prop:WHSA}, 
\begin{equation} \label{eq:D_large_far}
	D(x, t) \geq \dist((x, t), \pi(Q(\S^*))) \geq \abs{(x, t) - (x_{Q(\S^*)}, t_{Q(\S^*)})} - \diam(Q(\S^*)) \geq \frac{\kappa}{8} R_{\S^*}.
\end{equation}
Therefore, if we define 
\begin{equation*}
	\widetilde{\widetilde{\Strip}} 
	:= \!
	\left\{ (x_0, x, t) \in \R^{n+1} : \psi_{\S^*}(x, t)  < x_0 \leq \frac{c_2 \kappa}{16} R_{\S^*} \text{ and } (x, t) \in \Rn \setminus C'_{\frac14 \kappa R_{\S^*}}(x_{Q(\S^*)}, t_{Q(\S^*)}) \right\},
\end{equation*}
we claim that it holds
\begin{equation} \label{eq:inclusion_strips}
	\widetilde{\widetilde{\Strip}} 
	\; \subset
	\bigcup_{\substack{i \in \mathcal{I} \\ I_i \cap \big(\Rn \setminus C'_{\frac14 \kappa R_{\S^*}}(x_{Q(\S^*)}, t_{Q(\S^*)})\big) \neq \emptyset}} \SStrip(I_i).
\end{equation}
Indeed, pick any $(x_0, x, t) \in \widetilde{\widetilde{\Strip}}$. Then \eqref{eq:D_large_far} implies that $(x, t) \in I_i$ for some $i \in \mathcal{I}$. Since the definition of $\widetilde{\widetilde{\Strip}}$ tells us that $(x, t) \in \Rn \setminus C'_{\frac14 \kappa R_{\S^*}}(x_{Q(\S^*)}, t_{Q(\S^*)})$, then it must be the case that $ I_i \cap \big( \Rn \setminus C'_{\frac14 \kappa R_{\S^*}}(x_{Q(\S^*)}, t_{Q(\S^*)})\big) \neq \emptyset$ because $(x, t)$ lies in that intersection. Moreover, $(x_0, x, t) \in \SStrip(I_i)$ because using \eqref{eq:D_large_far} we can estimate
\begin{equation*}
	x_0 
	\leq 
	\frac{c_2 \kappa}{16} R_{\S^*}
	\leq 
	\frac{c_2}{2} D(x, t)
	\leq 
	\psi_{\S^*}(x, t) + c_2 D(x, t)
\end{equation*}
where in the last inequality we are using \eqref{eq:psi_not_huge}, \eqref{eq:D_large_far} and choosing $\varepsilon$ small to obtain 
\begin{equation*}
	-\psi_{\S^*}(x, t)
	\leq 
	\varepsilon^{1/4} R_{\S^*}
	\leq 
	\varepsilon^{1/4} \frac{8}{\kappa} D(x, t)
	\leq 
	\frac{c_2}{2} D(x, t).
\end{equation*}
This establishes \eqref{eq:inclusion_strips}.

On the other hand, we also claim that 
\begin{equation} \label{eq:psi_max}
	\norm{\psi_{\S^*}}_{L^\infty(\R^n)}
	\leq 
	4C_{\mathrm{LA}} \kappa \varepsilon^{1/2} R_{\S^*}
	\leq 
	\frac{c_2 \kappa}{20} R_{\S^*}.
\end{equation}
Indeed, if we pick any $(x, t)$ for which $\psi_{\S^*}(x, t) \neq 0$, then $(x, t) \in C'_{4\kappa R_{\S^*}}(x_{Q(\S^*)}, t_{Q(\S^*)})$ because of \eqref{eq:psi_zero_far}. Therefore we can pick $(z, \tau) \in \partial C'_{4\kappa R_{\S^*}}(x_{Q(\S^*)}, t_{Q(\S^*)})$ (for which $\psi_{\S^*}(z, \tau) = 0$ by \eqref{eq:psi_zero_far}, too) with $\norm{(x, t) - (z, \tau)} \leq 4\kappa R_{\S^*}$. Then, the first inequality in \eqref{eq:psi_max} follows by \eqref{eq:psi_lipschitz}. The second inequality follows by choosing $\varepsilon$ small.

\textbf{Step 2: Finishing the p-CME.}
With \eqref{eq:psi_max} in mind, we split
\begin{equation} \label{eq:A1A2A3}
	\Omega_{\S^*} \cap \mbf{C}_r(\overline{\mbf{y}})
	=
	A_1 \cup A_2 \cup A_3.
\end{equation}
where we consider the following decomposition 
\begin{align*}
	A_1
	& := 
	\left\{ (x_0, x, t) \in \R \times \R^{n-1} \times \R : x_0 > \frac{c_2\kappa}{16}R_{\S^*} \right\} \cap \mbf{C}_r(\overline{\mbf{y}})
	\\ 
	A_2
	& :=
	\left\{ (x_0, x, t) \in \R^{n+1} : x_0 \leq \frac{c_2\kappa}{16}R_{\S^*}, \; (x, t) \in C'_{\frac12 \kappa R_\S^*}(x_{Q(\S^*)}, t_{Q(\S^*)}) \right\} 
	\cap \Omega_{\S^*} 
	\cap \mbf{C}_r(\overline{\mbf{y}})
	\\ 
	A_3
	& :=
	\left\{ (x_0, x, t) \in \R^{n+1} : x_0 \leq \frac{c_2\kappa}{16}R_{\S^*}, \; (x, t) \in \Rn \setminus C'_{\frac12 \kappa R_\S^*}(x_{Q(\S^*)}, t_{Q(\S^*)}) \right\} 
	\cap \Omega_{\S^*} 
	\cap \mbf{C}_r(\overline{\mbf{y}}).
\end{align*}

The part of the p-CME estimate \eqref{goal:cme} that happens inside the set $A_2$ satisfies the correct bound, because of  the arguments in the two previous subsections, because \eqref{eq:assumption_local} is satisfied. The part inside the set $A_3$ also satisfies the desired bounds because it is contained in $\widetilde{\widetilde{\Strip}}$, and since we have the inclusion \eqref{eq:inclusion_strips}, we can repeat the arguments in Steps 1.2 and 1.3 from the previous subsections to get p-CME.  

In the set $A_1$ the situation is different: although p-CME in this subdomain holds because it is a half-space (so the region above the graph of a parabolic uniformly rectifiable set, whence we have Lemma~\ref{lem:ur_implies_cme}), the distance to the boundary of this set is not the same as $\delta_{\Omega_{\S^*}}$, so we need to be a little more careful. Indeed, write 
\begin{equation*}
	A_1^\star := \left\{ (x_0, x, t) \in \R \times \R^{n-1} \times \R : \; x_0 > \frac{c_2\kappa}{18}R_{\S^*} \right\}, 
\end{equation*}
so that, recalling that $\norm{\psi_{\S^*}}_\infty \leq \frac{c_2\kappa}{20} R_{\S^*}$, we have 
\begin{equation*}
	A_1 \subset A_1^\star \subset \Omega_{\S^*},
\end{equation*}
and there is a uniform separation of the order of $R_{\S^*}$ in between them. We claim that there is a large enough $M_2 \geq 1$ such that 
\begin{equation} \label{eq:comparison_distances}
	\delta_{\Omega_{\S^*}}(x_0, x, t) \leq M_2 \, \delta_{A_1^\star}(x_0, x, t), 
	\qquad 
	(x_0, x, t) \in A_1.
\end{equation}
Taking this for granted momentarily, let us quickly finish the proof of the p-CME estimate over $A_1$. Indeed, if $A_1 \cap \C_r(\overline{\mbf{y}}) \neq \emptyset$ (otherwise the estimate is trivial), pick $\mathbf{Z} \in \partial A_1^\star \cap \C_r(\overline{\mbf{y}})$ and estimate, using Lemma~\ref{lem:ur_implies_cme}\footnote{
	Actually, in this specific case, one does not need the full power of Lemma~\ref{lem:ur_implies_cme}: $A_1^\star$ is a half-space, which is far simpler than a generic graphical parabolic UR domain, for which Lemma~\ref{lem:ur_implies_cme} was devised. In this concrete case, significant simplifications arise from the fact that the Green's function is just the vertical coordinate (in which the graph is parametrized), so its derivatives simply vanish, and this makes the argument much easier. 
}
\begin{multline*}
	\iiint_{A_1 \cap \C_r(\overline{\mbf{y}})} \abs{\nabla u}^2 \delta_{\Omega_{\S^*}}
	\lesssim 
	\iiint_{A_1 \cap \C_r(\overline{\mbf{y}})} \abs{\nabla u}^2 \delta_{A_1^\star}
	\leq 
	\iiint_{A_1^\star \cap \C_r(\overline{\mbf{y}})} \abs{\nabla u}^2 \delta_{A_1^\star}
	\\ \leq 
	\iiint_{A_1^\star \cap \mbf{C}_{2r}(\mathbf{Z})} \abs{\nabla u}^2 \delta_{A_1^\star}
	\lesssim 
	r^{n+1}, 
\end{multline*}
which finishes the proof of Theorem~\ref{th:push_cme} (modulo checking \eqref{eq:comparison_distances}).

Let us finish by proving \eqref{eq:comparison_distances}. First note that by \eqref{eq:psi_max}, it holds
\begin{equation*}
	\delta_{\Omega_{\S^*}}(x_0, x, t) 
	\leq 
	\abs{x_0 - \psi_{\S^*}(x, t)}
	\leq 
	x_0 + \abs{\psi_{\S^*}(x, t)}
	\leq 
	x_0 + \frac{c_2\kappa}{20} R_{\S^*}, 
\end{equation*}
and on the other hand we clearly have $\delta_{A_1^\star}(x_0, x, t) = x_0 - \frac{c_2\kappa}{18} R_{\S^*}$. So in order for \eqref{eq:comparison_distances} to hold, it suffices that 
\begin{equation*}
	x_0 + \frac{c_2\kappa}{20} R_{\S^*}
	\leq 
	M_2 \left(x_0 - \frac{c_2\kappa}{18} R_{\S^*}\right), 
	\quad \text{ or equivalently, } \quad 
	x_0 
	\geq 
	\frac{\frac{c_2\kappa}{20} + \frac{c_2\kappa}{18} M_2}{M_2-1} R_{\S^*}, 
\end{equation*}
Since $x_0 > \frac{c_2\kappa}{16} R_{\S^*}$ (by definition of $A_1$) and the right hand side converges to $\frac{c_2\kappa}{18} R_{\S^*}$ as $M_2 \to +\infty$, we can choose $M_2$ large enough (uniform) so that the inequality is satisfied. This finishes the proof of \eqref{eq:comparison_distances}, and hence the proof of Theorem~\ref{th:push_cme}.



\section{Proofs of the main results} \label{sec:proof}

After all the work in the previous sections, the proofs of our main results, namely Theorems~\ref{th:main}, \ref{th:main2} and \ref{th:all}, consist in simply wrapping up the results.

\begin{proof}[\textbf{Proof of Theorem~\ref{th:main2}}]
	Since the caloric measure for $\Omega$ admits a corona decomposition, we apply Proposition~\ref{prop:WHSA} to find a refined decomposition of the dyadic grid $\D$ in subtrees $\S^*$ such that, associated to each $\S^*$, there is a $\Lip(1,1/2)$ graph approximating $\Sigma$ at the scale of $\S^*$, along with many extra properties. At this point, we do not know whether these approximating graphs are regular, so we cannot yet assert the packing condition in the last item of Proposition~\ref{prop:WHSA}.
	
	Now, we run Theorem~\ref{th:push_cme} (it also uses Lemma~\ref{lem:u_upper_bounds}, which is available under our assumptions) to find that the approximating graphs constructed right before are in fact regular (with constants uniform with respect to those in our assumptions). This allows us to assert that the last item in Proposition~\ref{prop:WHSA}, namely the packing condition for the graph trees, holds. This packing condition along with the rest of properties in Proposition~\ref{prop:WHSA}, concretely \eqref{eq:Gamma_approximates}, means that we have effectively found a coronization of $\Sigma$ by approximating regular $\Lip(1,1/2)$ graphs in the sense of Definition~\ref{def:unilateralcorona}. Thus, Theorem~\ref{th:corona_implies_UR} implies that $\Sigma$ is parabolic uniformly rectifiable.
\end{proof}

\begin{proof}[\textbf{Proof of Theorem~\ref{th:main}}]
	Using Theorem~\ref{th:corona}, we obtain a corona decomposition for $\omega$ with respect to $\sigma$ (remember that the packing condition for bad and maximal cubes strongly requires the use of the p-CME assumption, along with Proposition~\ref{prop:GMT}). Then, we can simply invoke Theorem~\ref{th:main2}.
\end{proof}

\begin{proof}[\textbf{Proof of Theorem~\ref{th:all}}]
	As already said after the statement of Theorem~\ref{th:all}, just put together Theorems~\ref{th:corona} and \ref{th:main2}, and \cite[Theorem 1.3]{BHHLN}.
\end{proof}

\begin{remark}[Dependencies of constants] \label{rem:dependencies}
	Along the previous sections we have used many constants, and one can check the dependencies do not lead to any circular argument. In case it helps the reader follow the arguments, let us sum up the (most relevant parts of the) graph of dependencies in the order that the constants are fixed. In all the proofs, the parameters $n, \vartheta, C_{\mathrm{ADR}}, C_{\mathrm{CDC}}, C_{\mathrm{CKS}}$ are fixed. The additional constants in Section~\ref{sec:preliminaries}, like $C_{\mathrm{cub}}$ and $c_{\mathrm{B}}$, depend only on those. In Section~\ref{sec:GMT}, $\tau = \tau(n, C_{\mathrm{ADR}}, C_{\mathrm{CDC}})$, and later $\varepsilon' = \varepsilon'(\tau)$. In Section~\ref{sec:corona}, $N = N(\tau, \varepsilon')$ and $\lambda = \lambda(\varepsilon')$, later $J = J(\lambda)$, and finally $C_{\mathrm{cor}} = C_{\mathrm{cor}}(C_{\mathrm{CME}}, \varepsilon', \tau, N, J)$. In Section~\ref{sec:WHSA}, $K_0, L, C_{\mathrm{LA}}$ depend on $n, C_{\mathrm{ADR}},C_{\mathrm{CDC}}, C_{\mathrm{cor}}$; and the parameter $\varepsilon$ is chosen depending on all these and, additionally, $K_0$ and $c_2$ from Section~\ref{sec:pushing_CME}: the dependence on $c_2$ is actually not a problem because $c_2 = c_2(n)$. The constant $C_{\mathrm{int}}$ from Lemma~\ref{lem:u_upper_bounds} depends on $n, C_{\mathrm{ADR}}, C_{\mathrm{CDC}}, C_{\mathrm{cor}}$ and additionally, $\varepsilon$. Later, in Section~\ref{sec:pushing_CME}, $k_0 = k_0(n)$, $C_0 = C_0(n, C_{\mathrm{ADR}})$, $c_1 = c_1(n)$, $c_2 = c_2(c_1, k_0)$, $\mu = \mu(n, c_2)$, $M_1 = M_1(\mu)$, and the final constant $C_{\mathrm{CME}}$ depends on all these (and we made dependencies on $\varepsilon$ explicit because it is somehow the last constant to be fixed). The reader can now see easily that the parameters have been fixed in the correct order.
\end{remark}

To conclude, let us point out that we expect the results to also hold for parabolic equations with variable coefficients, at least when the oscillations of these coefficients are controlled in a suitable Carleson measure sense, as is the case in \cite{BFHH}. Indeed, the integration by parts schemes in \cite[Section 6]{BFHH} should allow to adapt the arguments of Section~\ref{sec:pushing_CME} to the setting of variable coefficients, and the rest of the arguments should work very similarly as in the present paper, up to technical details. It is also possible that, if suitably interpreted, some of the implications in Theorem~\ref{th:all} hold under even weaker assumptions over the coefficients, as is the case in the elliptic setting, as shown in \cite{CHM}.


\appendix

\section{Some auxiliary computations}

Let us carry out in detail some computations that were used in the main body of the paper, which we have preferred to isolate here because they are classical or fairly routine. Nevertheless, we believe that including the details will be convenient for the reader.

\textbf{The weak-(1, 1) bound in Theorem~\ref{th:corona}.}
Let us provide a quick (classical) proof of the weak-(1, 1) estimate that is claimed in \eqref{eq:packing_HD}.

Given $\lambda > 0$, denote 
\begin{equation*}
	E_{\text{bad}} := \big\{ \mbf{x} \in Q_0 : \mathcal{M}_\sigma \omega^{\X_1}(\mbf{x}) > \lambda \big\}
\end{equation*}
Then, for any $\mbf{x} \in E_{\text{bad}}$, there exists some surface cylinder $\Delta_{\mbf{x}} := \C_r(\mbf{x}) \cap \Sigma$ such that $\omega^{\X_1}(\Delta_{\mbf{x}}) > \lambda \sigma(\Delta_{\mbf{x}})$. Therefore, we can compute, using the doubling property of $\sigma$ (because $\Sigma$ is p-ADR), 
\begin{multline*}
	\sigma \Big( \big\{ \mbf{x} \in Q_0 : \mathcal{M}_\sigma \omega^{\X_1}(\mbf{x}) > \lambda \big\} \Big)
	\leq 
	\sigma \Big( \bigcup_{\mbf{x} \in E_{\text{bad}} } \Delta_{\mbf{x}} \Big) 
	\leq 
	\sigma \Big( \bigcup_{\mbf{x} \in \mathcal{I}} 5\Delta_{\mbf{x}} \Big) 
	\\ \lesssim
	\sum_{\mbf{x} \in \mathcal{I}} \sigma(\Delta_{\mbf{x}})
	\leq 
	\frac{1}{\lambda} \sum_{\mbf{x} \in \mathcal{I}} \omega^{\X_1}(\Delta_{\mbf{x}})
	\leq 
	\frac{1}{\lambda} \omega^{\X_1}(Q_0).
\end{multline*}
where $\mathcal{I}$ is a family of indices provided by Vitali's covering theorem, so that the associated $\{\Delta_{\mbf{x}}\}_{\mbf{x} \in \mathcal{I}}$ are disjoint, yet $\bigcup_{\mbf{x} \in \mathcal{I}} 5\Delta_{\mbf{x}} $ covers the whole set $E_{\text{bad}}$.
Using this computation and the layer cake formula with $\lambda_0 = \frac{\omega^{\X_1}(Q_0)}{\sigma(Q_0)}$:
\begin{align*}
	\iint_{Q_0} (\mathcal{M}_\sigma \omega^{\X_1})^{1/2} d\sigma 
	& \approx 
	\int_0^{\lambda_0} \lambda^{-1/2} \sigma \left( \{ \mbf{x} : \mathcal{M}_\sigma \omega^{\X_1}(\mbf{x}) > \lambda \} \right) d\lambda 
	\\ & \quad + \int_{\lambda_0}^\infty \lambda^{-1/2} \sigma \left( \{ \mbf{x} : \mathcal{M}_\sigma \omega^{\X_1}(\mbf{x}) > \lambda \} \right) d\lambda 
	\\ & \lesssim 
	\int_0^{\lambda_0} \lambda^{-1/2} \sigma(Q_0) d\lambda 
	+ \int_{\lambda_0}^\infty \lambda^{-1/2} \lambda^{-1} \omega^{\X_1}(Q_0) d\lambda 
	\\ & \approx 
	\lambda_0^{1/2} \sigma(Q_0) 
	+ \lambda_0^{-1/2} \omega^{\X_1}(Q_0)
	\\ & \approx 
	\sqrt{\sigma(Q_0) \omega^{\X_1}(Q_0)},
\end{align*}
which is precisely what we used in \eqref{eq:packing_HD}.

\textbf{Poles are far from the region of integration in Theorem~\ref{th:push_cme}.} In the proof of Theorem~\ref{th:push_cme}, concretely in Substep 2.3, it is asserted that $H^*\widehat{u} = 0$ in $\Omega_{\S^*}'$. This follows from the fact that $\widehat{u}$ is an adjoint Green's function with pole at $(Y_\S, s_\S)$ (see \eqref{eq:def_hatu}), and one can check that $(Y_\S, s_\S) \notin \Omega_{\S^*}'$. Indeed, let us elaborate on why the latter is true.

On the one hand, $\Omega'_{\S^*} \subset \bigcup_{Q \in \S^*} \widetilde{\mbf{U}}_Q^\good$ by \eqref{eq:containment}, and each of the $\widetilde{\mbf{U}}_Q^\good$ are contained (see Subsection~\ref{subsec:whitney}) in $\mbf{C}_{\varepsilon^{-4}\ell(Q)}(Y_Q^\good, s_Q^\good)$ for some $(Y_Q^\good, s_Q^\good) \in \mbf{U}_Q^\good$. Moreover, $(Y_Q^\good, s_Q^\good) \in \mbf{C}_{C(n)K_0\ell(Q)}(\mbf{x}_Q)$ by \eqref{eq:whitney_bounded}, where $\mbf{x}_Q$ is the center of $Q$ (see Subsection~\ref{subsec:geometry}). Therefore, noting that $\ell(Q) \leq \ell(Q(\S^*))$ because $Q \in \S^*$, we infer that 
\begin{equation*}
	\Omega'_{\S^*} 
	\subset 
	\mbf{C}_{2C(n)K_0\varepsilon^{-4}\ell(Q(\S^*))}(\mbf{x}_{Q(\S^*)}) 
	\subset 
	\mbf{C}_{\varepsilon^{-5}\ell(Q(\S^*))}(\mbf{x}_{Q(\S^*)})
\end{equation*}
using also \eqref{eq:epsilon_K_0}.

On the other hand, by Theorem~\ref{th:corona} it holds 
\begin{equation*}
	\delta_\Omega(Y_\S, s_\S) 
	\geq 
	100\diam(Q(\S)) 
	\geq 
	C_{\mathrm{cub}}^{-1} 100\varepsilon^{-10}\ell(Q(\S^*)),
\end{equation*}
where we have also used the fact that $\ell(Q(\S^*)) \leq \varepsilon^{10}\ell(Q(\S))$ because $Q(\S^*) \in \S^*$ (indeed, as can be seen before (8.6) and in Propositions 7.69 and 8.12 in \cite{BHMN}, the graph trees $\S^*$ are all contained in $\D_{\F, Q(\S)}^* = \{ Q \in \D_{\F, Q(\S)} : \ell(Q) \leq \varepsilon^{10} \ell(Q(\S))\}$ from \cite[(6.35)]{BHMN}).

In view of the above, since we can choose $\varepsilon > 0$ small enough, it is clear that $(Y_\S, s_\S)$ lies very far outside $\Omega'_{\S^*}$, as desired.




\begin{thebibliography}{AAAAAAA} 
\parskip=0.1cm

\bibitem[Ai02]{Aikawa} Hiroaki Aikawa, \emph{H\"{o}lder continuity of the Dirichlet solution for a general domain}, Bull. London Math. Soc \textbf{34} (2002), 691--702.

\bibitem[Ar68]{A} Donald G. Aronson, \emph{Non-negative solutions of linear parabolic equations}, Ann. Sc. Norm. Super. Pisa Cl. Sci. \textbf{22} (1968), 607--694.

\bibitem[AGMT23]{AGMT} Jonas Azzam, John Garnett, Mihalis Mourgoglou and Xavier Tolsa, \emph{Uniform rectifiability, elliptic measure, square functions, and $\varepsilon$-approximability via an ACF monotonicity formula}, Int. Math. Res.
Not. \textbf{13} (2023), 10837--10941.

\bibitem[AHMMT20]{AHMMT} Jonas Azzam, Steve Hofmann, José María Martell, Mihalis Mourgoglou and Xavier Tolsa,
\emph{Harmonic measure and quantitative connectivity: geometric characterization of the $L^p$-solvability of the Dirichlet problem}, Invent. Math. \textbf{222} (2020), 881--993.

\bibitem[BFHH25]{BFHH} Simon Bortz, Sandra Ferris, Pablo Hidalgo-Palencia and Steve Hofmann, \emph{A variable coefficient free boundary problem for $L^p$ solvability of parabolic Dirichlet problem in graph domains}, arXiv preprint 2503.00873 (2025), 54 pp. 

\bibitem[BHHLN22a]{BHHLN_22a} Simon Bortz, John Hoffman, Steve Hofmann, José Luis Luna-García and Kaj Nystr\"{o}m, \emph{Coronizations and big pieces in metric spaces}, Ann. Inst. Fourier (Grenoble) \textbf{72} (2022), 2037--2078.

\bibitem[BHHLN22b]{BHHLN_22b} Simon Bortz, John Hoffman, Steve Hofmann, José Luis Luna-García and Kaj Nystr\"{o}m, \emph{On big pieces approximations of parabolic hypersurfaces}, Ann. Fenn. Math. \textbf{47} (2022), 533--571.

\bibitem[BHHLN23a]{BHHLN} Simon Bortz, John Hoffman, Steve Hofmann, José Luis Luna-García and Kaj Nystr\"{o}m, \emph{Carleson measure estimates for caloric functions and parabolic uniformly rectifiable sets}, Anal. PDE \textbf{16} (2023), 1061--1088.

\bibitem[BHHLN23b]{BHHLN_corona} Simon Bortz, John Hoffman, Steve Hofmann, José Luis Luna-García and Kaj Nystr\"{o}m, \emph{Corona decompositions for parabolic uniformly rectifiable sets}, J. Geom. Anal. \textbf{33} (2023), art. no. 96, 67 pp.

\bibitem[BHMN25a]{BHMN_graph} Simon Bortz, Steve Hofmann, José María Martell and Kaj Nystr\"{o}m, \emph{Solvability of the $L^p$ Dirichlet problem for the heat equation is equivalent to parabolic uniform rectifiability in the case of a parabolic Lipschitz graph}, Invent. Math. \textbf{239} (2025), 165--217.

\bibitem[BHMN25b]{BHMN} Simon Bortz, Steve Hofmann, José María Martell and Kaj Nystr\"{o}m, \emph{Solvability of the Dirichlet problem for the heat equation implies parabolic uniform rectifiability}, arXiv preprint 2510.22047 (2025), 101 pp.

\bibitem[CHM24]{CHM} Mingming Cao, Pablo Hidalgo-Palencia and José María Martell, \emph{Carleson measure estimates, corona decompositions, and perturbation of elliptic operators without connectivity}, Math. Ann. \textbf{390} (2024), 95--156.

\bibitem[CHMPZ25]{CHMPZ} Mingming Cao, Pablo Hidalgo-Palencia, José María Martell, Cruz Prisuelos-Arribas and Zihui Zhao, \emph{Elliptic operators in rough sets, and the Dirichlet problem with boundary data in H\"{o}lder spaces}, J. Funct. Anal. \textbf{288} (2025), art. no. 110801, 68 pp.

\bibitem[Ch90]{Christ} Michael Christ, \emph{A T(b) theorem with remarks on analytic capacity and the Cauchy integral}, Colloq. Math. \textbf{60/61} (1990), 601--628.

\bibitem[Da77]{D} Bj\"{o}rn E.J. Dahlberg, \emph{Estimates of harmonic measure}, Arch. Rational Mech. Anal. \textbf{65} (1977), 275--288.

\bibitem[DS91]{DS1} Guy David and Stephen Semmes, \emph{Singular integrals and rectifiable sets in $R^n$: Beyond Lipschitz graphs}, Astérisque \textbf{193} (1991), 145 pp.

\bibitem[DS93]{DS2} Guy David and Stephen Semmes, Analysis of and on uniformly rectifiable sets, vol. 38 of Math. Surveys Monogr. Amer. Math. Soc., Providence, RI, 1993.

\bibitem[DKP11]{DKP} Martin Dindo\v{s}, Carlos Kenig and Jill Pipher, \emph{BMO solvability and the $A_\infty$ condition for elliptic operators}, J. Geom. Anal. \textbf{21} (2011), 78--95.

\bibitem[DPP17]{DPP} Martin Dindo\v{s}, Jill Pipher and Stefanie Petermichl, \emph{BMO solvability and the $A_\infty$ condition for second order parabolic operators}, Ann. Inst. H. Poincaré Anal. Non Linéaire \textbf{34} (2017), 1155–1180.

\bibitem[EG82]{EG_Wiener} Lawrence C. Evans and Ronald F. Gariepy, \emph{Wiener’s criterion for the heat equation}, Arch. Ration. Mech. Anal. \textbf{78} (1982), 293--314.

\bibitem[EG15]{EG} Lawrence C. Evans and Ronald F. Gariepy, Measure theory and fine properties of functions. Textbooks in Mathematics, CRC Press, Boca Raton, FL, revised edition, 2015.

\bibitem[FGL89]{FGL} Eugene B. Fabes, Nicola Garofalo and Ermanno Lanconelli, \emph{Wiener’s criterion for divergence form parabolic operators with $C^1$-Dini continuous coefficients}, Duke Math. J. \textbf{59} (1989), 191--232.

\bibitem[FGS86]{FGS} Eugene B. Fabes, Nicola Garofalo and Sandro Salsa, \emph{A backward Harnack inequality and Fatou theorem for nonnegative solutions of parabolic equations}, Illinois J. Math. \textbf{30} (1986), 536–565.

\bibitem[FS72]{FS} Charles Fefferman and Elias M. Stein, \emph{$H^p$ spaces of several variables}, Acta Math. \textbf{129} (1972), 137--193.

\bibitem[GL88]{GL} Nicola Garofalo and Ermanno Lanconelli, \emph{Wiener’s criterion for parabolic equations with variable coefficients and its consequences}, Trans. Amer. Math. Soc. \textbf{308} (1988), 811--836.

\bibitem[GH20]{GH_RH} Alyssa Genschaw and Steve Hofmann, \emph{A weak reverse H\"{o}lder inequality for caloric measure}, J. Geom. Anal. \textbf{30} (2020), 1530--1564.

\bibitem[GH21]{GH_BMO} Alyssa Genschaw and Steve Hofmann, \emph{BMO solvability and absolute continuity of caloric measure}, Potential Anal. \textbf{54} (2021), 451--469.

\bibitem[GMT18]{GMT} John Garnett, Mihalis Mourgoglou and Xavier Tolsa, \emph{Uniform rectifiability from Carleson measure estimates and $\varepsilon$-approximability of bounded harmonic functions}, Duke Math. J. \textbf{167} (2018), 1473--1524.

\bibitem[HJ25]{HJ} John Hoffman and Benjamin Jaye, \emph{On singular integrals and quantitative rectifiability in parabolic space and the Heisenberg group}, arXiv preprint 2510.26934 (2025), 35 pp.

\bibitem[HL18]{HLe} Steve Hofmann and Phi Le, \emph{BMO solvability and absolute continuity of harmonic measure}, J. Geom. Anal. \textbf{28} (2018), 3278--3299.

\bibitem[HLMN17]{HLMN} Steve Hofmann, Phi Le, José María Martell and Kaj Nystr\"{o}m, \emph{The weak-$A_\infty$ property of harmonic and $p$-harmonic measures implies uniform rectifiability}, Anal. PDE \textbf{10} (2017), 513–558.

\bibitem[HL96]{HL} Steve Hofmann and John L. Lewis, \emph{$L^2$ solvability and representation by caloric layer potentials in time-varying domains}, Ann. of Math. \textbf{144} (1996), 349--420.

\bibitem[HL01]{HL01} Steve Hofmann and John L. Lewis, \emph{The Dirichlet problem for parabolic operators with singular drift terms}, Mem. Amer. Math. Soc. \textbf{151} (2001), 113 pp.

\bibitem[HL05]{HL05} Steve Hofmann and John L. Lewis, \emph{The $L^p$ Neumann problem for the heat equation in noncylindrical domains}, J. Funct. Anal. \textbf{220} (2005), 1--54.

\bibitem[HMM16]{HMM} Steve Hofmann, José María Martell and Svitlana Mayboroda, \emph{Uniform rectifiability, Carleson measure estimates, and approximation of harmonic functions}, Duke Math. J. \textbf{165} (2016), 2331–-2389.

\bibitem[HW25]{HW} Steve Hofmann and James Warta, \emph{A characterization of solvability of the parabolic $L^p$ Dirichlet problem on Lipschitz graph domains via Carleson measure estimates of bounded solutions}, arXiv preprint 2509.04701 (2025), 19 pp.

\bibitem[HHK25]{HHK} Pablo Hidalgo-Palencia, Cody Hutcheson and Joseph Kasel, \emph{The parabolic Dirichlet problem with continuous and H\"{o}lder boundary data, and rough coefficients}, arXiv preprint 2510.04833 (2025), 41 pp. 

\bibitem[HK12]{HytKai} Tuomas Hyt\"{o}nen and Anna Kairema, \emph{Systems of dyadic cubes in a doubling metric space}, Colloq. Math. \textbf{126} (2012), 1--33.

\bibitem[KW80]{KW} Robert Kaufman and Jang-Mei Wu, \emph{Singularity of parabolic measures}, Compositio Math. \textbf{40} (1980), 243--250.

\bibitem[KKPT14]{KKiPT} Carlos Kenig, Bernd Kirchheim, Jill Pipher and Tatiana Toro, \emph{Square functions and the $A_\infty$ property of elliptic measures}, J. Geom. Anal. \textbf{26} (2014), 2383--2410.

\bibitem[KKPT00]{KKoPT} Carlos Kenig, Herbert Koch, Jill Pipher and Tatiana Toro, \emph{A new approach to absolute continuity of elliptic measure, with applications to non-symmetric equations}, Adv. Math. \textbf{153} (2000), 231--298.

\bibitem[La73]{L} Ermanno Lanconelli, \emph{Sul problema di Dirichlet per l’equazione del calore}, Ann. Math. Pura Appl. \textbf{97} (1973), 83--114.

\bibitem[LM95]{LM} John L. Lewis and Margaret A. M. Murray, \emph{The method of layer potentials for the heat equation in time-varying domains}, Mem. Amer. Math. Soc., \textbf{114} (1995), viii+157.

\bibitem[LV06]{LV1} John L. Lewis and Andrew Vogel, \emph{Uniqueness in a free boundary problem}, Communications in PDE \textbf{31} (2006), 1591--1614.

\bibitem[LV07]{LV2} John L. Lewis and Andrew Vogel, \emph{Symmetry theorems and uniform rectifiability,} Boundary Value Problems (2007), art. no. 030190, 59 pp.

\bibitem[MP21]{MP} Mihalis Mourgoglou and Carmelo Puliatti, \emph{Blow-ups of caloric measure in time varying domains and applications to two-phase problems}, J. Math. Pures Appl. \textbf{152} (2021), 1--68.

\bibitem[Mu85]{M} Margaret A. M. Murray, \emph{Commutators with fractional differentiation and BMO Sobolev spaces}, Indiana Univ. Math. J., \textbf{34} (1985), 205--215.

\bibitem[NTV14]{NTV} Fedor Nazarov, Xavier Tolsa and Alexander Volberg, \emph{On the uniform rectifiability of AD-regular measures with bounded Riesz transform operator: the case of codimension 1}, Acta Math. \textbf{213} (2014), 237--321.

\bibitem[NS17]{NS} Kaj Nystr\"{o}m and Martin Str\"{o}mqvist, \emph{On the parabolic Lipschitz approximation of parabolic uniform rectifiable sets}, Rev. Mat. Iberoam. \textbf{33} (2017), 1397--1422.

\bibitem[Wa12]{Watson} Neil A. Watson, Introduction to heat potential theory, Mathematical Surveys and Monographs
182, Amer. Math. Soc., Providence, RI, 2012.

\bibitem[Wi24]{Wiener} Norbert Wiener, \emph{The Dirichlet problem}, J. Math. and Phys. \textbf{3} (1924), 127--146.

\end{thebibliography}
\end{document}